\documentclass[12pt,a4paper]{amsart}

\usepackage{amssymb,enumerate}

\usepackage{graphicx,amssymb,amsfonts,epsfig,amsthm,a4,amsmath,url,verbatim}
\usepackage[latin1]{inputenc}

\newtheorem{thm}{Theorem}[subsection]
\newtheorem{cor}[thm]{Corollary}
\newtheorem{lem}[thm]{Lemma}

\newtheorem{prop}[thm]{Proposition}

\theoremstyle{definition}
\newtheorem{defn}[thm]{Definition}

\newtheorem{que}[thm]{Question}
\newtheorem{conj}[thm]{Conjecture}
\newtheorem{exe}[thm]{Example}

\newtheorem{rem}[thm]{Remark}

\numberwithin{equation}{section}

\newcommand{\ppt}{\bar{\partial}}
\newcommand{\pt}{\partial}
\newcommand{\TConv}{\textnormal{TConv}}
\newcommand{\Sel}{\textnormal{Sel}}
\newcommand{\Comm}{\textnormal{Comm}}
\newcommand{\mapsot}{\reflectbox{\ensuremath{\mapsto}}}
\newcommand{\N}{\mathbf{N}}
\newcommand{\Z}{\mathbf{Z}}
\newcommand{\R}{\mathbf{R}}
\newcommand{\C}{\mathbf{C}}
\newcommand{\Q}{\mathbf{Q}}

\newcommand{\vb}{\textnormal{vb}}
\newcommand{\bbb}{\textnormal{b}}
\newcommand{\Supp}{\textnormal{Supp}}

\newcommand{\ind}{\textnormal{ind}}
\newcommand{\SL}{\textnormal{SL}}

\newcommand{\GL}{\textnormal{GL}}

\newcommand{\PSL}{\textnormal{PSL}}

\newcommand{\Ker}{\textnormal{Ker}}

\newcommand{\Aut}{\textnormal{Aut}}
\newcommand{\End}{\textnormal{End}}

\newcommand{\Isom}{\text{Isom}}

\newcommand{\CB}{\mathcal{B}}

\newcommand{\SX}{\mathcal{S}}

\newcommand{\eps}{\varepsilon}

\newcommand{\tu}{\bigtriangleup}
\newcommand{\VD}{\mathcal{VD}}

\newcommand{\tr}{\mathrm{tr}}

\begin{document}
\title[Commensurated subsets, wallings and cubings]{Group actions with commensurated subsets, wallings and cubings}
\author{Yves Cornulier}%
\address{Laboratoire de Math\'ematiques\\
B\^atiment 425, Universit\'e Paris-Sud 11\\
91405 Orsay\\FRANCE}
\email{yves.cornulier@math.u-psud.fr}




\date{August 25, 2015}

\begin{abstract}
We study commensurating actions of groups and the associated properties FW and PW, in connection with wallings, median graphs, CAT(0) cubings and multi-ended Schreier graphs.
\end{abstract}

\maketitle

\section{Introduction}

\subsection{Commensurated subsets and associated properties}

\subsubsection{Context}
The source of CAT(0) cubings can be found in several originally partly unrelated areas, including median graphs, group actions on trees, ends of Schreier graphs, Coxeter groups, cubulations of 3-dimensional manifolds. The link with Kazhdan's Property T was gradually acknowledged, first in the case of trees, then for finite-dimensional CAT(0) cubings and then for general CAT(0) cubings; the same arguments were also found at the same time in different languages, notably in terms of wall spaces. The present paper is an attempt to give a synthesis of those different point of views, and especially to advertise the most elementary approach, namely that of group actions with commensurated subsets.

Let us emphasize that actions on CAT(0) cubings have now reached a considerable importance in geometric group theory. However, maybe partly because of the scattering of points of view and the elaborateness of the notion of CAT(0) cubing, it is sometimes considered as a intermediate tool, for instance to prove that some group does not have Property T or has the Haagerup Property. This is certainly unfair, and CAT(0) cubings and consorts are worth much better than being subcontractors of those analytic properties, and therefore we introduce the following terminology.

\subsubsection{Actions with commensurated subsets}

Consider an action of a group $G$ on a discrete set $X$ (we assume throughout this introduction that groups are discrete, but will address the setting of topological groups as well). We say that a subset $M\subset X$ is {\em commensurated} by the $G$-action if $$\ell_M(g)=\#(M\tu gM)<\infty,\qquad \forall g\in G,$$ where $\tu$ denotes the symmetric difference\footnote{The notion of commensurated subset is the set analogue of the notion of {\em commensurated subgroup} which is not considered in this paper.}. We say that $M$ is {\em transfixed} if there exists a $G$-invariant subset $N$ with $M\tu N$ finite. Transfixed subsets are the trivial instances of commensurated subsets.

\subsubsection{Property FW}

\begin{defn}We say that $G$ has Property FW if for every $G$-set, every commensurated subset is transfixed.
\end{defn}

For every $G$-set $X$ with a commensurated subset $M$, define the function $\ell_M(g)=\#(gM\tu M)$; every such function on $G$ is called a {\em cardinal definite function}\footnote{When $X$ is replaced by an arbitrary measure space and cardinality is replaced by the measure, we obtain the notion of {\em measure definite function} addressed by Robertson and Steger in~\cite{RS}.} on $G$.

Property FW for the group $G$ turns out to have several equivalent characterizations, both geometric and combinatorial, see Section \ref{comp} for the relevant definitions and Proposition \ref{fwint} for the proofs.
\begin{enumerate}[(i)]
\item\label{fwdef} $G$ has Property FW;
\item\label{fwcb} every cardinal definite function on $G$ is bounded;
\item\label{fwca} every cellular action on any CAT(0) cube complex has bounded orbits for the $\ell^1$-metric (we allow infinite-dimensional cube complexes);
\item\label{fwcf} every cellular action on any CAT(0) cube complex has a fixed point;
\item\label{fwmeb} every action on a connected median graph has bounded orbits;
\item\label{fwme} every action on a nonempty connected median graph has a finite orbit;
\item\label{fwcod} (if $G$ is finitely generated and endowed with a finite generating subset) every Schreier graph of $G$ has at most 1~end;
\item\label{fww} for every set $Y$ endowed with a walling and compatible action on $Y$ and on the index set of the walling, the action on $Y$ has bounded orbits for the wall distance;
\item\label{zhilb} every isometric action on an ``integral Hilbert space" $\ell^2(X,\Z)$ ($X$ any discrete set), or equivalently on $\ell^2(X,\R)$ preserving integral points, has bounded orbits;
\item\label{zx} for every $G$-set $X$ we have $H^1(G,\Z X)=0$.
\end{enumerate}

The implication (\ref{fwca})$\Rightarrow$(\ref{fwcf}) looks at first sight like a plain application of the center lemma; however this is not the case in general since the CAT(0) cube complex is not complete and we would only deduce, for instance, a fixed point in the $\ell^2$-completion. The argument (which goes through median graphs) is due to Gerasimov and is described in \S\ref{mger}.

Note that FW means ``fixed point property on walls", in view of (\ref{fww}). \footnote{The terminology FW is borrowed from Barnhill and Chatterji \cite{BC}.}


It follows from (\ref{zhilb}) that Property FH implies Property FW. Recall that Property FH means that every isometric action on a real Hilbert space has bounded orbits, and is equivalent (for countable groups) to the representation-theoretic Kazhdan's Property T by Delorme-Guichardet's Theorem (see \cite{BHV} for a general introduction to these properties). Also, Property FH was characterized by Robertson-Steger \cite{RS} in a way very similar to the above definition of Property FW, namely replacing the action on a discrete set $X$ and cardinality of subsets by a measure-preserving action on a measured space and measure of its measurable subsets.

 Also, since trees are the simplest examples of CAT(0) cube complexes, it follows from (\ref{fwca}) that Property FW implies Serre's Property FA: every action on a tree has bounded orbits. This can also be directly viewed with commensurating actions: if a group $G$ acts on a tree, then it acts on the set of oriented edges, and for each fixed vertex $x_0$, the set of oriented edges pointing towards $x_0$ is a commensurated subset, which is transfixed only if the original action has bounded orbits. 
 
 Thus FW is a far-reaching strengthening of Property FA, while weaker than Property FH. Note that Property FH is of much more analytical nature, but is has no combinatorial characterization at this time. A considerable work has been done to settle partial converses to the implication FW$\Rightarrow$FA, the first of which being Stallings' characterization of finitely generated groups with several ends. However, there are a lot of groups satisfying FA but not FW, see Example \ref{FApFW}.

\subsubsection{Property PW}

In view of (\ref{fwcb}), it is natural to introduce the opposite property PW, which was explicitly introduced in \cite{CSVa}:

\begin{defn}The group $G$ has Property PW if it admits a proper commensurating action, in the sense that the cardinal definite function $\ell_M$ is proper on $G$.
\end{defn}
(Recall that $f:G\to\R$ proper means $\{x: |f(x)|\le r\}$ is finite for all $r<\infty$.) Obviously, a group has both Properties PW and FW if and only if it is finite. We say that an isometric action of a discrete group on a metric space $X$ is {\em proper} if for some $x\in X$, the function $g\mapsto f(x,gx)$ is proper; then this holds for all $x\in X$. Property PW has, in a similar fashion, equivalent restatements:

\begin{enumerate}[(i')]
\addtocounter{enumi}{2}
\item\label{pwca} there exists a proper cellular action on a (possibly infinite-dimensional) complete CAT(0) cube complex with the $\ell^1$-metric;
\addtocounter{enumi}{1}
\item\label{pwme} there is a proper isometric action on a connected median graph;
\addtocounter{enumi}{2}
\item\label{pww} there exists a set $Y$ endowed with a walling and compatible actions on this set and on the index set of the walling, such the action on $Y$  endowed with the wall distance is metrically proper;
\item\label{pzhilb} there exists a proper isometric action on an ``integral Hilbert space" $\ell^2(X,\Z)$ (for some discrete set $X$), actually extending to $\ell^2(X,\R)$. 
\end{enumerate}

Here we enumerate in accordance with the corresponding characterizations of Property FW; for instance we omit the analogue of (\ref{fwcb}) because it would be tautological, while (\ref{fwcod}) has no trivial restatement. Actually I do not know any purely combinatorial definition of Property PW (not explicitly involving any kind of properness); nevertheless it is a very important feature in geometric group theory, and its refinements (such as proper cocompact actions on finite-dimensional proper cube complexes) play an essential role in the understanding of 3-manifold groups, among others. 

Similarly as above, we see that Property PW implies the Haagerup Property, which for a countable group asserts the existence of a proper isometric action on a Hilbert space. One of the main strengths of the Haagerup Property for a group $G$ is that it implies that $G$ satisfies the Baum-Connes conjecture in its strongest form, namely with coefficients \cite{HK}.


\subsubsection{How to read this paper?} After a possible touchdown at the examples in Section \ref{claex}, the reader will find basic notions in Section \ref{comp}. The study of commensurating actions is developed in Section \ref{comto} (where all groups can be assumed to be discrete in a first reading), and is specified to the study of Property FW and its cousins in Section \ref{fwetc}. Section \ref{s_ab} applies basic results on commensurating actions from Section \ref{comto} to the study of cardinal definite functions on abelian groups, with applications to properties FW and PW. Finally, Section \ref{medgr}, which only uses as a prerequisite the definitions of Section \ref{comp}, especially surveys previous work about median graphs; this survey is at this point is far from comprehensive, but includes important notions such as the Sageev graph associated to a commensurating action, and Gerasimov's theorem that any bounded action on a connected median graph has a finite orbit.

\smallskip

\noindent {\em Warning.} A significant part of this paper consists of non-original results (see \S\ref{abt}); however the way they are stated here can differ from the classical (and divergent) points of view. This paper is partly an attempt to a synthesis of these points of view (see Section \ref{comp}).


\smallskip

\noindent {\bf Acknowledgements.} I thank Serge Bouc, Pierre-Emmanuel Caprace, Vincent Guirardel, Fr\'ed\'eric Haglund, Michah Sageev, Todor Tsankov and Alain Valette for useful conversations. I am grateful to Peter Haissinsky for a number of corrections.

\tableofcontents

\section{Classical examples}\label{claex}

\subsection{Examples with Property PW}

Let us give examples of groups satisfying Property PW. It can be hard to give accurate references, insofar as the link different characterizations of Property PW were not originally well-understood, and also because Property PW was often obtained accidentally. Therefore, in the following enumeration, I use a chronological order taking into account the availability of the methods rather than any findable explicit assertion. From this prospective, we can agree that Property PW for the trivial group was realized soon after the Big Bang and that probably a few years later Lucy considered as folklore that it is also satisfied by $\Z$ (but I would be happy to acknowledge any earlier reference).
 
The first next examples are groups acting properly on trees; however a finitely generated group with this property is necessarily {\bf virtually free}, so this gives a small class of groups. However, unlike the property of acting properly on a tree, Property PW is stable under taking {\bf direct products} and {\bf overgroups of finite index} (see Proposition \ref{pwfi}). It follows, for instance, that every finitely generated, {\bf virtually abelian} group has Property PW; interestingly this provides the simplest counterexamples to the implication FA$\Rightarrow$FW (e.g., a nontrivial semidirect product $\Z^2\rtimes (\Z/3\Z)$ does the job). All the previous examples are groups acting properly on a finite product of trees. Further instances of such groups are {\bf lattices} (or discrete subgroups) {\bf in products of rank one simple groups} over non-Archimedean local fields, and also Burger-Mozes' groups, which are infinite, finitely presented and simple. A typical example that is not cocompact is the {\bf lamplighter} group $F\wr\Z$ (where $F$ is any finite group), which acts properly on the product of two trees, or a non-cocompact lattice such as $\SL_2(\mathbf{F}_q[t,t^{-1}])$.

In chronological order, the next examples seem to be {\bf Coxeter groups}. This was proved by Bozejko, Januszkiewicz and Spatzier \cite{BJS}. Actually, they use a certain action on a space with walls which was explicitly provided (modulo the language of walls), along with all necessary estimates, by Tits \cite[2.22]{Tit}. While this proves Property PW for all Coxeter groups, they claimed as main result the much weaker assertion that their infinite subgroups of Coxeter groups do not have Kazhdan's Property~T. Niblo and Reeves \cite{NRe} refined the result by exhibiting a proper action of any finitely generated Coxeter group on a finite-dimensional proper CAT(0) cube complex; their action is cocompact when the Coxeter group is Gromov-hyperbolic but not in general. Eventually Haglund-Wise \cite{HW2} showed that any finitely generated Coxeter group admits a finite index subgroup with cocompact such action.


Let us point out that at that time, the Haagerup Property, explicitly studied and characterized in Akemann-Walter \cite{AW}, was not yet popular as it became after being promoted by Gromov \cite[7.A,7.E]{Gro} and then in the book by Valette and the other authors of \cite{CCJJV}.

Let us give some more recent examples.
\begin{itemize}
\item Wise proved in \cite{Wi04} that every group with a finite {\bf small cancelation} presentation of type $C'(1/6)$ or $C'(1/4)$-$T(4)$ acts properly cocompactly on a finite-dimensional, locally finite cube complex, and thus have Property PW. Arzhantseva and Osajda \cite{AO} recently proved that finitely generated groups with infinite $C'(1/6)$ presentations also have Property~PW.


\item Bergeron and Wise then obtained a cubulation which, combined with a result of Kahn and Markovic \cite{KM}, implies that every {\bf torsion-free cocompact lattice of} $\SL_2(\C)$ admits a proper cocompact action on a finite-dimensional, locally finite CAT(0) cubing. 

\item Hsu and Wise \cite{HsW} proved that the fundamental group of a graph of free groups with cyclic edge groups, provided it has no non-Euclidean Baumslag-Solitar subgroups, has a proper cocompact action on a locally finite CAT(0) cubing.

\item Ollivier and Wise \cite{OW} proved that Gromov random groups at density $d<1/6$ act properly cocompactly on a finite dimensional, locally finite CAT(0) cube complex. 

\item Gautero \cite{Gau} proved that the Formanek-Procesi group, a non-linear semidirect product of two non-abelian free groups of finite rank, has a proper action on a finite-dimensional CAT(0) cube complex.
\end{itemize}

All the previous examples act properly on a finite-dimensional, locally finite CAT(0) cube complex. A very different example was discovered by Farley, namely {\bf Thompson's groups} \cite{Far} and more generally diagram groups over finite semigroup presentations. This was extended by Hughes to suitable groups of local similarities of compact ultrametric spaces \cite{Hu}\footnote{Note that Hughes and Farley calls a proper commensurating action (i.e., whose associated cardinal definite function is proper) a {\em zipper action}. One reason for which we do not follow this terminology is that in order to understand a (possibly proper) commensurating action, it can be useful to decompose the set into orbits and analyse individually each orbit; when passing to an orbit we generally lose the properness. Commensurated subsets are called {\em immobile} in \cite{Ner}.}.

The author, Stalder and Valette proved in \cite{CSVa} that Property PW is stable under taking {\bf standard wreath products}. In general (e.g., in the case of $\Z\wr\Z$), the methods outputs an infinite-dimensional cube complex even in both groups act on finite-dimensional ones. The statement given in \cite{CSVa} was weaker, because the authors had primarily the Haagerup Property in mind, but the method in \cite{CSV} allows to get the general case. See \S\ref{s_wp} and Proposition \ref{stpw}.

There are natural strengthenings of Property PW, such as, in order of strength: having a proper action on a {\em finite-dimensional} CAT(0) cube complex, having a proper action on a {\em finite-dimensional locally finite} CAT(0) cube complex, having a proper {\em cocompact} action on a finite-dimensional CAT(0) cube complex. We do not address them in this paper, although they are of fundamental importance. These classes are much more restricted. For instance, it was observed in \cite{CSVa} that the wreath product $(\Z/2\Z)\wr F_2$ has Property FW but does not have any proper action on a finite-dimensional CAT(0) cube complex. N.~Wright's result that finite-dimensional CAT(0) cube complexes have finite asymptotic dimension \cite{Wr} provides a wealth of other examples, including the wreath product $\Z\wr\Z$, which has Property PW by \cite{CSVa} (completed in Proposition \ref{stpw}), or Thompson's group $F$ of the interval, which has Property PW by Farley \cite{Far}. I do not know if, for a finitely generated group, having a proper cellular action on a (finite-dimensional~/ proper~/ finite-dimensional and proper) CAT(0) cube complex are distinct properties. [For non-proper actions, there exists \cite{ABJLMS,CK} a finitely generated group $Q$ with no non-trivial cellular action on any locally finite finite-dimensional cube complex; so the free product $Q\ast Q$ also shares this property, but it also acts with unbounded orbits on a (not locally finite) tree.]

\subsection{Examples with the negation of Property FW}\label{negfw}
If $G$ is a finitely generated group and $H$ is a subgroup, let us say that $H$ is {\em coforked} in $G$ if $G/H$ has at least 2 ends. This is called ``codimension 1 subgroup" by Sageev \cite{S95} and in some subsequent papers but this terminology can be confusing as there is in general no underlying notion of dimension, or other unrelated such notions in some cases, and because it is not well-reflected by the geometry of $H$ inside $G$, as shown by the following example:

\begin{exe}
Consider the infinite dihedral group $G_1=\langle a,b\mid a^2,b^2\rangle$ with generating set $\{a,b,aba\}$ and the group $G_2=\langle t,u\mid [t,u],u^2\rangle\simeq\Z\times (\Z/2\Z)$ with generating set $\{tu,u\}$. Then there exists a isomorphism between the Cayley graphs of these groups mapping $\{1,a\}$ to $\{1,u\}$. On the other hand, $\{1,a\}$ is not coforked in $G_1$, while $\{1,u\}$ is coforked in $G_2$.
\end{exe}

Note that the trivial group is coforked in the infinite dihedral group, showing that an overgroup of finite index of a coforked subgroup may not be coforked.
On the other hand, a finite index subgroup of a coforked subgroup is always coforked in the whole group.


All the above examples of Property PW groups, provided they are infinite, fail to have Property FW, i.e., admit a coforked subgroup (assuming they are finitely generated). Note that any group having a quotient without Property FW also fails to have Property FW. In particular, any finitely generated group virtually admitting a homomorphism onto $\Z$, or equivalently having an infinite virtually abelian quotient, fails to have Property FW. In this case, the kernel of a homomorphism of a finite index subgroup onto $\Z$ is a coforked subgroup.

Also, having in mind that Property FW is inherited by finite index subgroups, all groups having a finite index subgroup splitting as an nontrivial amalgam fails to have Property FW, a coforked subgroup being the amalgamated subgroup. Also, countable infinitely generated groups fail to have Property FW, since they do not have Property FA \cite{Ser}.


Ollivier and Wise \cite{OW} proved that Gromov random groups at density $d<1/5$ admit a coforked subgroup. 

A wreath product $A\wr_S B=A^{(S)}\rtimes B$, with $A$ nontrivial, $S$ an infinite $B$-set and $B$ countable never has Property FW (Proposition \ref{wnfw}). Provided $S$ admits an infinite orbit $Bs$, a careful look at the construction shows that $A^{S\smallsetminus\{1\}}$ is a coforked subgroup. Of course if $A$ or $B$ does not have Property PW, the wreath product does not have Property PW. 

The first Grigorchuk group and the Basilica group are examples of self-similar groups with a natural action on a regular rooted tree. The first has subexponential growth and thus is amenable \cite{Gri} while the second has exponential growth \cite{GrZ} but yet is amenable \cite{BV}. Both admit Schreier graphs with 2 ends; for the Grigorchuk group this is established in \cite{GrK} and for the Basilica group this is obtained in \cite{DDMN}.  
In both cases, the coforked subgroup is obtained as the stabilizer of a suitable boundary point of the rooted tree. Accordingly, these groups do not have Property FW; it is not known if they have Property PW.

The derived subgroup of the full topological group associated to an infinite minimal subshift is, by a result of Matui \cite{Ma1}, an infinite, finitely generated simple group; it was subsequently shown to be amenable by Juschenko and Monod \cite{JM}. This group admits a transitive action on $\Z$ for which every element acts with bounded displacement. In particular, it commensurates $\N$; this shows that this group does not have Property FW. It is not known whether it has Property PW.

\subsection{Examples with the negation of Property PW}

For a time the only known source of non-PW (countable) groups was the class of groups without the Haagerup Property, and the only known source of groups without the Haagerup Property was groups with an infinite subgroup with relative Property T. Other examples appeared in \cite{CorJLT}, yet, as observed in \cite{Cor}, they have an infinite {\em subset} with relative Property T. 

The question whether any group with the Haagerup Property has Property PW has circulated for many years until solved in the negative by Haglund, who proved that a group with Property PW has all its infinite cyclic subgroups undistorted (a more direct approach of Haglund's theorem is given as Corollary \ref{zundi}). In particular, plenty of amenable groups do not have Property FW, although all amenable groups have the Haagerup Property \cite{AW}. 

Another approach can be used by used by using Houghton's characterization of coforked subgroups of virtually polycyclic groups: it follows that a virtually polycyclic group has Property PW if and only if it is virtually abelian (see Section~\ref{s_po}).

\subsection{Examples with Property FW}

Groups with Kazhdan's Property T are the initial source of examples of groups with Property FW. The question of finding FW groups without Property T was asked in the second arxiv version of \cite{CDH} and no example seems to have been written before \cite{Cor2}. One example can be derived from Haglund's theorem and a bounded generation argument due to Carter and Keller, namely $\SL_2(\Z[\sqrt{k}])$, where $k$ is an arbitrary non-square positive integer. This group actually has the Haagerup Property. In general, I conjecture:

\begin{conj}
Let $S$ be a connected semisimple Lie group with no compact simple factor, whose Lie algebra has $\R$-rank at least 2. Then every irreducible lattice in $S$ has Property FW.
\end{conj}

(Irreducible means that the projection of the lattice modulo any simple factor of $S$ has a dense image.) The conjecture holds for $S$ when it has Kazhdan's Property T since this is inherited by the lattice. It also includes the above case of $\SL_2(\Z[\sqrt{k}])$, which is an irreducible lattice in $\SL_2(\R)^2$. Further instances of the conjecture are given in \cite{Cor2}.

\subsection{Between Property FW and PW}

The main general question is to describe cardinal definite functions on a given group. Since this is a bit too much, as it is sensitive on the choice of a commensurated subset in its commensuration class, the question can be relaxed to: determine cardinal definite functions up to addition of bounded functions. This has a complete answer for finitely generated abelian groups, see Proposition \ref{decab}.
It would be simplistic to reduce this to the bare study to Properties PW and FW. Also, even if a group has Property PW, there is still an interesting work to understand in which ways it can act properly. The natural study, given a group, say finitely generated, is to consider the class of commensurating actions: the first step being to characterize its coforked subgroups among its subgroups, and then to describe the commensurated subsets in the associated coset spaces (or more generally the finite partitions by commensurated subsets), which essentially amounts to describing the space of ends of the corresponding Schreier graphs.


\section{Commensurating actions and other properties}\label{comp}

If $G$ is a group, we say that a function $G\to\R$ is cardinal definite\footnote{The terminology ``cardinal definite" is in analogy with the notion of measure definite function introduced by Robertson and Steger \cite{RS}, where discrete sets with the counting measure are replaced by general measure spaces.} if there exists a $G$-set $X$ and a $G$-commensurated subset $M\subset X$ such that $f(g)=\ell_M(g)=\#(M\tu gM)$ for all $g\in G$. When $G$ is a topological group, we require, in addition, that $X$ is a continuous discrete $G$-set (i.e., stabilizers of points are open subgroups of $G$) and $M$ has an open stabilizer. We actually check (Lemma \ref{lemcd}) that $f$ is cardinal definite on the topological group $G$, if and only if it is a continuous function on $G$ and is cardinal definite as function on $G^\delta$, which denotes $G$ endowed with the discrete topology.

\subsection{Wallings}

There is a well-known close connection between commensurating actions and wall spaces of Haglund-Paulin \cite{HP}. The suitable definition to make this connection a perfect dictionary is the following. 

\begin{defn}
Given a set $V$ and denoting by $2^V$ its power set, a {\em walling} on $V$ with index set $I$ is a map
\begin{eqnarray*}
I\stackrel{\mathcal{W}}\to 2^V\\
i\mapsto W_i
\end{eqnarray*}
such that for all $u,v\in V$, the number $d_\mathcal{W}(u,v)$ of $i$ such that $W_i$ {\em separates} $u$ and $v$ (in the sense that $u\neq v$ and $\#(W_i\cap\{u,v\})=1$) is finite. The $W_i$ are called halfspaces. 

Given a group $G$ and a $G$-set $V$, a $G$-{\em walling} (or $G$-{\em equivariant} walling) is the data of an index $G$-set $I$ and a walling $I\to 2^V$, which is $G$-equivariant, $G$ acting by left multiplication on the power set $2^V$. 

If $f:V_1\to V_2$ is a map and $\mathcal{W}$ a walling on $V_2$, define the {\em pull-back} of $\mathcal{W}$ to $V_1$ by $f$ as the composite map $f^*\mathcal{W}=(f^{-1})\circ\mathcal{W}$ where $(f^{-1})$ is the inverse image map $2^{V_2}\to 2^{V_1}$.

\end{defn}
Note that $d_\mathcal{W}(u,v)$ is a pseudo-distance on $V$ and, in the equivariant case, is preserved by $G$. Also note that the pull-back $f^*\mathcal{W}$ is clearly a walling, and that the map $f$ such as above is an isometry $(V_1,d_{f^*\mathcal{W}})\to (V_2,d_{\mathcal{W}})$; moreover if $f$ is a $G$-map between $G$-sets and $\mathcal{W}$ is equivariant, then so is $f^*\mathcal{W}$. 

\begin{prop}\label{can_bij}
Given a $G$-set $X$, there is are canonical reciprocal bijections
\begin{eqnarray*}
\{\textnormal{comm. subsets of }X\} & \longrightarrow & \{G\textnormal{-wallings on }G\textnormal{ indexed by }X\}\\
A & \mapsto & \{h\in G\mid x\in hA\}_{x\in X};\\
\{x\in X\mid 1\in W_x\} & \mapsot & (W_x)_{x\in X}.
\end{eqnarray*}
(Here $G$ is viewed as a $G$-set under left multiplication.) If the commensurated subset $A\subset X$ and the walling $\mathcal{W}$ correspond to each other under this bijection, then the pseudo-distances $d_A$ and $d_\mathcal{W}$ are equal.
\end{prop}
\begin{proof}
These maps are reciprocal, defining an even larger bijection between subsets of $X$ and $G$-equivariant maps $X\to 2^G$. Given corresponding $A\subset X$ and the $G$-equivariant mapping $\mathcal{W}:X\to 2^G$, a straightforward verification shows that $A$ is commensurated if and only if $\mathcal{W}$ is a walling, also showing the last statement.
\end{proof}

\begin{cor}\label{ccardw}
Let $G$ be a group and $f:G\to\R$ a function. Equivalences:
\begin{enumerate}
\item\label{wa1} $f$ is cardinal definite;
\item\label{wa2} there exists a $G$-set $V$, a $G$-walling $\mathcal{W}$ on $V$ and $v\in V$ such that $f(g)=d_\mathcal{W}(v,gv)$ for all $g$.
\item\label{wa3} there is a $G$-walling $\mathcal{W}$ on $G$ such that $f(g)=d_\mathcal{W}(1,g)$ for all $g$.
\end{enumerate}
\end{cor}
\begin{proof}
The equivalence between (\ref{wa1}) and (\ref{wa3}) is immediate from Proposition \ref{can_bij} and (\ref{wa3})$\Rightarrow$(\ref{wa2}) is trivial, while (\ref{wa2})$\Rightarrow$(\ref{wa3}) is obtained by pulling back the walling to $G$ by the orbital map $g\mapsto gv$. 
\end{proof}

\begin{rem}
There are several nuances between the original definition of wall spaces and the definition used here:
\begin{enumerate}
\item we allow $\emptyset$ and $V$ among the halfspaces;
\item we allow multiple halfspaces in the sense that $i\mapsto W_i$ is not required to be injective (although clearly, each halfspace occurs with finite multiplicity, except possibly $\emptyset$ and $V$);
\item we use an explicit $G$-structure on the indexing family of the walling, rather than considering a subset of the power set (as in the original definition) or a discrete integral measure on the power set (as in \cite{CSV}). 
\item\label{no2} we consider subsets (that is, half-spaces, as in \cite{CSV}) rather than partitions into two subsets (as in the original definition). Unlike in \cite{NRo,HrWi}, we do not assume stability under complementation.
\end{enumerate}
Among those nuances, the most essential is probably the one in (\ref{no2}). The necessity of considering half-spaces rather than partitions was observed in many places, including \cite{HrWi, CSV}. The construction of \cite{CSVa,CSV} (see \S\ref{s_wp}), which maps a walling (or a measured walling) on a group $\Gamma$ to a walling on the wreath product $A\wr\Gamma$, is very far from symmetric with respect to complementation! (For instance, it maps the walling by singletons to a bounded walling, but maps the walling by co-singletons to an unbounded walling as soon as $\Gamma$ is infinite.)
\end{rem}

\begin{rem}
A variant of the above notion of walling is when the halfspaces are self-indexed. Here, this corresponds to the assumption that the family of halfspaces is $x\mapsto W_x$ is injective. Although easier at first sight, this condition is pretty unnatural; for instance, it prevents the use of pullbacks by non-surjective maps. Also, it does not seriously change the definitions. For example, if $G$ is finitely generated, then for any $G$-invariant walling on $G$, the multiplicity of any halfspace is actually bounded (see Corollary \ref{finw}), so removing redundant halfspaces preserves, for instance, the properness of the wall distance. See also Remark \ref{r_wa}.
\end{rem}

Let us provide a topological version of Proposition \ref{can_bij}. Let $G$ be a topological group. A continuous discrete $G$-set is a $G$-set in which all pointwise stabilizers $G_x$, $x\in X$, are open subgroups of $G$. We call {\em topological} commensurated subset of $X$ a commensurated subset $A$ whose global stabilizer is open in $G$, or equivalently if the cardinal definite function $\ell_A$ is continuous. We will see in Lemma \ref{autcon2} that if $G$ is locally compact, then any commensurated subset of $X$ is actually a topological commensurated subsets, but this is not true for an arbitrary topological group.

\begin{prop}\label{can_bijt}Let $G$ be a topological group and $X$ a discrete $G$-set. The bijection of Proposition \ref{can_bij} actually is a bijection
$$\{\textnormal{comm. subsets of }X\} \stackrel{\sim}\longrightarrow  \{\textnormal{clopen }G\textnormal{-wallings on }G\textnormal{ indexed by }X\}.$$
If the commensurated subset $A\subset X$ and the walling $\mathcal{W}$ correspond to each other under this bijection, then the pseudo-distances $d_A$ and $d_\mathcal{W}$ are equal. In particular, the above bijection restricts to a bijection
$$ \left\{\begin{array}{c}
 \textnormal{top. comm.}  \\
 \textnormal{subsets of }X 
\end{array}\right\}\stackrel{\sim}\longrightarrow
 \left\{\begin{array}{c}
 \textnormal{clopen $G$-wallings of $X$}  \\
 \textnormal{with continuous wall pseudodistance} 
\end{array}\right\}.$$
\end{prop}
\begin{proof}
The stabilizer $G_x$ of $x$ is an open subgroup and we have $G_xW_x=W_x$. Therefore $W_x$, being a union of right cosets of $G_x$, is clopen. Thus the bijection of Proposition \ref{can_bij} actually has the above form. The second statement is contained in Proposition \ref{can_bij} and the third follows.
\end{proof}

\begin{rem}The data of a commensurated subset $A$ in $G/H$ is obviously equivalent to that of a subset $B\subset G$ which is right-$H$-invariant and satisfies $\#((B\tu gB)/H)<\infty$ for all $g\in G$ (up to taking the inverse of $B$, these conditions are precisely those considered by Sageev in \cite[Theorem 2.3]{S95}). Namely, denoting by $\pi$ is the projection $G\to G/H$, we have $B=\pi^{-1}(A)$ and $A=\pi(B)$. In this language, the halfspace $W_g$ associated to $g\in G/H$ in Proposition \ref{can_bij} is just $W_g=B^{-1}g^{-1}$. 
\end{rem}


\subsection{Ends of Schreier graphs and coforked subgroups}

Here we extend the notion given above for finitely generated discrete groups.

\begin{defn}We say that an open subgroup $H$ of a topological group $G$ is {\em coforked} if there is in $G/H$ a commensurated subset, infinite and coinfinite, with an open stabilizer.
\end{defn}
The openness of the stabilizer is automatic when $G$ is locally compact (see Corollary \ref{cauco}), so in that case, if $H$ is an open subgroup of $G$, it is coforked in the topological group $G$ if and only it is coforked in $G_\delta$, where $G_\delta$ denotes $G$ endowed with the discrete topology. There are natural extensions of this definition for the subgroups that are not supposed open, but they are not relevant here (see Remark \ref{r_codim})

\begin{defn}If $G$ is a group and $S$ a symmetric generating subset of $G$, and $X$ is a $G$-set, the {\em Schreier graph} $\textnormal{Sch}(X,S)$ is the graph whose vertex set is $X$, and with an oriented edge $(x,sx)$ labeled by $s\in S$ for each $(s,x)\in S\times X$. We say that $(X,S)$ is of {\em finite type} if the 0-skeleton of the Schreier graph is locally finite (i.e., every vertex is linked to finitely many vertices).
The $S$-{\em boundary} of a set $Y\subset X$ of vertices is the set of elements in $Y$ joined by an edge to an element outside $Y$.
\end{defn}

Note that $(X,S)$ is of finite type when $S$ is finite, but also holds in the important case where $G$ admits a group topology for which $S$ is compact and $X$ is a continuous discrete $G$-set. (In practice, it means that when $G$ is a topological group with a compact generating set, we can apply the following to the underlying discrete group $G^\delta$.)

We have the following, which for $S$ finite is essentially the contents of \cite[Theorem 2.3]{S95}.

\begin{prop}\label{ends}
If $(X,S)$ is of finite type, a subset of $X$ with finite $S$-boundary is commensurated. Conversely, if $G$ is a topological group, $X$ is a continuous discrete $G$-set, $S$ is compact and $M$ is a commensurated $G$-set with open stabilizer, then $M$ has a finite $S$-boundary.

In particular, if $G$ is locally compact and $S$ a compact generating subset, and $X$ is a continuous discrete $G$-set, the set of subsets of $X$ with finite $S$-boundary is equal to the set of $G$-commensurated subsets in $X$ (which in particular does not depend on $S$).
\end{prop}
\begin{proof}
If $(X,S)$ is of finite type and $s\in S$, then $M\tu sM$ is contained in the union of $M\smallsetminus sM$ and $s(M\smallsetminus s^{-1}M)$, which are both finite, so $M$ is commensurated.

Conversely, under the additional assumptions and assuming $M$ commensurated, the function $g\mapsto gM$ is locally constant and thus has a finite image in restriction to the compact set $S$. So the $S$-boundary of $M$, which is the union $\bigcup_{s\in S}M\smallsetminus sM$ and thus is a finite union of finite sets, is finite.

The second statement follows modulo the fact that for a locally compact group and continuous discrete $G$-set, the stabilizer of a commensurated subset is automatically open, see Corollary \ref{cauco}.
\end{proof}

Let us define the space of ends as follows. First, given a graph (identified with its set of vertices) and a subset $Y$, we define $\pi_0(Y)$ as the set of components of $Y$, where two vertices in $Y$ are in the same component if and only if they can be joined by a sequence of edges with all vertices in $Y$. Then the space of ends of a graph $L$ is the (filtering) projective limit of $\pi_0(L\smallsetminus F)$, where $F$ ranges over finite subsets of $L$; endowing $\pi_0(L\smallsetminus F)$ with the discrete topology, the projective limit is endowed with the projective limit topology; if $L$ has finitely many components and has a locally finite 0-skeleton, this is a projective limit of finite discrete sets and thus is a compact totally disconnected space, and is metrizable if $L$ is countable.

Moreover, the space of ends is nonempty if and only if $L$ is infinite, and has at least two points if and only if there exists a finite set of vertices whose complement has two infinite components.
Also note that the space of ends of $X$ is the disjoint union of ends of spaces of its connected components.

\begin{cor}Let $G$ be a locally compact, compactly generated group and $H$ an open subgroup in $G$. Then $H$ is coforked in $G$ if and only if the Schreier graph of $G/H$, with respect to some/any compact generating subset of $G$, has at least 2 ends.\qed
\end{cor}

The space of ends can be conveniently described using Boolean algebras. Given a Boolean algebra $A$, define a {\em multiplicative character} as a unital ring homomorphism of $A$ onto $\Z/2\Z$. Multiplicative characters are also sometimes called ultrafilters, because a multiplicative character on a power set $2^X=(\Z/2\Z)^X$ is the same as an ultrafilter on $X$, and an multiplicative character on $2^X/2^{(X)}$ is the same as a non-principal ultrafilter on $X$. The set $\chi(A)$ multiplicative characters of $A$ is a compact set, for the topology of pointwise convergence. A basic result is the following


\begin{lem}\label{ssstone}
For any Boolean algebra $A$, we have $\bigcap_{\xi\in\chi(A)}\Ker(\xi)=\{0\}$. In particular, we have $\chi(A)=\emptyset$ if and only if $A=\{0\}$.
\end{lem}
\begin{proof}
Clearly, a Boolean algebra admits no nonzero nilpotent element. 
On the other hand, if $A$ is an arbitrary commutative ring, an elementary application of Zorn's lemma shows that the intersection of all prime ideals is the set of nilpotent elements \cite[Theorem 1.2]{Mat}. Thus if $A$ is a Boolean algebra, this intersection is trivial. Moreover, a Boolean algebra which is a domain is necessarily equal to $\{0,1\}$; in other words, the prime ideals in a Boolean algebra are exactly the kernels of multiplicative characters. So the proof is complete.
\end{proof}

If $E$ is a topological space, write $C(X)$ for the Boolean algebra of continuous functions $E\to\Z/2\Z$.

\begin{thm}[Stone's representation theorem]
Let $E$ be a topological space. Then we have a canonical continuous map $E\to \chi(C(E))$, given by evaluation. If $E$ is a totally disconnected compact space, then this is a homeomorphism.

Let $A$ be a Boolean algebra. Then we have a canonical unital algebra homomorphism $A\to C(\chi(A))$, given by evaluation, and it is an isomorphism.
\end{thm}
\begin{proof}The injectivity of $E\to\chi(C(E))$ is equivalent to the fact that clopen subsets separate points. For its surjectivity, consider $\phi\in\chi(C(E))$; it is enough to show that there exists $x\in E$ such that $\phi(f)=f(x)$ for all $f\in C(E)$. Assume the contrary: for all $x\in E$ there exists $f\in C(E)$ such that $\phi(f)=1$ and $f(x)=0$ (as we can choose by replacing $f$ by $1-f$ if necessary). Since $E$ is compact and totally disconnected, hence 0-dimensional, this implies that there exists a clopen partition $E=E_1\sqcup\dots\sqcup E_k$ and $f_1,\dots, f_k\in C(E)$ such that $f_i$ vanishes on $E_i$ and $\phi(f_i)=1$. Hence if $f=\prod f_i$ we get $f=0$ and $\phi(f)=1$, a contradiction.

The injectivity of $A\to C(\chi(A))$ is established by Lemma \ref{ssstone}. For the surjectivity, it is enough to show that $1_\Omega$ is in the image for any basic clopen subset of $C(\chi(A))$, namely of the form $\Omega=\{\phi:\phi(f_1)=\phi(f_2)=\dots=\phi(f_k)=1\}$. Since taking $f=\prod f_i$ it can also be described as $\Omega=\{\phi:\phi(f)=1\}$, we see that $1_\Omega$ is the image of $f$.
\end{proof}

The assignments $E\mapsto C(E)$ and $\mapsto\chi(A)$ are contravariant functors from the category of topological spaces to the category of Boolean algebras, and the Stone representation theorem thus shows that they define inverse equivalences of categories between the category of totally disconnected compact topological spaces and the category of Boolean algebras.

We now turn back to our setting.
As above, we assume that $(X,S)$ is of finite type; for instance $G$ is a compactly generated locally compact group and $S$ is a compact generating subset. Let $X$ be a discrete $G$-set, and let $\textnormal{Comm}_G(X)$ the set of topological $G$-commensurated subsets of $X$, which is a Boolean subalgebra of $2^X$. For $M\subset X$ and $x\in X$, we define $\delta_x(M)$ to be 1 or 0 according to whether $x\in M$. 

\begin{prop}
If $X$ has finitely many $G$-orbits, the natural compactification
\begin{eqnarray*}
X & \to &\chi(\textnormal{Comm}_G(X))\\
x & \mapsto & \delta_x,
\end{eqnarray*}
coincides with the end compactification of the Schreier graph $\textnormal{Sch}(X,S)$, and restricts 
to a natural homeomorphism between the set of ends of the graph $\textnormal{Sch}(X,S)$ and the space of multiplicative characters $\chi(\textnormal{Comm}_G(X)/2^{(X)})$.
\end{prop}
\begin{proof}
We can view $\chi(\textnormal{Comm}_G(X)/2^{(X)})$ as a closed subset of $\chi(\textnormal{Comm}_G(X))$, consisting of those multiplicative characters vanishing on $2^{(X)}$. An easy verification shows that its complement consists of the Dirac multiplicative characters $(\delta_x)_{x\in X}$. This is an open discrete, and dense subset in $\chi(\textnormal{Comm}_G(X))$.

To check that this is exactly the end compactification of $\textnormal{Sch}(X,S)$, we just have to show that given a graph with locally finite 0-skeleton, with vertex set $L$ and finitely many components, if its set of ends is defined as above, then it is naturally identified with $\chi(Q(L)/2^{(L)})$, where $Q(L)$ is the set of subsets of $L$ with finite boundary. 

So $Q(L)/2^{(L)}$ is the inductive limit, over finite subsets $F$ of $L$, of the $Q_F(L)/2^F$, where $Q_F(L)$ is the set of subsets of $L$ with boundary contained in $F$; thus $\chi(Q(L)/2^{(L)})$ is the projective limit of all $\chi(Q_F(L)/2^F)$. The Boolean algebra $Q_F(L)/2^F$ can be described as the  set of unions of connected components of $L\smallsetminus F$. Since the latter is finite, $\chi(Q_F(L)/2^F)$ can be thus identified with the set of connected components of $L\smallsetminus F$. Thus $\chi(Q(L)/2^{(L)})$ is identified with the set of ends of $L$.
\end{proof}

We will prove in the sequel that for a compactly generated locally compact group $G$, Property FW can be tested on transitive $G$-sets (see Proposition \ref{fwfwp}(\ref{fwww})). In particular, such a group $G$ has Property FW if and only it has no coforked open subgroup.

\begin{rem}\label{r_codim}
In this paper we only consider coset spaces of open subgroups. But ends of pairs of groups $H\subset G$ with $G$ locally compact compactly generated and $H$ closed in $G$ can also be considered. We do not give the general definition here (due to Houghton \cite{Ho74}), but in case $G$ is a connected Lie group, it corresponds to the ends of the connected manifold $G/H$. For instance, the upper unipotent subgroup $U$ of $\SL_2(\R)$ has a 2-ended coset space $\SL_2(\R)/U$. More generally, for $n\ge 2$, in $\SL_n(\R)$, the stabilizer of a nonzero vector in $\R^n$ has a 2-ended coset space, namely $\R^n\smallsetminus\{0\}$ (which is homeomorphic to $\R\times\mathbf{S}^{n-1}$). Since $\SL_n(\R)$ has Kazhdan's Property T for $n\ge 3$, this shows that the present study does not carry over this context.
\end{rem}

\subsection{On the language of ``almost invariant" subsets}

In the literature, the following language is sometimes used: let $G$ be a group and $H$ a subgroup. An $H$-almost invariant subset of $G$ means a left-$H$-invariant subset of $G$ whose projection on $H\backslash G$ is commensurated by the right action of $G$.

The usual data in this context is the following: a finite family $(H_i)$ of subgroups of $G$ and an $H_i$-almost invariant subset $X_i$ of $G$, and $E$ is the family of left $G$-translates of all $H_i$.

We should emphasize that these data are equivalent to that of a $G$-set $X$ and a commensurated subset $M\subset X$, with the additional specification of a finite family $(x_1,\dots,x_n)$ including one element in each orbit. Namely given the above data, consider the disjoint union $\bigsqcup_{i=1}^n H_i$ and $M=\bigcup X_i^{-1}/H_i$; conversely given $X$, $M$, and $x_i$ as above, we can define $H_i$ as the stabilizer of $x_i$ and $X_i=\{g\in G\mid g^{-1}x_i\in M\}$.

Inasmuch as the indexing of orbits and choice of representative points is artificial, the data of the $G$-set $X$ and the commensurated subset $A$ seem enough. Actually, the family of the left translates of the $X_i^{-1}$ are nothing else than the walling $(W_x)$ given in Proposition \ref{can_bij}. To emphasize the difference of point of views, we can observe that when we consider an action of $G$ on a set with a commensurated subset, it is very natural to restrict this action to a smaller subgroup (this is used many times in \S\ref{s_ab}, when we restrict to abelian subgroups), while that this does look natural using the additional data (notably because the orbit decomposition of the subgroup is not the same).

\begin{rem}
Given a $G$-invariant involution $\sigma:X\to X$, the condition that $W_{\sigma(x)}=W_x^*$ for all $x$ ($^*$ denoting the complement) is equivalent to the requirement that $A$ is a fundamental domain for $\sigma$, i.e., $X=A\sqcup\sigma(A)$ (see \S\ref{icm}). 
\end{rem}


The Kropholler conjecture, usually termed in the previous language can be restated as follows (bi-infinite means infinite and co-infinite):

\begin{conj}[Kropholler]
Let $G$ be a finitely generated group and $G/H$ a transitive $G$-set. Assume that $G$ commensurates an bi-infinite subset $M\subset G/H$. Assume in addition that $M$ is $H$-invariant. Then $G$ splits over a subgroup $L$ commensurate (in the group sense) to $H$, i.e.\ such that $H\cap L$ has finite index in both $H$ and $L$.
\end{conj}

The conjecture holds when $H$ is finite by Stallings' theorem. The assumption that $M$ is $H$-invariant is essential as otherwise the conjecture would infer that Properties FW and FA are equivalent (which is not the case, see Example \ref{FApFW}).

Dunwoody extends the conjecture to arbitrary discrete groups; the conclusion being replaced by the existence of an unbounded action of $G$ on a tree for which each edge orbit contains an edge whose stabilizer is commensurate (in the group sense) to $H$.

Let us also mention, in the language of commensurated subsets, a basic lemma of Scott \cite[Lemma~2.3]{Sco98}. Let $G$ be a group acting on a set $X$ with a commensurated subset $M$. Define (as in Proposition \ref{can_bij}) $W_x=\{g\in G\mid x\in gM\}$. Denote by $p_x$ the anti-orbital map $g\mapsto g^{-1}x$; note that it maps $W_x$ into $M$.

\begin{prop}
Let $G$ be a compactly generated locally compact group and $X$ a continuous discrete $G$-set with commensurated subset $M$. Let $x,y\in X$ be points such that $M\cap Gx$ and $M\cap Gy$ are infinite. Then $p_x(W_x\cap W_y)$ is finite if and only if $p_y(W_x\cap W_y)$ is finite. 
\end{prop}
\begin{proof}
By symmetry, we only have to prove the forward implication. Fix a compact symmetric generating subset $S$ of $G$. Assume that $p_x(W_x\cap W_y)$ is finite. Since $M\cap Gx$ is not transfixed, there exists $z\in M$ not in $p_x(W_x\cap W_y)$. Since the latter is finite, there exists $k$ such that for every $x'\in p_x(W_x\cap W_y)$ there exists $s\in S^k$ such that $z=sx$. In other words, for every $g\in W_x\cap W_y$, there exists $s\in S^k$ such that $z=sgx=p_x(g^{-1}s^{-1})$. Since $z\in M\cap Gx$ and $z\notin p_x(W_x\cap W_y)$, we have $g^{-1}s^{-1}\in W_x$ and it follows that $g^{-1}s^{-1}\notin W_y$, i.e.\ $sgy\notin M$. Thus we have proved that for every $v\in p_y(W_x\cap W_y)$ (which is contained in $M$), there exists $s\in S^k$ such that $sv\notin M$. So the elements in $p_y(W_x\cap W_y)$ are at bounded distance to the boundary of $M$ in the Schreier graph of $X$ with respect to $S$. Since the latter is locally finite and $M$ has a finite boundary, we deduce that $p_y(W_x\cap W_y)$ is finite. 
\end{proof}

If $M\cap Gx$ is finite but not $M\cap Gy$, the conclusion of the previous proposition fails, as shown by elementary examples in \cite[Remark 2.4]{Sco98}.

\section{Commensurating actions of topological groups}\label{comto}

\subsection{The commensurating symmetric group}

If $X$ is a set, let $\SX(X)$ be the group of its permutations. It is endowed with its usual Polish topology, which is a group topology, for which a basis of neighborhoods of the identity is given by point stabilizers. 

Given a topological group $G$ and an abstract $G$-set $X$ endowed with the discrete topology, it follows from the definitions that the following conditions are equivalent:
\begin{itemize}
\item for every $x\in X$, the stabilizer $G_x$ is open in $G$;
\item the structural map $G\times X\to X$ is continuous;
\item the structural homomorphism $G\to\SX(X)$ is continuous.
\end{itemize}
We then say that $X$ is a continuous discrete $G$-set.

Consider now a set $X$ and a subset $M$. Let $\SX(X,M)$ be the group of permutations of $X$ commensurating $M$. It acts faithfully on the set $\textnormal{Comm}_M(X)$ of subsets of $X$ commensurate to $M$. We endow $\SX(X,M)$ with the group topology induced by the inclusion in $\SX(\textnormal{Comm}_M(X))$. 

\begin{lem}\label{basxm}
A basis of neighborhoods of the identity in $\SX(X,M)$ is given by the subgroups $H_M(F)$ for $F$ finite subset of $G$, where $H_M(F)$ is the pointwise stabilizer of $F$ in the stabilizer of $M$. In particular, the inclusion $\SX(X,M)\to\SX(X)$ is continuous.
\end{lem}
\begin{proof}
Let us first observe that the action of $\SX(X,M)$ on $X$ is continuous. Indeed, the stabilizer of $x\in X$ contains the intersection of the stabilizers of $M\cup\{x\}$ and of $M\smallsetminus\{x\}$, and therefore is open. This shows that all $H_M(F)$ are open in $\SX(X,M)$.

Let $\mathcal{P}$ be a finite subset of $\textnormal{Comm}_M(X)$ and $H$ its pointwise stabilizer. Define $F=\bigcup_{N\in\mathcal{P}}M\tu N$. Then $F$ is finite and $H$ contains $H_M(F)$. This shows that the $H_M(F)$ form a basis of neighborhoods of the identity.
\end{proof}

Given a continuous discrete $G$ set and commensurated subset, a natural requirement is that $M$ has an open stabilizer. This is (obviously) automatic if $G$ is discrete, but also, more generally, when $G$ is locally compact, or has an open Polish and separable subgroup, see \S\ref{auco}.
It follows from Lemma \ref{basxm} that this holds if and only if the homomorphism $G\to\SX(X,M)$ is continuous. 

We should note that the automatic continuity does not hold in general. The simplest counterexample is the tautological one: if $X$ is any infinite set and $M$ an infinite subset with infinite complement, then $\SX(X,M)$ is dense in $\SX(X)$ and we see that the stabilizer of $M$ in $\SX(X)$ is not open. Therefore, if we define $G$ as the group $\SX(X,M)$ endowed with the topology induced by the inclusion into $\SX(X)$, then $X$ is a continuous discrete $G$-set and $M$ is commensurated but does not have an open stabilizer.

\begin{rem}
There is a natural faithful action of $\SX(X,M)$ on the two-point compactification of $X$ given by the disjoint union of the one-point compactifications of $M$ and its complement. The compact-open topology then coincides with the topology of $\SX(X,M)$ described above. When both $M$ and its complement are infinite, $\SX(X,M)$ has index 1 or 2 in the full homeomorphism group of this 2-point compactification (the index is 2 precisely when $M$ and its complement have the same cardinality).
\end{rem}

\begin{defn}\label{d_cd}Let $G$ be a topological group. We define a {\em cardinal definite} function on $G$ as a function of the form $\ell_M$ for some continuous discrete $G$-set $X$ and commensurated subset $M\subset X$ with open stabilizer.
\end{defn}

\begin{lem}\label{cardefsu}
On a topological group, a sum of finitely many cardinal definite functions is cardinal definite. More generally, an arbitrary sum of cardinal definite functions, if finite and continuous, is cardinal definite.
\end{lem}
\begin{proof}
Let $(\ell_i)_{i\in I}$ be cardinal definite functions on $G$. Write $\ell_i=\ell_{M_i}$ with $M_i\subset X_i$ is a commensurated subset with open stabilizer, $X_i$ being a continuous discrete $G$-set. Assume that $\sum\ell_i$ is everywhere finite and continuous. Let $X$ be the disjoint union of the $X_i$ and $M\subset X$ the union of the $M_i$. Then $X$ is a continuous discrete $G$-set, $M$ is commensurated by $G$ and $\ell_M=\sum\ell_i$; since it is continuous, $M$ has an open stabilizer (namely $\{\ell_M=0\}$) and thus $\ell_M$ is cardinal definite.
\end{proof}

Let $G$ be a topological group. Denote by $G_\delta$ the group $G$ endowed with the discrete topology.

\begin{lem}\label{lemcd}
A function $G\to\R$ is cardinal definite if and only if it is cardinal definite on $G_\delta$ and is continuous on $G$.
\end{lem}
\begin{proof}
The ``only if" condition is clear. Conversely, suppose that $\ell$ is continuous and cardinal definite on $G_\delta$. So $\ell=\ell_M$ for some $G$-set $X$ and commensurated subset $M$. Note that $X$ may not be a continuous $G$-set. Decomposing $X$ into $G$-orbits, each of the corresponding cardinal definite functions on $G_\delta$ are continuous (since they vanish on the stabilizer of $M$, which is open). Hence in view of Lemma \ref{cardefsu}, we can suppose that $X$ is $G$-transitive. We can suppose that $M$ is not invariant; thus $\bigcup_{g,h\in G}gM\tu hM$ is not empty; being $G$-invariant, it is therefore equal, by transitivity, to all of $X$. The Boolean algebra generated by the $gM\tu hM$ when $g,h$ range over $G$ defines an invariant partition of $X$ by finite subsets; by transitivity all these subsets have the same cardinal $n$.
 
Note that $M$ is a union of components of the partition. Define $X'$ as the quotient of $X$ by this partition and $M'$ the image of $M$ in $X'$. The stabilizer of any element of $X'$ is the stabilizer of some component of the partition of $X$ and thus, as a finite intersection of subsets of the form $gM\tu hM$, is open. So $X'$ is a continuous discrete $G$-set. Since $\ell_M$ is open, the stabilizer of $M$, and hence of $M'$, is open. So $\ell_{M'}$ is cardinal definite. Since $\ell_{M}(g)=n\ell_{M'}(g)$, we deduce from Lemma \ref{cardefsu} that $\ell_M$ is cardinal definite. 
\end{proof}


\begin{que}
If $P$ is a set, define a {\em cardinal definite kernel} on $P$ as a function $P\times P\to\R$ of the form \[(x,y)\mapsto\kappa(x,y)=\#(A_x\tu A_y)\] for some set $X$ and function $x\mapsto A_x$ from $P$ to the power set of $X$, requiring that $\kappa$ takes finite values, i.e.\ that the $A_x$ are in a single commensuration class. Clearly is $G$ is a group and $f$ is a cardinal definite function on $G$ then $(g,h)\mapsto f(g^{-1}h)$ is a left-invariant cardinal definite kernel. Conversely, if $\kappa$ is a left-invariant cardinal definite kernel, is the function $g\mapsto 2f(1,g)$ cardinal definite? The analogous question for measure definite kernels and functions is also open, see \cite[Remark 2.9]{CSV}. (Here the factor 2 is needed, otherwise this is false, e.g.\ if $G$ is a cyclic group of order 2.)
\end{que}

\subsection{Cofinality $\neq\omega$}\label{s_cof}

Here we give a finiteness result on commensurating actions of topological groups with uncountable cofinality. It is useful even in the special case of a finitely generated acting group. The following definition is classical for discrete groups.

\begin{defn}
We say that a topological group has {\em uncountable cofinality} (or {\em cofinality} $\neq\omega$) if it cannot be written as the union of an infinite (strictly) increasing sequence of open subgroups, or equivalently if any continuous isometric action of $G$ on any ultrametric space has bounded orbits. \end{defn}

For instance, if $G$ is generated by a compact subset, then it has cofinality $\neq\omega$. For $\sigma$-compact locally compact groups, the converse holds; in particular, a countable discrete group has cofinality $\neq\omega$ if and only if it is finitely generated. On the other hand, there exist uncountable discrete groups with cofinality $\neq\omega$, such as the full group of permutations of any infinite set.


\begin{prop}\label{cgco2}
Let $G$ be a topological group with uncountable cofinality. Let $X$ be a discrete continuous $G$-set.
Let $M\subset X$ be a commensurated subset with an open stabilizer, and let $(X_i)_{i\in I}$ be the orbit decomposition of the $G$-set $X$. Then $X_i\cap M$ is $G$-invariant for all but finitely many $i$'s. 
\end{prop}
\begin{proof}
Let us first give the argument when $G$ is generated by a compact set $S$. Since $M$ has an open stabilizer, the function $g\mapsto gM$ is locally constant and thus has a finite image in restriction to $S$, and in particular the union $\bigcup_{s\in S}(M\tu sM)$ is finite and thus is contained in the union of finitely many $G$-orbits; if $Z$ is another $G$-orbit, it follows that $Z$ is $G$-invariant.

Let us now prove the general case. Let $J$ be the set of $i$ such that $X_i\cap M$ is not $G$-invariant. We need to show that $J$ is finite. Otherwise, there exists a decreasing sequence of non-empty subsets $J_n\subset J$ with $\bigcap J_n=\emptyset$. Define 
\[G_n=\{g\in G:\;\forall i\in J_n:g(M\cap X_i)=M\cap X_i\}.\]
Note that $G_n$ contains the stabilizer of $M$, which is open by assumption. 
So $(G_n)$ is an ascending sequence of open subgroups, and $\bigcup G_n=G$ because for a fixed $g$, if $n$ is large enough, the finite subset $M\tu gM$ does not intersect $\bigcup_{i\in J_n}X_i$. So by the cofinality assumption, $G=G_n$ for some $n$, that is, $M\cap X_i$ is $G$-invariant for all $i\in J_n$. This contradicts the definition of $J$.
\end{proof}

\begin{cor}\label{ficd}
Let $G$ be a topological group with uncountable cofinality. Then every cardinal definite function on $G$ is a finite sum of cardinal definite functions associated to transitive actions of $G$.\qed
\end{cor}

Using the dictionary between commensurating actions and wallings (Proposition \ref{can_bijt}), we get

\begin{cor}\label{finw}
Let $G$ be a topological group with uncountable cofinality. Then for every clopen $G$-walling on $G$, there are finitely many $G$-orbits of halfspaces distinct from $\emptyset$ and $G$. In particular, halfspaces distinct from $\emptyset$ and $G$ have a bounded multiplicity.
\qed\end{cor}

Proposition \ref{cgco2} also has a geometric interpretation for actions on median graphs, see Corollary \ref{fmo}.


\subsection{Automatic continuity}\label{auco}



\begin{prop}\label{autcon1}
Let $G$ be a Baire separable topological group (e.g., a Polish group). 
Let $G$ act continuously on a discrete set $X$ with a commensurated subset $M$.
 Then the corresponding homomorphism $G\to\SX(X,M)$ is continuous (i.e., $M$ has an open stabilizer, or, still equivalently, $\ell_M$ is continuous).
\end{prop}
\begin{proof}
Let $D$ be a dense countable subgroup. Define $Y=\bigcup_{g\in D}M\tu gM$. Then $Y$ is a countable $D$-invariant subset of $X$; since $Y$ has a closed stabilizer in $G$, it follows that $Y$ is $G$-invariant. Then $M\cap Y^c$ is $G$-invariant. Thus the stabilizer $H$ of $M$ in $G$ is equal to the stabilizer of $M\cap Y$, which has countable index since the orbit of $M\cap Y$ ranges over subsets of $Y$ that are commensurate to $M\cap Y$. Moreover, $H$ is a closed subgroup of $G$ since the $G$-action on $X$ is continuous. By the Baire property, it follows that $H$ is an open subgroup of $G$. 
\end{proof}

Note that the result immediately extends to topological groups having a dense separable open subgroup. Let us now provide a result encompassing locally compact groups.

For a topological group $G$, consider the following property:

\begin{itemize}
\item[(*)] if $o(G)$ is the intersection of all open subgroups of $G$, then $G/o(G)$ is non-Archimedean, in the sense that it admits a basis of neighborhoods of 1 consisting of open subgroups. 
\end{itemize}


\begin{exe}
For a Hausdorff topological group $G$, we have the following chain of implications: $G$ is non-Archimedean $\Rightarrow$ the intersection of open subgroups of $G$ is reduced to $\{1\}$ $\Rightarrow$ $G$ is totally disconnected.

For a locally compact group, all implications are equivalences and thus every locally compact group satisfies (*). 
However, in general both implications are not equivalences, even in the realm of abelian Polish groups. An example of a nontrivial Polish group that is totally disconnected, but with no open subgroup other than itself is given by C.~Stevens in \cite{Ste}. An example of a Polish group in which the intersection of open subgroups is trivial, but that is not non-Archimedean is the group of functions $\N\to\Q$ such that $\lim_{+\infty}f=0$, under addition, with the topology defined by the invariant complete distance 
\[d(f,g)=\|f-g\|_\infty+\sum_{n\ge 0}2^{-n}\delta_{0,f(n)-g(n)},\]
where $\delta$ is the Kronecker symbol.

Note that all topological subgroups of $\SX(X)$, for any set $X$, are non-Archimedean and thus satisfy (*). Actually, a Polish group is non-Archimedean if and only if it is isomorphic to some closed subgroup of $\SX(\N)$.
\end{exe}


\noindent 


\begin{prop}\label{autcon2}
Let $G$ be a Baire topological group with Property (*). 
Then $G$ satisfies the automatic continuity property of Proposition \ref{autcon1}.
\end{prop}
\begin{proof}

Use the notation of Proposition \ref{autcon1}. Observe that
\begin{align*}\ell_M(g)= & \#(M\tu gM)=\#(gM\smallsetminus M)+\#(g^{-1}M\smallsetminus M)\\ =& \sum_{x\notin M}\mathbf{1}_{gM}(x)+\mathbf{1}_{g^{-1}M}(x).\end{align*}
Since each $x\in X$ has an open stabilizer, each function $x\mapsto u_x(g)=\mathbf{1}_{gM}(x)+\mathbf{1}_{g^{-1}M}(x)$ is continuous, as well as each finite sum of these. It follows that $\ell_M$, as a filtering supremum of continuous functions, is lower semicontinuous. Hence for every $r$, $K_r=\{x\in G:\ell(x)\le r\}$ is closed. By the Baire property, there exists $r$ such that $K_r$ has non-empty interior. Note that $K_r$ is symmetric and $K_rK_r\subset K_{2r}$. It follows that $K_{2r}$ is a neighborhood of 1 in $G$. 


The action of $o(G)$ on $X$ is trivial and $\ell_M$ is $o(G)$-invariant; thus $\ell_M$ is bounded on $o(G)K_{2r}$. By (*), the latter contains an open subgroup $L$ of $G$. Since $\ell_M$ is bounded on $L$, the subset $M$ is $L$-transfixed (by Theorem \ref{btx}), and hence has the same stabilizer as some finite subset of $X$. By continuity of the action on $X$, we deduce that $M$ has an open stabilizer.
\end{proof}






\begin{cor}\label{cauco}
Let $G$ be a locally compact group. Then for every continuous discrete $G$-set, every commensurated subset has an open stabilizer. In other words, for every pair of sets $M\subset X$ and homomorphism $f:G\to\SX(X,M)$, the continuity of the composite map $G\to\SX(X)$ implies the continuity of $f$.\qed
\end{cor}

\subsection{Affine $\ell^p$-action}

 
Let $X$ be a set and $M$ a subset. We denote by $\R^X$ the space of all functions $X\to\R$, and by $\ell^\circ(X)=\R^{(X)}$ its space consisting of of finitely supported functions. 
By $p$ we denote any real number in $[1,\infty\mathclose[$. By the symbol $\star$, we mean either $p$ or $\circ$.
Define
\[\ell^\star_M(X)=\{f\in\R^X:f-1_M\in\ell^\star(X)\}=\ell^\star(X)+1_M.\]  It only depends on the commensuration class of $M$, and for $M$ finite it is equal to $\ell^\star(X)$. The subset $\ell^p_M(X)$ is endowed with a canonical structure of an affine space over $\ell^p(X)$ and the corresponding $\ell^p$-distance. We have $\ell^\circ_M(X)\subset \ell^p_M(X)\subset \ell^q_M(X)$ for all $p\le q$. We also denote, for $I\subset\R$, the subset $\ell^\star_M(X,I)$ as the set of elements in $\ell^\star_M(X)$ with values in $I$.


There is a natural linear action of $\SX(X)$ on $\R^X$, which preserves the subspaces $\ell^\star(X)$. The stabilizer of each of the affine subspaces $\ell^\star_M(X)$ is precisely $\SX(X,M)$.

\begin{lem}For every $p$, the action of $\SX(X,M)$ on $\ell^p_M(X)$ is continuous.
\end{lem}
\begin{proof}
Since this action is isometric, it is enough to check that the orbital map $i_x:g\mapsto gx$ is continuous for every $x$ ranging over a dense subset, namely $\ell^\circ_M(X)$. For such an $x$, the stabilizer is open and thus the continuity of the orbital map follows.
\end{proof}

Let us observe that the normed affine spaces $\ell^\star_M(X)$ as well as the actions of $\SX(X,M)$ only depend on the commensuration class of $M$.

We endow $\SX(X,M)$ with the left-invariant pseudo-distance $d_M(g,h)=\#(gM\tu hM)$. Note that is is continuous, but does not define the topology of $\SX(X,M)$ since it is not Hausdorff (for $\#(X)\ge 3$); however the topology of $\SX(X,M)$ is defined by the family of pseudo-distances $d_N$ when $N$ ranges over subsets commensurate to $M$. We have $d_M(g,h)=\ell_M(g^{-1}h)$, with the length $\ell_M$ defined by $\ell_M(g)=\#(M\tu gM)$.

\begin{prop}\label{sxlp2}
The action of $\SX(X,M)$ on $\ell^p_M(X)$ is faithful, continuous and metrically proper. Moreover, the injective homomorphism
\[\alpha_p:\SX(X,M)\to\textnormal{Isom}(\ell^p_M(X))\]
has a closed image, namely the set $\Xi_M^p(X)$ of affine isometries of $\ell^p_M(X)$ that preserve the set of points in $\ell^p_M(X,\{0,1\})$, and whose linear part preserves the closed cone $\ell^p(X,[0,\infty\mathclose[)$.
\end{prop}
\begin{proof}
Note that the set $\ell^p_M(X,\{0,1\})$ is equivariantly identified with the set of indicator functions of elements of $\textnormal{Comm}_M(X)$. Since $\SX(X,M)$ acts faithfully on $\textnormal{Comm}_M(X)$, it follows that the action on $\ell^p_M(X)$ is faithful.

Another consequence is that if both $M$ and its complement are infinite, $\alpha_p(\SX(X,M))$ and $\Xi_M^p(X)$ are both transitive on $\ell^p_M(X,\{0,1\})$. If $M$ or its complement is finite, it still holds that $\alpha_p(\SX(X,M))$ and $\Xi_M^p(X)$ have the same orbits on $\ell^p_M(X,\{0,1\})$, by an argument left to the reader.

Let us check that $\alpha_p(\SX(X,M))=\Xi_M^p(X)$. The inclusion $\subset$ is clear; conversely, given $\phi\in\Xi_M^p(X)$, after composition by an element of $\alpha_p(\SX(X,M))$ (using the previous observation about orbits), we obtain an element $\phi_1$ with $\phi_1(1_M)=1_M$. The 1-sphere in $\ell^p_M(X,\{0,1\})$ around $1_M$ can be described as the disjoint union $A\sqcup B$, wherein 
$A$ consists of elements of the form $1_{M\cup\{x\}}=1_M+\delta_x$ for $x\notin M$ and $B$ of elements of the form $1_{M\smallsetminus\{x\}}=1_M-\delta_x$ for $x\in M$. Note that $\phi_1$ preserves this 1-sphere. Since it moreover satisfies the condition on the linear part, it preserves both $A$ and $B$. Thus the actions of $\phi_1$ on $A$ and $B$ defines a permutation $\sigma$ of $X$ preserving $M$ by $\phi_1(1_M+\delta_x)=1_M+\delta_{\sigma(x)}$ for $x\notin M$ and $\phi_1(1_M-\delta_x)=1_M-\delta_{\sigma(x)}$ for $x\in M$. Thus $\alpha_p(\sigma)$ and $\phi_1$ coincide on $1_M$ and $A\cup B$, which together generate affinely $\ell^p_M(X)$. Thus $\phi_1=\alpha_p(\sigma)$ and we deduce that $\phi\in\alpha_p(\SX(X,M))$.

Finally we see the metric properness as a consequence of the equality $\|g1_M-h1_M\|^p=d_M(g,h)$.
\end{proof}

\subsection{Boundedness and commensuration}

Note that if $N$ is commensurate to $M$, then $d_M-d_N$ is bounded. In particular, the bornology on $\SX(X,M)$ defined by $d_M$ is canonical.

The affine action gives a short proof of the following combinatorial result of Brailovsky, Pasechnik and Praeger \cite{BPP}. Recall from the introduction that in a $G$-set $X$, a subset $M$ is {\em transfixed} if there is a $G$-invariant subset $N$ commensurate to $M$, i.e.\ satisfying $\#(M\tu N)<\infty$.

\begin{thm}\label{btx}
Let $G$ be a subgroup of $\SX(X,M)$. Then $\ell_M(G)$ is bounded if and only if $M$ is transfixed by $G$.
\end{thm}
\begin{proof}
Obviously if $N$ is commensurate to $M$ and $G$-invariant then $d_M$ is bounded by $2\#(M\tu N)$ on $G$. Conversely, assume that $d_M$ is bounded on $G$. Then the action of $G$ on $\ell^2_M(X)$ has bounded orbits. By the center lemma (see \cite[Lemma 2.2.7]{BHV}), it has a fixed point $f$. Then, since $f$ is $G$-invariant, the subset $\{x\in X:\;f(x)\ge 1/2\}$ is $G$-invariant; moreover since $f\in\ell^2_M(X)$, this subset is commensurate to $M$.
\end{proof}

This provides an analogue of Corollary \ref{ficd} for {\em bounded} cardinal definite functions, relaxing the cofinality hypothesis.

\begin{cor}\label{bdf}Let $G$ be a topological group. Then every bounded cardinal definite function on $G$ is a finite sum of (bounded) cardinal definite functions associated to transitive actions of $G$, and cannot be written as an infinite sum of nonzero cardinal definite functions.
\end{cor}
\begin{proof}
If $\ell=\ell_M$ is cardinal definite and bounded, then $M$ is transfixed by Theorem \ref{btx}, i.e.\ is commensurate to a $G$-invariant subset $N$. The finite subset $M\tu N$ meets finitely many $G$-orbits; if $V$ is any other orbit, then $M\cap V$ is invariant. The result follows.
\end{proof}

\begin{rem}Brailovsky, Pasechnik and Praeger \cite{BPP} proved that if $\sup_G\ell_M<\infty$ then $G$ preserves a subset $N$ commensurate to $M$ (with an explicit but non-optimal bound on $\#(N\tu M)$. An almost optimal result was subsequently provided by P. Neumann \cite{Ne}: if $\sup_G\#(gM\smallsetminus M)=m<\infty$, then there exists $N$ $G$-invariant with $\#(N\tu M)\le\max(0,2m-1)$.

This can be restated with only symmetric differences. First note that because of the existence of $N$, $gM\smallsetminus M$ and $M\smallsetminus gM$ have the same cardinality for all $g$. Therefore Neumann's result can be restated as: if $\sup_G\ell_M=m<\infty$, then $m$ is even and there exists a $G$-invariant subset $N$ of $X$, commensurate to $M$ with $\#(N\tu M)\le\max(0,m-1)$. 
\end{rem}

\begin{rem}
Let $s(m)$ be the optimal bound in the above result, so that Neumann's result can be stated as: $s(m)\le m-1$ for all $m\ge 1$ and $s(m)=s(m-1)$ for odd $m$, so we can focus on $s(m)$ for even $m$.

The inequality $s(m)\le m-1$ is maybe an equality for all even $m$. This holds when $m=2^d\ge 1$ is a power of 2, taking $X$ to be the projective space $\mathbf{P}^d(\mathbf{F}_2)$ and $M$ a hyperplane, so $\#(M)=2^d-1$, $\#(X\smallsetminus M)=2^d$; if $G$ is any subgroup of $\GL_{d+1}(\mathbf{F}_2)$ transitive on $X$, then 
$\sup_{g\in G}\ell_M(g)=2^d$ and the only $G$-invariant subsets $N$ are $\emptyset$ and the whole projective space, so the one minimizing $\#(M\tu N)$ is $N=\emptyset$, which satisfies $\#(M\tu N)=\#(M)=2^d-1$.

In general, write its binary representation as $m=\sum_{j\in J}2^j$ (since $m$ is even, $J$ is a finite subset of the positive integers), define $X=\bigsqcup_{j\in J}\mathbf{P}^j(\mathbf{F}_2)$ and $M=\bigsqcup H_j$, where $H_j$ is a hyperplane in $\mathbf{P}^j(\mathbf{F}_2)$. Then $\#(M)=m-\#(J)$, $\sup_{g\in G}\ell_M(g)=m$, and the $G$-invariant subset $N$ minimizing $M\tu N$ is $N=\emptyset$. So if $j_m$ is the number of digits 1 in the binary writing of $m$ (so $j_m\le\log_2(m)+1$), then we have $m-j_m\le s(m)\le m-1$ for all even $m\ge 2$ (thus Neumann's upper bound is ``asymptotically optimal").

The left bound $m-j_m$ is not always sharp. For $m=6$ (where $m-j_m=4$), consider the transitive action of $G=X=\Z/10\Z$ on itself. Consider the subset $M=\{0,1,2,5,7\}$. Then a direct verification shows that $\#(M\cap (g{+}M))\ge 2$ for all $g$, so $\#(M\tu (g{+}M))\le 6$ for all $g$; thus $s(6)=5$. In general, I do not know if for odd $n$, $\Z/2n\Z$ always contains an $n$-element subset $M$ such that $\#(M\cap (q{+}M))\ge (n-1)/2$ for all $q\in\Z/2n\Z$. (For $n=9$ the subset $\{0,1,2,4,5,9,11,15,17\}$ works, thus $s(10)=9$.)
\end{rem}


\begin{prop}\label{elggn}
Let $G$ be a topological group and $N$ a normal subgroup. Let $\ell$ be a cardinal definite function on $G$ whose restriction to $N$ is bounded. Then there exists a cardinal definite function $\bar{\ell}$ on $G/N$ such that, denoting by $\pi:G\to G/N$ the natural projection, the function $\ell-\bar{\ell}\circ\pi$ is bounded.
\end{prop}

\begin{lem}\label{cargn}
Let $G$ be a topological group and $N$ a normal subgroup. Let $\ell$ be a cardinal definite function on $G$ vanishing on $N$ (and hence $N$-invariant). Then the resulting cardinal definite function $\bar{\ell}$ on $G/N$ is cardinal definite.
\end{lem}
\begin{proof}
Note that the function $\bar{\ell}$ on $G/N$ is continuous by definition of the quotient topology.

Let $X$ be a continuous discrete $G$-set and $M\subset X$ a $G$-commensurated subset with open stabilizer such that $\ell_M=\ell$. We begin with the case when $X$ is $G$-transitive. Let $X'$ be the quotient of $X$ by the $N$-action and $M'$ the image of $M$ in $X'$. Then all fibers of $X\to X'$ have the same cardinal $\alpha$, and the inverse image of $gM'\tu M'$ is $gM\tu M$. In follows that either $M$ is $G$-invariant (in which case $\ell=0$ and there is nothing to prove), or that $\alpha$ is finite. Then $\ell_M=\alpha\ell_{M'}$. Since the action on $X'$ factors through a continuous action of $G/N$, we see that $\ell_{M'}$ is cardinal definite on $G/N$, and hence $\bar{\ell}=\alpha\ell_{M'}$ is cardinal definite on $G/N$ as well (by Lemma \ref{cardefsu}).

In general, assume $X$ arbitrary. Decompose $X$ into $G$-orbits as $X=\bigsqcup X_i$, yielding a decomposition $\ell=\sum\ell_i$. Since $M\cap X_i$ is $N$-invariant, $\ell_i$ is $N$-invariant and hence factors, by the transitive case, through a cardinal definite function $\bar{\ell}_i$ on $G/N$. Then since $\bar{\ell}=\sum\bar{\ell}_i$ is finite and continuous, it is cardinal definite by Lemma \ref{cardefsu}.
\end{proof}

\begin{proof}[Proof of Proposition \ref{elggn}]
Let $X$ be a continuous discrete $G$-set and $M\subset X$ a $G$-commensurated subset with an open stabilizer such that $\ell_M=\ell$. 
 By Theorem \ref{btx}, $M$ is commensurate to an $N$-invariant subset $M'$. Then $\ell_{M'}$ factors through a cardinal definite function on $G/N$ by Lemma \ref{cargn}, proving the proposition.
 \end{proof}

\subsection{Induction}

Let $G$ be a group and $H$ a subgroup. Let $X$ be an $H$-set. Endow $G\times X$ with left and right commuting actions of $G$ and $H$ by 
\[ g (g_0,x_0) h=(gg_0h,h^{-1}x),\]
and define the {\em additive\footnote{There is also a {\em multiplicative} induced action \[\textnormal{Ind}_H^G(X)=\{\xi:G\to X\mid \forall h\in H,g\in G,\; \xi(gh)=h^{-1}\cdot\xi(g)\},\] where $G$ acts by $g\cdot\xi(x)=\xi(g^{-1}x)$, which is notably used in representation theory; we do not consider it here.} induced action} 
\[\ind_H^G(X)=(G\times X)/H,\]
which naturally inherits from the structure of a left $G$-set. For instance, for every subgroup $L$ of $H$, we have a natural identification $\ind_H^G(H/L)=G/L$. 

Denote by $\pi$ the projection $G\times X\to (G\times X)/H$. Note that if $F$ is a right transversal (so that $G$ is set-wise the product $FH$) then $\pi$ restricts to a bijection from $F\times X$ to $(G\times X)/H=\ind_H^G(X)$. Also, note that $\pi$ is injective in restriction to $\{1\}\times X$, giving rise an $H$-equivariant embedding of $X$ into $\ind_H^G(X)$.

\begin{lem}
Suppose that $G$ is a topological group, $H$ is open in $G$ and that $X$ is a continuous discrete $H$-set. Then $\ind_H^G(X)$ is a continuous discrete $G$-set.
\end{lem}
\begin{proof}
We need to show that the stabilizer in $G$ of $\pi(g_0,x_0)$ is open for every $(g_0,x_0)\in G\times X$. An element $g\in G$ belongs to this stabilizer if and only there exists $h\in H$ such that $g(g_0,x_0)=(g_0,x_0)h$, that is, $h^{-1}x_0=x_0$ and $g=g_0hg_0^{-1}$. Thus the stabilizer of $\pi(g_0,x_0)$ is equal to $g_0H_{x_0}g_0^{-1}$, which is open in $H$ and hence in $G$. 
\end{proof}

\begin{prop}\label{fi}
Assume that $H$ has finite index in $G$. Suppose that $M$ is an $H$-commensurated subset of $X$ and $F$ is a right transversal of $G$ modulo $H$ (so $G=FH$), with $1\in F$. Identify $M$ to $\pi(\{1\}\times M)$ and define $M'=\bigcup_{f\in F}fM$. Then $M'$ is commensurated by the $G$-action and $M'\cap X=M$. In particular, the restriction of $\ell_{M'}$ to $H$ is $\ge\ell_M$.

If $G$ is a topological group, $H$ is open and $M$ has an open stabilizer in $H$, then $M'$ has an open stabilizer in $G$.
\end{prop}
\begin{proof}
Fix $g\in G$. For every $f\in F$, there exists a unique $f'\in F$ such that $gf\in f'H$. Write $gf=f'h_f$. Then $gfM=f'h_fM\subset f'M\cup (M\tu h_fM)$, where $M\tu h_fM$ is finite. Thus, using that $F$ is finite, $gM'\smallsetminus M'$ is finite for all $g\in G$. Since $G$ is a group, it follows that $M'\tu gM'$ is finite for all $g$.

Since $1\in F$, we have $M'=M\cup\bigcup_{f\in F\smallsetminus\{1\}}fM$, while if $f\notin H$ we have $fM\cap X=\emptyset$. Hence $M'\cap X=M$.

If $L$ is the stabilizer of $M$ in $H$, then $L$ is open by assumption. We then see that $M'$ is stabilized by $\bigcap_{f\in F}fLf^{-1}$, which is open. 
\end{proof}

\begin{cor}\label{fico}
Let $G$ be a topological group and $H$ an open subgroup of finite index. Then for every cardinal definite function $\ell$ on $H$, there exists a cardinal definite function $\ell'$ on $G$ such that $\ell'|_H\ge\ell$.
\end{cor}

\subsection{Wreath products}\label{s_wp}

If $H$ is a discrete group, $G$ is a topological group, $Y$ is a continuous discrete $G$-set, the wreath product $H\wr_Y G$ is by definition the semidirect product $H^{(Y)}\rtimes G$, where $G$ acts by shifting the direct sum (or restricted direct product) $H^{(Y)}=\bigoplus_{y\in Y}H$. Since the action of $G$ on the discrete group $H^{(Y)}$ is continuous, this semidirect product is a topological group with the product topology.

There is a simple way to define a commensurating action of $H\wr_Y G$ out of a commensurating action of $H$. Let $H$ act on a set $X$, commensurating a subset $M$. Then $H\wr_Y G$ acts on $X\times Y$, where the action of the $y$-th summand $H^{(y)}\simeq H$ in $H^{(Y)}$ is given by the standard $H$-action on $X\times\{y\}$ and is the trivial action on $X\times (Y\smallsetminus\{y\})$, and the action of $G$ permutes the components. Note that this action is continuous, the stabilizer of a point $(x,y)$ being the open subgroup $H^{(Y\smallsetminus\{y\})}H^{(y)}_xG_y$.
This action commensurates $M\times Y$, which has an open stabilizer (as it contains the open subgroup $G$), and the length is given by 
\[\ell_{M\times Y}((h_y)_{y\in Y}g)=\sum_{y\in Y}\ell_M(h_y).\]
Interestingly, this length is usually unbounded even if $\ell_M$ is bounded. For record:

\begin{prop}\label{wcd}
For every cardinal definite function $\ell$ on $H$, the function $(h_y)_{y\in Y}g\mapsto\sum_{y\in Y}\ell_M(h_y)$ is cardinal definite on $H\wr_YG$. In particular, the function $fg\mapsto2\#(\Supp(f))$ (where $f\in H^{(Y)}$ and $g\in G$) is cardinal definite on $H\wr_YG$.
\end{prop}
\begin{proof}
The first statement has already been proved. The second statement is the particular case where $\ell=2\mathbf{1}_{H\smallsetminus\{1\}}$; it is indeed cardinal definite, associated to the left action of $H$ on itself and the commensurated subset $M=\{1\}$.
\end{proof}

We now proceed to describe another more elaborate construction due to the author, Stalder and Valette \cite{CSVa}. The construction is described in \cite{CSVa} in terms of wallings so we need to translate it into commensurating actions. We here deal with a standard wreath product $H\wr G$ (i.e., $G$ is discrete and $Y$ is $G$ with the left action by translation).

Start from a $G$-set $X$ with a commensurated subset $M$. For $x\in X$, define $W_x=\{h\in G\mid x\in hM\}$. Let $Z_X$ be the set of pairs $(x,\mu)$, where $x\in X$ and $\mu$ is a finitely supported function from the complement $W_x^*=G\smallsetminus W_x$ to $H$. Let $H\wr G$ act on $Z_X$ as follows: 
\[g\cdot (x,\mu)=(gx,g\cdot\mu);\quad \lambda\cdot (x,\mu)=(x,\lambda|_{W_x^*}\mu),\qquad g\in G,\lambda\in H^{(G)}\]
where $g\cdot\mu(\gamma)=\mu(g^{-1}\gamma)$.
Define $N=M\times\{1\}\subset Z_X$.

\begin{prop}\label{zxest}
The subset $N$ of $Z_X$ is commensurated by the $G$-action and we have the following two lower bounds, for $g\in G$ and $w\in H^{(G)}$
\[\ell_N(wg)\ge\#(M\tu gM);\quad \ell_N(wg)\ge\sup_{\gamma\in\Supp(w)}\#(M\smallsetminus \gamma M)\]
and the upper bound
\[\ell_N(wg)\le\#(M\tu gM)+\sum_{\gamma\in\Supp(w)}\#(M\smallsetminus \gamma M)+\sum_{\gamma\in g^{-1}\Supp(w)}\#(M\smallsetminus \gamma M).\]
\end{prop}
\begin{proof}
This actually follows from Proposition \ref{can_bij} and the verifications in \cite{CSVa}, but it is instructive to provide a direct proof.

For $w\in H^{(G)}$ and $g\in G$, let us describe $N\smallsetminus wgN$. Elements $(x,\mu)$ in this set satisfy $x\in M$, $\mu=1$, and also $g^{-1}w^{-1}(x,1)\notin N$. The latter condition means that either $g^{-1}x\notin M$ or $w|_{W_x^*}\neq 1$. Note that $w|_{W_x^*}\neq 1$ means $\Supp(w)\nsubseteq W_x$, which in turn means $\Supp(w)^{-1}x\nsubseteq M$. In other words, $N\smallsetminus wgN=A\cup B$ where 
\[A=\{(x,1):\; x\in M\smallsetminus gM\};\quad B=\{(x,1):\; x\in M,\;\Supp(w)^{-1}x\nsubseteq M\};\]
similarly $wgN\smallsetminus N=C\cup D$ where 
\[C=\{(x,w|_{W_x^*}):\; x\in gM\smallsetminus M\};\quad D=\{(x,w|_{W_x^*}):\;x\in gM,\;\Supp(w)^{-1}x\nsubseteq M\}\]

Note that $\#(A\cup C)=\#(M\tu gM)$ and thus $\ell_N(wg)\ge \#(M\tu gM)$. On the other hand, we have $B=\bigcup_{\gamma\in\Supp(w)}M\smallsetminus\gamma M$ and $D=\bigcup_{\gamma\in\Supp(w)}gM\smallsetminus\gamma M$.
Since $\#(B)\le\ell_N(wg)\le \#(A\sqcup C)+\#(B)+\#(D)$, this gives the second lower bound and the upper bound.
\end{proof}

Let us also observe that if $X$ is $G$-transitive, then $Z_X$ is $(H\wr G)$-transitive, and if $L\subset G$ is the stabilizer of $x_0\in X$, then the stabilizer of $(x_0,1)\in Z_X$ is $H^{(A_{x_0})}L$.

\subsection{The transfer character}

Let $X$ be a set and $M$ a subset. We define a map, which by anticipation of Proposition \ref{transfer} we call {\em transfer character}.

\begin{eqnarray*}
\tr_M: \SX(X,M) & \to & \Z\\
g & \mapsto & \#(g^{-1}M\smallsetminus M)-\#(M\smallsetminus g^{-1}M)\\
& & =\sum_{x\in X} \mathbf{1}_{g^{-1}M}(x)-\mathbf{1}_M(x)
\end{eqnarray*}

Denote by $\SX_0(X)$ the group of finitely supported permutations of $X$, and $\SX_0^+(X)$ its subgroup of index of alternating permutations (which has index 2 in $\SX_0(X)$ unless $X$ is empty or a singleton). 

\begin{prop}\label{transfer}
The function $\tr_M$ is a continuous homomorphism from $\SX(X,M)$ to $\Z$, and is bounded above by $\ell_M$. It is surjective, unless $M$ or $M^c$ is finite (in which case it is zero). It does not depend on the choice of $M$ within its commensurability class, and $\tr_{M^c}=-\tr_M$. If $X$ is infinite, its kernel $\SX^\circ(X,M)$ is a perfect group, and is generated by $\SX(M)\cup\SX(M^c)\cup\SX_0^+(X)$.
\end{prop}
\begin{proof}
The upper bound is obvious.

Let us check that $\tr_M=\tr_N$ when $M,N$ are commensurate; it is enough to prove it when $M=N\sqcup F$ with $F$ finite.
We write
\begin{align*}
\tr_M(g)- \tr_N(g)= & \left(\sum_{x\in X} \mathbf{1}_{g^{-1}M}(x)-\mathbf{1}_M(x)\right)-\left(\sum_{x\in X}\mathbf{1}_{g^{-1}N}(x)-\mathbf{1}_N(x)\right)\\
 = & \sum_{x\in X} \mathbf{1}_{g^{-1}F}(x)-\mathbf{1}_F(x)=\tr_F(g)=0.
\end{align*}
Now we have 
\begin{align*}
\tr_M(gh)= & 
\sum_{x\in X} \mathbf{1}_{(gh)^{-1}M}(x)-\mathbf{1}_M(x)\\
= & \sum_{x\in X} \mathbf{1}_{(gh)^{-1}M}(x)-\mathbf{1}_{h^{-1}M}(x) + \sum_{x\in X} \mathbf{1}_{h^{-1}M}(x)-\mathbf{1}_M(x)\\
= & \sum_{x\in X} \mathbf{1}_{g^{-1}M}(hx)-\mathbf{1}_{M}(hx) + \tr_M(h)=\tr_M(g)+\tr_M(h)\end{align*}

The homomorphism is continuous because its kernel contains the stabilizer of $M$, which is open by definition of the topology of $\SX(X,M)$.

Now let $g$ belong to the kernel of $\tr_M$. Then the finite sets $g^{-1}M\smallsetminus M$ and $M\smallsetminus g^{-1}M$ have the same cardinal, hence there exists a finitely supported permutation $s$ exchanging these two finite subsets and being the identity on the complement; also let $\tau$ be either the identity when $s$ is even, or a transposition supported by $M$ or $M^c$ when $s$ is odd. Then $\tau s$ is an even permutation, and $\tau s g$ stabilizes $M$. 

This shows the generation statement. If both $M,M^c$ are infinite, all of $\SX(M)$, $\SX(M^c)$ and $\SX_0^+(X)$ are perfect groups and hence it follows that the kernel of $\tr_M$ is a perfect group. If one (and only one) of $M$ and $M^c$ is finite, then this kernel is just $\SX(X)$, which is perfect.
\end{proof}

\begin{cor}
If $X$ is a topological group and $X$ is a continuous discrete $G$-set and $M$ a commensurate subset with open stabilizer, then the above transfer map is a continuous homomorphism, bounded above by $\ell_M$. \qed
\end{cor}

Of course the transfer map $\tr_M$ can be bounded on $G$ even when the action is not transfixing, for instance when $G$ admits no continuous homomorphism onto~$\Z$.

It is possible to classify normal subgroups of $\SX(X,M)$. For the sake of simplicity, let us focus on the countable case. Recall that $\SX(\Z)$ has exactly 4 normal subgroups: $\{1\}$, $\SX(\Z)$, $\SX_0(\Z)$ and $\SX_0^+(\Z)$.

Define $\SX_M(X)$ the group of permutations of $X$ that are identity on a cofinite subset of $M$.

\begin{prop}
The normal subgroups of $\SX(\Z,\N)$ are the following:
\begin{itemize}
\item $\{1\}$, $\SX_0(\Z)$, $\SX_0^+(\Z)$;
\item $\SX_\N(\Z)$, $\SX_{-\N}(\Z)$
\item $\SX^\circ(\Z,\N)$ and the subgroups properly containing it (which have finite index and are indexed by positive integers, since $\SX(\Z,\N)/\SX^\circ(\Z,\N)\simeq\Z$).
\end{itemize}
In particular, the only closed normal subgroups of $\SX(\Z,\N)$ are $\{1\}$ and the subgroups containing $\Ker(\tr_\N)$ (which are open).
\end{prop}
\begin{proof}
Let $N$ be a normal subgroup. A standard argument, left to the reader, shows that if $N$ is not contained in $\SX_0(\Z)$, then it contains $\SX_0(\Z)$, which we now assume.


Define $N'=N\cap\SX^\circ(\Z,\N)$. The argument of the proof of Proposition \ref{transfer} shows that $N$ is generated by $\SX_0(\Z)$ and the stabilizer of $\N$ in $N$, which can be viewed as a normal subgroup of $\SX(\N)\times\SX(-\N)$ containing finitely supported permutations, and therefore, by simplicity of $\SX(\N)/\SX_0(\N)$ is one of the four possibilities: $\SX_0(\N)\times\SX_0(-\N)$, $\SX_0(\N)\times\SX(-\N)$, $\SX(\N)\times\SX_0(-\N)$, $\SX(\N)\times\SX(-\N)$. Accordingly, $N'$ is equal to $\SX_0(\Z)$, $\SX_\N(\Z)$, $\SX_{-\N}(\Z)$, or $\SX^\circ(\Z,\N)$. 

We claim that if $\tr_\N$ is nonzero on $N$, then $N'=\SX^\circ(\Z,\N)$.

Granting the claim, we deduce that either $N'=N$, in which case $N$ is equal to one of the above four subgroups, or $N$ contains $\SX^\circ(\Z,\N)$, concluding the proof. 

To check the claim, we see that if $\tr_\N(g)\neq 0$, then $g$ or $g^{-1}$ has at least one infinite orbit starting in $M^c$ and ending in $M$. Composi

A suitable commutator then provides an element in $N$ with infinite support contained in $M$, and another one in $M^c$, so that $N'=\SX^\circ(\Z,\N)$.
\end{proof}

When $X$ is uncountable the description takes a little more effort; still we have:

\begin{prop}
If $X$ is infinite and $M$ a subset, the closure of $\SX_0^+(X)$ in $\SX(X,M)$ is equal to $\SX^\circ(X,M)$, and is a topologically simple topological group.
\end{prop}
\begin{proof}
If $g\in\SX^\circ(X,M)$, then there exists $s\in\SX_0^+(X)$ such that $sg(M)=M$. Further, given a finite subset $F$ of $X$ there exists $s'\in\SX_0^+(X)$, stabilizing $M$, such that $s'sg$ is the identity on $F$. This shows that $\SX_0(X)$ is dense in $\SX^\circ(X,M)$. Since the latter is a closed subgroup, we deduce that the closure of $\SX_0(X)$ is $\SX^\circ(X,M)$. 

Since, by a simple argument, any normal subgroup of $\SX^\circ(X,M)$ not contained in $\SX_0(X)$ contains the dense subgroup $\SX_0(X)$, it follows that $\SX^\circ(X,M)$ is topologically simple.
\end{proof}




Let us turn back to the definition of the transfer character (the following remarks follow discussions with Pierre-Emmanuel Caprace). Define, on $\SX(X,M)$, the gain map $s_M(g)=\#(g^{-1}M\smallsetminus M)$. Note that by definition, we have $\tr_M(g)=s_M(g)-s_M(g^{-1})$ and $\ell_M(g)=s_M(g)+s_M(g^{-1})$. We have the following immediate properties:

\begin{prop}
The gain map $s_M:\SX(X,M)\to\N$ satisfies:
\begin{itemize}
\item $s_M$ is bi-invariant by the stabilizer of $M$ and in particular is continuous;
\item $s_M$ is sub-additive: $s_M(gh)\le s_M(g)+s_M(h)$ for all $g,h\in\SX(X,M)$;
\item is $M,N$ are commensurate, then $|s_M-s_N|\le \#(M\tu N)$;
\item in restriction to $\SX^\circ(X,M)$, we have $s_M=\frac12\ell_M$.
\end{itemize}
\end{prop}
\begin{proof}
Define $S_M(g)=g^{-1}M\smallsetminus M$, so that $s_M=\#\circ S_M$.

Then for all $g,h\in\SX(X,M)$ such that $hM=M$, we have $S_M(hg)=S_M(g)$ and $S_M(gh)=h^{-1}S_M(g)$; in particular, $s_M(hg)=s_M(gh)=s_M(g)$.

Also for all $g,h\in\SX(X,M)$, we have $S_M(gh)\subset S_M(h)\cup h^{-1}S_M(g)$, which implies the sub-additivity.

The third property follows from the particular case when $\#(M\tu N)=1$, which is checked by hand, and the last assertion is trivial.
\end{proof}

A consequence of the sub-additivity is that $s_M(g^n)/n$ converges when $n\to +\infty$, to a number $\sigma_M(g)$; note that $\sigma_M(g)-\sigma_M(g^{-1})=\tr_M(g)$. An argument similar to that of Proposition \ref{decab} actually shows that this limit is an integer (namely, the number of infinite $\langle g\rangle$-orbits that start in $M^c$ and end up in $M$). Unlike $s_M$, the function $\sigma_M$ is invariant under conjugation: $\sigma_M(hgh^{-1})=\sigma_M(g)$ for all $g,h$.

\begin{rem}
Unlike the transfer character, the function $\sigma_M$, which is upper semi-continuous as an infimum of continuous functions (namely $g\mapsto s_M(g^n)/n$) fails to be continuous on $\SX(X,M)$ when $M$ and $M^c$ are both infinite: indeed, take $X=\Z\times\{\pm 1\}$ and $M=\N\times\{\pm 1\}$. Define $f,f_n\in\SX(X,M)$ as follows: $f(m,\eps)=m+\eps$; $f_n(m,1)=m+1$ if $-n\le m<n$, $f_n(n,1)=(n,-1)$, $f_n(m,1)=(m,1)$ if $|m|>n$, and $f_n(m,-1)=-f_n(-m,1)$. Then $f_n(M)=f(M)$ for all $n\ge 2$ and $(f_n)$ converges pointwise to $f$. Therefore $(f_n)$ tends to $f$ in $\SX(X,M)$. On the other hand, $\sigma_M(f_n)=0$ because $f_n$ has finite order, while $\sigma_M(f)=1$. Note that an alternative asymptotic definition of $\tr_M$ is to define it as $\sigma_M(g)-\sigma_M(g^{-1})$. 

Beware that $\sigma_M$ is not sub-additive in general; still we have $\sigma_M(gh)\le \sigma_M(g)+\sigma_M(h)$ when $g,h$ commute. \end{rem}

\section{Property FW etc.}\label{fwetc}
\subsection{Property FW}

\begin{defn}\label{d_fw}
Let $G$ be a topological group. We say that $G$ has {\em Property FW} if for every continuous discrete $G$-set, any commensurated subset with open stabilizer is transfixed (i.e., is commensurate to an invariant subset).
\end{defn}

By Theorem \ref{btx}, this amounts to saying that every cardinal definite (Definition \ref{d_cd}) function on $G$ is bounded. This allows the following generalization: if $L\subset G$, we say that $(G,L)$ has {\em relative Property FW} if every cardinal definition function on $G$ is bounded on $L$. In case $L$ is a subgroup, Theorem \ref{btx} shows that this means that for every continuous discrete $G$-set $X$ and commensurated subset $M$ with open stabilizer, $M$ is transfixed in restriction to $L$.

If in Definition \ref{d_fw} we restrict to transitive actions, we get the following a priori weaker notion.

\begin{defn}
Let $G$ be a topological group. We say that $G$ has {\em Property FW'} if for every continuous discrete transitive $G$-set, any commensurated subset is either finite or cofinite.
\end{defn}

The negation of Property FW' is also known as ``semisplittable".

We use the notion of topological groups with uncountable cofinality (or cofinality $\neq\omega$) from \S\ref{s_cof}; important examples of such groups are compactly generated groups and in particular finitely generated discrete groups. 

\begin{prop}\label{fwfwp}
Let $G$ be a topological group.
\begin{enumerate}
\item\label{fwww} If $G$ has uncountable cofinality, then $G$ has Property FW if and only it has Property FW';
\item\label{fwwc} if $G$ has countable cofinality, then $G$ does not have Property FW.
\end{enumerate}
\end{prop}
\begin{proof}The first part immediately follows from Proposition \ref{cgco2}. For the second, if $G=\bigcup G_n$ with $(G_n)$ an nondecreasing union of proper open subgroups, then if $T=\bigsqcup G/G_n$ is endowed with the natural $G$-action and $x_n$ is its base-point, then $\{g_n\mid n\ge 0\}$ is commensurated but not transfixed.
\end{proof}

\begin{rem}
If $G$ has countable cofinality (e.g.\ is discrete, infinitely generated and countable), it does not have Property FW by Proposition \ref{fwfwp}(\ref{fwwc}), while $G$ may have either Property FW' or not. For instance, no infinite countable locally finite group has Property FW' (by a result of D.\ Cohen \cite{coh}, using the action of $G$ on itself), while the infinitely generated group $\SL_n(\Q)$ has Property FW' for all $n\ge 3$ \cite{Cor2}.
\end{rem}

\subsection{Features of Property FW}

\begin{prop}\label{fwfi}
Let $G$ be a topological group and $H$ an open subgroup of finite index. Then $G$ has Property FW if and only if $H$ has Property FW.
\end{prop}
\begin{proof}
We begin with the easier implication. If $H$ has Property FW and $\ell$ is a cardinal definite function on $G$, write $G=FH$ with $F$ finite; if $m$ is an  upper bound for $\ell$ on $F\cup H$ then $2m$ is an upper bound for $\ell$.

Conversely suppose that $G$ has Property FW and let $\ell$ be a cardinal definite function on $H$. Then by Corollary \ref{fico} (which uses additive induction of actions), there exists a cardinal definite function $\ell'$ on $G$ such that $\ell'|_H\ge\ell$. By Property FW, $\ell'$ is bounded and hence $\ell$ is bounded. 
\end{proof}

The simplest example of a group without Property FW is $\Z$, using that the left action on itself commensurates $\N$; combined with Proposition \ref{fwfi} this yields.

\begin{cor}\label{vb1}
For every finitely generated group $\Gamma$ with Property FW, every finite index subgroup of $\Gamma$ has a finite abelianization. More generally, if $G$ is a totally disconnected locally compact group, then for every open subgroup of finite index $H$, the quotient $H/\overline{[H,H]}$ is compact.
\end{cor}
\begin{proof}
By Proposition \ref{fwfi}, we are reduced to check that any totally disconnected locally compact abelian group with Property FW is compact. Indeed, modding out by a compact open subgroup, we are reduced to the discrete case. Now since any infinite discrete abelian group has an infinite countable quotient, we are reduced to a infinite countable discrete abelian group $D$. By Proposition \ref{fwfwp}(\ref{fwwc}), Property FW for $D$ implies that $D$ is finitely generated, and hence admits $\Z$ as a quotient, a contradiction.
\end{proof}

Property FW is obviously stable under taking quotients. The following proposition shows it is also stable by taking extensions.

\begin{prop}\label{extfw}
Let $G$ be a topological group and $N$ a normal subgroup. Suppose that $(G,N)$ has relative Property FW and $G/N$ has Property FW. Then $G$ has Property FW.
\end{prop}
\begin{proof}
Let $\ell$ be a cardinal definite function on $G$. Then by relative Property FW, $\ell$ is bounded on $N$. Hence by Proposition \ref{elggn}, there exists a function $\ell'$ on $G$ factoring through a cardinal definite function on $G/N$ such that $\ell-\ell'$ is bounded. By Property FW for $G/N$, $\ell'$ is bounded. So $\ell$ is bounded as well.
\end{proof}

The next result shows that except some degenerate cases, wreath products (defined in \S\ref{s_wp}) never have Property FW.

\begin{prop}\label{wnfw}
Let $G$ be a topological group, $Y$ an infinite continuous discrete $G$-set and $H$ a nontrivial discrete group. Then the wreath product $H\wr_Y G$ does not have Property FW. In particular, if $H,G$ are discrete groups with $H$ nontrivial and $G$ infinite then $H\wr G$ does not have Property FW.
\end{prop}
\begin{proof}
By Proposition \ref{wcd}, the function $wg\mapsto 2\#(\Supp(w))$ is cardinal definite on $H\wr_YG$; it is unbounded as soon as $Y$ is infinite. (The proof of Proposition \ref{wcd} also shows that if $Y$ has an infinite $G$-orbit then $H\wr_YG$ does not have Property FW'.
\end{proof}

Proposition \ref{wnfw} was obtained for standard wreath products with a similar argument in \cite[Theorem~3]{CMV}, although claiming a weaker statement. 

\begin{rem}
A more careful look at the proof of Proposition \ref{wnfw} (see the stabilizer computations before Proposition \ref{wcd}) shows that if $H$ is a nontrivial finitely generated discrete group, $G$ is a compactly generated locally compact group and $y\in Y$ has an infinite $G$-orbit, then $H^{(Y\smallsetminus\{y\})}G_y$ is coforked in $H\wr_Y G$. In particular, if $G$ is an infinite finitely generated discrete group, then $H^{(G\smallsetminus\{1\})}$ is coforked in the standard wreath product $H\wr G$.
\end{rem}

We now recall some geometric rigidity properties
\begin{itemize}
\item A topological group is {\em strongly bounded (SB)} if every continuous isometric action on a metric space has bounded orbits, or equivalently if every continuous subadditive nonnegative real-valued function is bounded (strong boundedness is sometimes called Bergman Property, strong Bergman Property, strong uncountable cofinality).
\item A topological group has {\em Property FH} if for every continuous isometric action on a Hilbert space there is a fixed point, or equivalently (by the center lemma) orbits are bounded. For $\sigma$-compact locally compact groups, the Delorme-Guichardet Theorem \cite[\S 2.12]{BHV} states that Property FH is equivalent to Kazhdan's Property T (defined in terms of unitary representations).
\item A topological group has {\em Property FA} if for every continuous isometric action on the 1-skeleton of a tree, there is a fixed point, or equivalently (by the center lemma) orbits are bounded. By Bass-Serre Theory \cite{Ser}, a topological group has Property FA if and only if it satisfies the following three conditions 
\begin{itemize}
\item it has no continuous homomorphism onto $\Z$;
\item it has no decomposition as a nontrivial amalgam over open subgroups;
\item it has uncountable cofinality (as a topological group).
\end{itemize}
\item Cofinality $\neq\omega$ (see \S\ref{s_cof}) can also be characterized as: every continuous isometric action on any ultrametric space has bounded orbits.
\end{itemize}

\begin{prop}\label{fhfwfa}For a topological group $G$, we have the following implications
\[\textnormal{SB}\Rightarrow\textnormal{FH}\Rightarrow\textnormal{FW}\Rightarrow\textnormal{FA}\Rightarrow(\textnormal{cofinality}\neq\omega).\]
\end{prop}
\begin{proof}
The left implication is trivial. The last implication is due to Serre \cite[\S 6.1]{Ser}: let $(G_n)$ be a nondecreasing sequence of subgroups whose union is $G$) and define $T$ as the $G$-set given as the disjoint union $\bigsqcup G/G_n$; endow it with the graph structure joining any $g\in G/G_n$ to its image in $G/G_{n+1}$. Then $T$ is an unbounded tree on which $G$ acts transitively by automorphisms. So $G$ does not have Property FA.

The implication FH$\Rightarrow$FW is a consequence of Proposition \ref{sxlp2} for $p=2$ (it is essentially half of the original proof by Alperin \cite{Al} of the implication FH$\Rightarrow$FA). 

For the implication FW$\Rightarrow$FA, we essentially repeat the other half of the same result of Alperin: assume that $G$ has Property FW and let $G$ act continuously on a tree. Consider the action of $G$ on the set $X$ of oriented edges. This action is continuous, as the stabilizer of a given oriented edge is the pointwise stabilizer of a pair of vertices. Fix a vertex $x_0$ and let $M$ be the set of edges pointing towards $x_0$. Then the stabilizer of $M$ is open, since it contains the stabilizer of $x_0$. Moreover, $M$ is commensurated and $\ell_M(g)=d(x_0,gx_0)$ for all $g\in G$. By Property FW, $\ell_M$ is bounded. Thus the orbit of $x_0$ is bounded and hence there is a fixed point in the 1-skeleton.\end{proof}

\begin{rem}
The implications of Proposition \ref{fhfwfa} are not equivalences, even for countable discrete groups. Let us begin with the easiest:
\begin{itemize}
\item (uncountable cofinality) $\Rightarrow\!\!\!\!\!\!\!/\;$FA: $\Z$ is a counterexample;
\item FH$\Rightarrow\!\!\!\!\!\!\!/\;$SB: consider any infinite discrete (finitely generated) group with Property T, e.g.\ $\SL_3(\Z)$.
\item That FA$\Rightarrow\!\!\!\!\!\!\!/\;$FW is now well-known; see Example \ref{FApFW}.
\item FW$\Rightarrow\!\!\!\!\!\!\!/\;$FH: a counterexample (from \cite{Cor2}) is $\SL_2(\Z[1/2])$, see Example \ref{sl2z2}
\end{itemize}
\end{rem}


\begin{exe}\label{FApFW}
Let us provide two types of finitely generated groups with Property FA but not Property FW (the first examples appeared in \cite[Example 2.5]{Sco77}).
\begin{enumerate}
\item Finitely generated groups with Property FA but without hereditary Property FA, i.e.\ with a finite index subgroup without Property FA. By Proposition \ref{fwfi}, such groups do not have Property FW. There are a lot of such groups. For instance, any infinite finitely generated Coxeter group whose Coxeter graph has no $\infty$-label has this property. This includes the examples in \cite[Example 2]{ChN} (which are lattices in the group of isometries of the Euclidean plane). Such groups (as all finitely generated Coxeter groups) actually have Property PW \cite{BJS}, and have Property FA, as any group generated by a finite set $S$ such that $S\cup S^2$ consists of torsion elements \cite[Corollaire~2, p.~90]{Ser}. Another elementary example is the following: let $D$ be a finitely generated group with finite abelianization with a finite index subgroup $D_1$ with infinite abelianization (e.g.\ the infinite dihedral group) and $F$ a nontrivial finite group. Then the standard wreath product $D\wr F=D^F\rtimes F$ has Property FA by the general criterion of the author and A.~Kar \cite[Theorem 1.1]{CoK}; besides it admits $D_1^F$ as a subgroup of finite index with infinite abelianization and thus failing to have Property FA. Also, if $E$ is a finitely generated group with finite abelianization but splits as a nontrivial amalgam, then the wreath product $E\wr F$ has Property FA again by \cite[Theorem 1.1]{CoK} but admits the finite index subgroup $E^F$, which fails to have Property FA. In case every finite index subgroup of $E$ has a finite abelianization (e.g.\ $E$ is the free product of two infinite finitely generated simple groups), $E\wr F$ is then an example of a finitely generated group with Property FA, without hereditary Property FA but all of whose finite index subgroups have a trivial abelianization.

\item Finitely generated groups with hereditary Property FA and without Property FW. For instance, if $B$ is any nontrivial finitely generated group with finite abelianization (e.g.\ a nontrivial finite group) and if $\Gamma$ is any infinite finitely generated group with hereditary Property FA (e.g. with Property FW), then by the criterion of \cite[Theorem 1.6]{CoK}, the standard wreath product $B\wr\Gamma$ has hereditary Property FA, while it does not have Property FW by Proposition \ref{wnfw}. Other examples with hereditary Property FA and without Property FW are Grigorchuk's groups \cite{Gri}, as well as topological full groups of minimal Cantor systems (see \S\ref{negfw}).
\end{enumerate}
\end{exe}

\begin{rem}\label{relaver}
Proposition \ref{fwfwp}(\ref{fwww}) can be extended to relative versions. Namely, if $L\subset G$ is any subset and $G$ has cofinality $\neq\omega$, then relative Property FW for $(G,L)$ can be tested on transitive actions of $G$.

Also, Proposition \ref{fhfwfa} extends to the relative case: just define relative Property FH, FA, cofinality $\neq\omega$ by saying that every continuous isometric action of $G$ on a Hilbert space (resp.\ tree, resp. ultrametric space) is bounded in restriction to $L$.

Concerning Proposition \ref{fwfi}, it is obvious that relative Property FW for $(G,L)$, when $L$ is a subgroup, does not change if $L$ varies in its group commensuration class. Also, its proof shows that, for any open subgroup $H$ of $G$ and subset $L$ of $H$, $(G,L)$ has relative Property FW if and only $(H,L)$ has Property FW.

Proposition \ref{extfw} also extends the following relative version: if $(G,N)$ has relative Property FW, if $L$ is an $N$-invariant subset of $G$ and $(G/N,L/N)$ has relative Property FW then $(G,L)$ has relative Property FW.
\end{rem}

\subsection{Property PW}

\begin{defn}
A locally compact group $G$ has Property PW if it has a proper cardinal definite function.
\end{defn}

Observe that if a locally compact group $G$ has both Properties PW and FW then it is compact. More generally, if it has Property PW and $(G,L)$ has relative Property FW for some subset $L$, then $\overline{L}$ is compact. Also observe that from the bare existence of a proper continuous real-valued function, every locally compact group with Property PW is $\sigma$-compact.

Observe that Property PW for locally compact groups is stable by taking closed subgroups.

\begin{prop}\label{pwfi}
Let $G$ be a locally compact group. Then

\begin{enumerate}
\item\label{pwfi1} if $H$ an open subgroup of finite index, then $G$ has Property PW if and only if $H$ has Property PW;
\item\label{pwfi3} if $G$ is a topological direct product $G_1\times G_2$ where $G_1,G_2$ are closed subgroups with Property PW, then $G$ has Property PW.
\item\label{pwfi2} if $W$ is a compact normal subgroup, then $G$ has Property PW if and only if $G/W$ has Property PW.
\end{enumerate}
\end{prop}

\begin{proof}
We begin with (\ref{pwfi1}). Leaving aside the trivial implication, assume that $H$ has Property PW. Let $\ell$ be a proper cardinal definite function on $H$. Using additive induction of $H$-sets (Corollary \ref{fico}), there exists a cardinal definite function $\ell'$ on $G$ with $\ell'|_H\ge\ell$. In particular, $\ell'|_H$ is proper. Now it is easily checked for an arbitrary length function that properness on an open finite index subgroup implies properness; thus $\ell'$ is proper and $G$ has Property PW.

The assertion (\ref{pwfi3}) is essentially immediate, by choosing commensurating actions of $G_i$ on sets $X_i$ and considering the action of $G_1\times G_2$ on the disjoint union $X_1\sqcup X_2$.

For (\ref{pwfi2}), the ``if" implication is trivial and the converse follows from the fact that compact groups have Property FW and Proposition \ref{elggn}.
\end{proof}

Let us now mention the following slight generalization of the main result from \cite{CSVa}. There, $H$ was assumed finite, but the easy trick to carry over arbitrary $H$ was used in \cite{CSV} in a similar context.

\begin{thm}\label{stpw}
Let $H,G$ be discrete groups with Property PW. Then the standard wreath product $H\wr G$ has Property PW.
\end{thm}
\begin{proof}
Start from an action of $G$ on a set $X$ with a commensurated subset $M$ such that $\ell_{M}$ is proper, and consider the action on $X'\times\{0,1\}$ with commensurated subset $M'=M\times\{0\}\cup M^*\times\{1\}$. Then perform, out of the latter action, the construction $Z_{X'}$ of Proposition \ref{zxest} to obtain a cardinal definite function $\ell$ on $H\wr G$ satisfying
\[\ell(wg)\ge \sup_{\gamma\in\{g\}\cup \Supp(w)}\ell_M(\gamma).\]
By properness of $\ell_M$, for every $n$ the subset $F_n=\ell_M^{-1}([0,n])$ of $G$ is finite. By the above inequality, for every $wg$ such that $\ell(wg)\le n$, we have $\{g\}\cup\Supp(w)\subset F_n$. (Note that in case $H$ is finite, this shows that $\ell$ is proper.)

Now let us also use that $H$ has Property PW, let $\ell_0$ be a proper cardinal definite function on $H$; by Proposition \ref{wcd}, the function $wg\mapsto\ell'(wg)$, defined as $\ell'(wg)=\sum_{\gamma\in G}\ell_0(h_\gamma)$, is cardinal definite. Define $F'_n=\ell_0^{-1}([0,n])$; this is a finite subset of $H$.

Defining $\ell_1=\ell+\ell'$, if $\ell_1(wg)\le n$, then $\ell(wg)\le n$ so $\{g\}\cup\Supp(w)\subset F_n$ by the above, and $\ell'(wg)\le n$, so $w_\gamma\in F'_n$ for all $\gamma\in G$. Thus $w\in (F'_n)^{F_n}$ and $g\in F_n$, which leaves finitely many possibilities for $wg$. Thus the cardinal definite function $\ell_1$ is proper. 
\end{proof}

\begin{rem}
There is an analogue of Property FW and its cousins for actions by group automorphisms. Namely, a group $G$ has {\em Property FG} if whenever it acts by automorphisms on a discrete group $H$ commensurating (in the group sense) a subgroup $P$ also commensurated by $H$, there exists a subgroup $P'$ commensurate to $P$ and invariant by $G$ (there is an obvious relative version). The existence of $P'$ can also be characterized by the boundedness of the length $L_P(g)=\log([P:P\cap g(P)][g(P):P\cap g(P)])$, by a result 
of Schlichting (rediscovered by Bergman and Lenstra \cite{BLe}). By a projective limit construction \cite{ShW}, these properties can also be characterized in terms of actions by topological automorphisms on totally disconnected locally compact group, namely by the condition that every such action preserves a compact open subgroup. Property FG implies Property FW (since any action on a set $X$ commensurating a subset $M$ induces an action on the group $C^{(X)}$ commensurating the subgroup $C^{(M)}$, where $C$ is a 2-element group), but the converse does not hold, since $\SL_3(\Z[1/2])$ does not have Property FG, as we see by using its action by conjugation on $\SL_3(\Q_2)$. In particular, Property FG does not follow from Property T. On the other hand, Shalom and Willis \cite[Theorem 1.3]{ShW} proved that $\SL_d(\Z)$ for $d\ge 3$ as well as many other non-uniform lattices have Property FG; the argument also carrying over $\SL_2(\Z[\sqrt{2}])$. The only source of Property FG I know uses distortion of abelian subgroups. 
\end{rem}


\section{Cardinal definite function on abelian groups and applications}\label{s_ab}

\subsection{Cyclic groups}
Let $\N$ be the set of nonnegative integers.

\begin{prop}\label{cdz}
Let $\ell$ be an unbounded cardinal definite function on $\Z$ associated to a transitive $\Z$-set $X$. 
Then there exists a bounded function $b:\Z\to\N$ such that we have $\ell(n)=|n|+b(n)$ for all $n\in\Z$.
\end{prop}
\begin{proof}
Since $\ell=\ell_M$ is unbounded, $X$ is infinite and hence we can identify $X$ to $\Z$. Then $M\subset\Z$ is commensurated by translations and hence has a finite boundary in the standard Cayley graph of $\Z$. It follows that $M$ is commensurate to $\emptyset$, $\Z$, $\N$ or $-\N$. The first two cases are excluded since $\ell$ is unbounded, and in the last two cases, we obtain that $\ell(n)=|n|+b(n)$ with $b:\Z\to\Z$ bounded. It turns out that $b(\Z)\subset\N$; indeed, for all $k\ge 1$ and $n$, we have, by subadditivity
\[|kn|+b(kn)=\ell(kn)\le k\ell(n)=k|n|+kb(n),\]
whence, dividing by $k$ we obtain $\min(b)/k\le b(n)$ for all $k\ge 1$, and picking $k>-\min(b)$ we deduce that $b(n)>-1$, whence $b(n)\ge 0$.
\end{proof}

\begin{defn}Let $G$ be a compactly generated locally compact group and $Z$ an infinite discrete cyclic subgroup. Let $|\cdot|$ be the word length in $G$ with respect to a compact generating subset. If $g$ is a generator of $Z$, the limit $\lim_{n\to\infty}|g^n|/n$ exists; $Z$ is called {\em distorted} if this limit is zero and {\em undistorted} otherwise. (This does not depend on the choice of the compact generating subset of $G$.)

Let us also say that $G$ has {\em uniformly undistorted discrete cyclic subgroups} (the word ``discrete" can be dropped if $G$ is discrete) if \[\inf_g\lim_{n\to\infty}|g^n|/n>0,\] where $g$ ranges over non-elliptic elements, i.e., generators of infinite discrete cyclic subgroups.
\end{defn}

The following corollary was first obtained indirectly by F.\ Haglund by studying the dynamics of isometries of CAT(0) cube complexes.

\begin{cor}\label{zundi}
Let $G$ be a compactly generated locally compact group and $Z$ an infinite discrete cyclic subgroup. If $Z$ is distorted then $(G,Z)$ has relative Property FW. In particular, if $G$ has Property PW then $Z$ is undistorted; actually $G$ has uniformly undistorted cyclic subgroups.
\end{cor}
\begin{proof}
Let $f$ be a cardinal definite function on $G$. Then $f$ is a length function, and in particular $f$ is asymptotically bounded by the word length of $G$. If $Z$ is distorted, it follows that $f$ is sublinear on $Z$ with respect to the usual word length on $Z\simeq\Z$. By Lemma \ref{cdz}, it follows that $f$ is bounded on $Z$. For the uniform statement, assume that $f$ is proper; then for every non-elliptic element $g$, Lemma \ref{cdz} implies $f(g^n)\ge |n|$ for all $n\ge 1$. For some constant $c>0$, we have $f\le c|\cdot|$, where $|\cdot|$ is the word length on $G$. Hence $\lim_{n\to\infty}|g^n|/|n|\ge 1/c$, which is independent of $g$.
\end{proof}

\begin{exe}
The central generator $z$ of the discrete Heisenberg group $H$ is quadratically distorted, in the sense that $|z^n|\simeq\sqrt{n}$; it follows that $H$ does not have Property PW (this example was noticed in \cite{Hag}).
\end{exe}

\begin{exe}
Let us indicate examples of finitely generated groups in which infinite cyclic subgroups are all undistorted, but not uniformly.

\begin{enumerate}\item Let $u\in\GL_4(\Z)$ be a matrix whose characteristic polynomial $P$ is irreducible over $\Q$ with exactly 2 complex eigenvalues on the unit circle; for instance $P=(X^2+(1+\sqrt{2})X+1)(X^2+(1-\sqrt{2})X+1)$ and $u$ is its companion matrix. Then the corresponding semidirect product $\Z^4\rtimes_u\Z$ has its cyclic subgroups undistorted, but not uniformly. Indeed, if $p$ is the projection to the sum in $\R^4$ of eigenvalues of modulus $\neq 1$, then for $g\in\Z^4\smallsetminus\{0\}$ (written multiplicatively), we have $\lim_{n\to\infty}|g^n|/n\simeq \|p(g)\|$, which is never zero but accumulates at zero. It follows that $\Z^4\rtimes_u\Z$ does not have Property PW (this also follows from Proposition \ref{polyfw}).

\item If a finitely generated group admits a subgroup isomorphic to $\Z[1/p]$ for some prime $p$, or more generally a non-cyclic torsion-free locally cyclic subgroup, then its cyclic subgroups are not uniformly undistorted. It is likely that there are examples in which they are all undistorted. 
\end{enumerate}
\end{exe}

\begin{exe}
Groups with Property PW are not the only instances of finitely generated groups with uniformly undistorted cyclic subgroups. Other examples include cocompact lattices in semisimple Lie groups, and finitely generated subgroups of the group of permutations of $\Z$ with bounded displacement.
\end{exe}

\begin{exe}\label{sl2z2}
Let $A$ be the ring of integers a number field which is not $\Q$ or an imaginary quadratic extension, e.g., $A=\Z[\sqrt{2}]$. Then $\SL_2(A)$ has Property FW. Note that in contrast, it has the Haagerup Property, as a discrete subgroup (actually a non-uniform irreducible lattice) in some product $\SL_2(\R)^{n_1}\times\SL_2(\C)^{n_2}$ with $n_1+n_2\ge 2$. For instance, $A=\Z[\sqrt{2}]$ is a lattice in $\SL_2(\R)^2$. To see Property FW, first observe that the condition on $A$ implies that $A^\times$ has an infinite order element; a first consequence is that, denoting by $U_{12}$ and $U_{21}$ the upper and lower unipotent subgroups in $A$, that $U_{12}$ and $U_{21}$ are both exponentially distorted. In particular, since they are finitely generated abelian groups and are thus boundedly generated by cyclic subgroups, it follows from Corollary \ref{zundi} that $(\SL_2(A),U_{12}\cup U_{21})$ has relative Property FW. Now a more elaborate consequence of the condition on $A$ is the theorem of Carter, Keller and Paige, see Witte Morris \cite{W}: $\SL_2(A)$ is boundedly generated its two unipotent subgroups. It follows, using the trivial observation that $(G,L)$ has relative Property FW implies $(G,L^n)$ has relative Property FW for all $n\ge 1$, that $\SL_2(A)$ has Property FW.
\end{exe}


\begin{prop}\label{cez}
Let $G$ be a compactly generated locally compact abelian group. Let $Z$ be a normal infinite cyclic discrete subgroup. Then either $(G,Z)$ has relative Property FW, or $Z$ is a topological direct factor in some open normal subgroup of finite index of $G$ containing $Z$.
\end{prop}
\begin{proof}
Assume that $(G,Z)$ does not have relative Property FW.
Let $H_1$ be the centralizer of $Z$ in $G$, which is open of index at most 2.  
Let $X$ be a continuous discrete $G$-set and $M$ a commensurated subset with open stabilizer such that $\ell_M$ is unbounded on $Z$.  
Decompose $X$ into $Z$-orbits as $\bigcup_{i\in I}X_i$. Note that $G$ permutes the $Z$-orbits and thus naturally acts on $I$; since the stabilizer of $X_i$ contains the stabilizer of any $x\in X_i$, this action is continuous. Let $J$ be the set of $i$ such that $M_i=X_i\cap M$ is infinite and coinfinite in $X_i$. Then $X_J=\bigcup_{i\in J}X_i$ is $G$-invariant. Define $K=I\smallsetminus J$, $M_J=M\cap X_J$, $M_K=M\smallsetminus X_J$. Then $\ell_M=\ell_{M_J}+\ell_{M_K}$. Since $Z$ is compactly generated, it follows from Proposition \ref{cgco2} that $M_K$ is transfixed by $Z$, so $\ell_{M_K}$ is bounded on $Z$. Thus $\ell_{M_J}$ is unbounded on $Z$ and in particular $J\neq\emptyset$. Again by Proposition \ref{cgco2}, $J$ is finite. So some open subgroup of finite index $H_2$ of $G$ fixes $J$ pointwise. 
Define $H=H_1\cap H_2$.

Let us pick $j\in J$. Then $Z$ acts on $X_j$ with a single infinite orbit; therefore if $L$ is the pointwise stabilizer of $X_j$ in $H$, we have $H=Z\times L$ as an abstract group; thus the continuous homomorphism $Z\times L\to H$ is bijective; since $Z\times L$ is $\sigma$-compact, it is a topological group isomorphism.
\end{proof}

\begin{exe}\label{sl2ti}
Let $\Gamma$ be a cocompact lattice in $\PSL_2(\R)$ and $\tilde{\Gamma}$ its inverse image in $\widetilde{\PSL}_2(\R)$. Then $\tilde{\Gamma}$ does not have Property PW. Indeed, its center $Z\simeq\Z$ is not virtually a direct factor: this is well-known and follows from the fact that for any finite index subgroup $\Lambda$ of $\tilde{\Gamma}$ there exist $g\ge 2$ and $2g$ elements $x_1,y_1,\dots,y_g\in\tilde{\Gamma}$ (actually generating a finite index subgroup of $\Lambda$, namely $\Lambda$ itself if $\Lambda$ is torsion-free) such that $\prod_{i=1}^g[x_i,y_i]$ is a nontrivial element of $Z$.  

On the other hand, $\tilde{\Gamma}$ has its infinite cyclic subgroups undistorted (and actually uniformly undistorted). Note that it also has the Haagerup Property, because $\widetilde{\PSL}_2(\R)$ has the Haagerup Property, by \cite[Chap.~4]{CCJJV}.
\end{exe}

\subsection{Abelian groups}

\begin{lem}\label{cgacd}
Let $A$ be a compactly generated locally compact abelian group and $X$ a continuous discrete transitive $A$-set with a non-transfixed commensurated subset; denote by $K$ the kernel of the $A$-action on $X$. Then there is a continuous surjective homomorphism $\chi:A\to\Z$ (unique up to multiplication by $-1$) such that the kernel $K$ is an open finite index subgroup of $\Ker(\chi)$. In particular, the associated cardinal definite function on $A$ has the form $g\mapsto [\Ker(\chi):K]|\chi(g)|+b(g)$ where $b:A\to\N$ is a bounded continuous function.
\end{lem}
\begin{proof}
Denote $A'=A/K$. In particular, the action of $A'$ is simply transitive; since point stabilizers are open it follows that $A'$ is discrete, and hence finitely generated by assumption. Since the action is simply transitive, we also deduce that $A'$ is multi-ended, and therefore is virtually infinite cyclic. Write $A'=\Z\times F$ with $F$ finite abelian. Define $M'=\bigcup_{g\in F}gM$; it is $F$-invariant and commensurate to $M$. By Lemma \ref{cdz}, $\ell_{M'}(n,f)=|F||n|+b'(n)$ with $b'$ bounded and hence $\ell_{M}(n,f)=|F||n|+b(n)$ with $b$ bounded; the same argument as in the proof of Lemma \ref{cdz} shows that $b\ge 0$. Letting $\chi$ denote the composite homomorphism $A\to A'\to \Z$, observe that $|F|=[\Ker(\chi):K]$ and $\chi$ is determined up to the sign, so the proof is complete.
\end{proof}

\begin{lem}\label{absch}
Let $A$ be an abelian group. Let $(\chi_i)_{i\in I}$ be a family of pairwise non-proportional nonzero homomorphisms $A\to\R$. Then the family $(|\chi_i|)$ is linearly independent over $\R$, in the space of functions $A\to\R$ modulo bounded functions.
\end{lem}
\begin{proof}
We have to show that for every nonzero finitely supported family of real scalars $(\lambda_i)_{i\in I}$, the function $\sum_i\lambda_i|\chi_i|$ is unbounded. Assuming by contradiction the contrary (i.e., the above sum is bounded), we can suppose $I$ nonempty finite and that all $\lambda_i$ are nonzero. Then we can find a finitely generated subgroup of $A$ on which the restrictions of the $\chi_i$ are pairwise non-proportional. Therefore we are reduced to the case when $A$ is finitely generated, and actually we can also suppose that $A=\Z^k$, since all homomorphisms to $\R$ vanish on the torsion subgroup of~$A$.



Write $V=\R^k$, so that $A\subset V$.  Extend $\chi_i$ to a continuous homomorphism $\hat{\chi}_i:V\to\R$. Writing each element in $V$ as the sum of an element of $A$ and a bounded element, we see that the function $g\mapsto\hat{f}(g)=\sum_{i\in I}\lambda_i|\hat{\chi}_i(g)|$ is still bounded on $V$.
Since $\hat{f}(g^n)=n\hat{f}(g)$ for all $g\in V$ and $n\ge 0$, by computing $\lim_{n\to\infty}\hat{f}(g^n)/n$, we see that actually $\hat{f}(g)=0$ for all $g\in V$.  We thus have 
\[|\hat{\chi}_1(g)|=-\sum_{i>1}\frac{\lambda_i}{\lambda_1}|\hat{\chi}_i(g)|\quad\forall g\in V;\]
let $V_i\subset V$ be the hyperplane $\{\hat{\chi}_i=0\}$. The right-hand term is a smooth function outside $\bigcup_{i>1}V_i$, which does not contain $V_1$ because the hyperplanes $V_i$ are pairwise distinct. But the left-hand term is smooth at no point of $V_1$. This is a contradiction.
\end{proof}

\begin{prop}\label{decab}
Let $A$ be a compactly generated locally compact abelian group and $f$ a continuous cardinal definite function on $A$. Then there exist finitely many continuous homomorphisms $\chi_i:A\to\Z$ and a bounded function $b:A\to\N$ such that, for all $g\in A$ we have
$$f(g)=\sum_i|\chi_i(g)|+b(g).$$
Moreover, if the $\chi_i:A\to\Z$ are required to be surjective, this decomposition is unique modulo the ordering of the $\chi_i$ and changing $\chi_i$ into $-\chi_i$. We call the term $f_0(g)=\sum_i|\chi_i(g)|$ the homogeneous part of $f$; it is given by $f_0(g)=\lim_{n\to\infty}f(g^n)/n$, and satisfies $f_0(ng)=|n|f_0(g)$ for all $n\in\Z$.
\end{prop}
\begin{proof}
To prove the existence, first by Corollary \ref{ficd} we have $f=\sum_{i=1}^kf_i$ where $f_i$ is a cardinal definite function associated to a transitive continuous action. By Lemma \ref{cgacd}, we have $f_i(g)=n_i|\chi_i(g)|+b_i(g)$, where $n_i$ is a nonnegative integer and $\chi_i$ is a homomorphism to $\Z$ (which we can suppose to be surjective if $f_i$ is unbounded and 0 otherwise), and $b_i$ bounded and valued in $\N$.

The uniqueness statement immediately follows from Lemma \ref{absch} and the last statement is clear.
\end{proof}

Recall that given a compactly generated locally compact group $G$, a closed subgroup $H$ is {\em undistorted} if it is compactly generated and the word metric of $H$ is equivalent to the restriction of the word metric of $G$.
The following corollary generalizes the second statement in Corollary \ref{zundi}.

\begin{cor}\label{aund0}Let $A$ be a compactly generated locally compact abelian group; let $|\cdot|$ be the word length with respect to some compact generating subset of $A$ and let $\ell$ be a cardinal definite function on $A$. Assume that $\ell$ is proper. Then there exists some constants $c>0$, $c'\in\R$ such that $\ell\ge c|\cdot|-c'$.
\end{cor}
\begin{proof}
Any compactly generated locally compact abelian group $A$ has a cocompact lattice isomorphic to $\Z^k$ for some $k$, so we can suppose $A\simeq\Z^k$.

The result then follows from the following claim: let $f$ be a proper cardinal definite function on $\Z^k$. Then for some norm on $\R^k$ we have $f\ge \|\cdot\|$. 


In view of Proposition \ref{decab}, we can suppose that $f=\sum_{i\in I}|\chi_i|+b$ with $\chi_i:\Z^k\to\Z$ a surjective homomorphism and $b\ge 0$ a bounded function. Denote by $\chi_i$ the unique extension $\hat{\chi}_i$ as a continuous homomorphism $\R^k\to\R$. Define $\hat{f}=\sum|\hat{\chi}_i|$. Then $\hat{f}$ is a seminorm on $\R^k$; since its restriction to $\Z^k$ is proper, it is actually a norm. Since $f\ge\hat{f}$ on $\Z^k$, the claim is proved.
\end{proof}

\begin{cor}\label{aund}
If $G$ is a compactly generated locally compact group with Property PW, then any compactly generated closed abelian subgroup $A$ is undistorted.
\end{cor}
\begin{proof}
Let $f$ be a proper cardinal definite function on $G$. Let $l$ be the word length in $G$ with respect to a compact generating subset. Then $f\le cl$ for some $c>0$. By Corollary \ref{aund}, the restriction of $f$ to $A$ is equivalent to the word length of $A$. Hence the latter is asymptotically bounded by the restriction to $A$ of the word length of $G$, which means that $A$ is undistorted.
\end{proof}





\begin{rem}
Proposition \ref{decab} shows that if $G$ is an abelian compactly generated locally compact group and $f$ a cardinal definite function on $G$, there exists a (unique) least element $f_0$ in the set of cardinal definite functions $f'$ such that $f-f'$ is bounded. I do not know if such a statement holds for more general compactly generated locally compact groups.

Let us mention in the abelian case that if $f=\ell_M$, then $f_0$ has, by construction (see the proof of Lemma \ref{cgacd}), the form $\ell_{M'}$ with $M'$ commensurate to $M$. 
\end{rem}

\begin{lem}\label{zdfini}
Let $G$ be a compactly generated locally compact group with a closed normal discrete subgroup $A$ isomorphic to $\Z^d$. Suppose that $A$ is undistorted in $G$. Then the homomorphism $G\to\GL_d(\Z)$ has a finite image. 
\end{lem}
\begin{proof}
Let $S$ be the image in $\GL_d(\Z)$ of a compact, symmetric generating subset of $G$. Since $\Z^d$ is undistorted, for every $x\in\Z^d$ there exists a constant $C$ such that, in $\Z^d$, for all $n\in\N$ and $g\in S^n$, we have $\|gx^ng^{-1}\|\le Cn+C$ (where we write the law of $\Z^d$ multiplicatively and fix a norm $\|\cdot\|$ on $\R^d$). Rewriting this additively gives $\|n(g.x)\|\le Cn+C$; dividing by $n\ge 1$, this yields $\|g.x\|\le C+C/n\le 2C$ for all $g\in S^n$, this bound does not depend on $n$ and this shows that for every $x\in\Z^d$, the subset $\{g.x:g\in G\}$ is bounded. In particular, the union of orbits of all basis elements in $\Z^d$ is finite, so there is some finite index subgroup of $G$ fixing all basis elements. Thus $G\to\GL_d(\Z)$ has a finite image.
\end{proof}

The following proposition partly generalizes Proposition \ref{cez}.

\begin{prop}\label{cea}
Let $G$ be a compactly generated locally compact abelian group with Property PW. Let $A$ be a closed normal abelian subgroup; suppose that $A$ is discrete and free abelian of finite rank. Then $A$ is a topological direct factor in some open normal subgroup of finite index of $G$ containing $A$.
\end{prop}
\begin{proof}
By Property PW, the identity component of $G$ is compact. By Corollary \ref{aund}, $A$ is undistorted in $G$. By Lemma \ref{zdfini}, the centralizer $H_1$ of $A$ in $G$ (which is closed and contains $A$) has finite index in $G$, hence is open in $G$. Let $X$ be a continuous discrete $G$-set and $M$ a commensurated subset with open stabilizer such that $\ell_M$ is proper (we will only use that $\ell_M$ is proper on $A$). 

Arguing exactly as in the proof of Proposition \ref{cez}, we can suppose that $I=J$ in the notation therein, i.e., we can suppose that $I=J$ is finite and for every $j\in J$ we have $M_j$ and $X_j\smallsetminus M_j$ both infinite. Let $H$ be a normal open subgroup of finite index of $G$ contained in $H_1$ and fixing $J$ pointwise.

For each $j\in J$, let $A_j$ be the kernel of the action of $A$ on $X_j$. Then $X_j$ can be identified to $A/A_j$, where $A$ acts by left translations. The centralizer of the group of left translations in any group $\Gamma$ in $\SX(\Gamma)$ is the group of right translations, which means left translations in case $\Gamma$ is abelian. Since $H$ centralizes $A$, we deduce that $H$ acts on $X_j$ by left translations. Let $L_j$ be the kernel of the action of $H$ on $X_j$. Let $\chi_j:A\to\Z$ be the surjective homomorphism associated to the action of $A$ on $X_j$; pick $j_1,\dots,j_k$ so that $\chi_{j_1},\dots,\chi_{j_k}$ are independent, where $A\simeq\Z^k$ (they exist by properness of $\ell_{M_j}$ on $A$). Define $L=\bigcap_{i=1}^kL_{j_i}$. We claim that $H$ is the direct product $A\times L$ as a topological group. Indeed, we have $A\cap L=\{1\}$, because since $\bigcap_{i=1}^k\Ker(\chi_{j_i})=\{1\}$ on $A$; moreover $AL=H$ because $L$ is the kernel of the image of the action homomorphism $H\to\SX(\bigcup_{i=1}^kX_{j_i})$ and the image of this homomorphism coincides with its restriction to $A$. Thus $A\times L\to H$ is a bijective continuous homomorphism; since $A\times L$ is $\sigma$-compact it follows that is is a topological group isomorphism.
\end{proof}

\subsection{Polycyclic groups}\label{s_po}

The following theorem is chronologically the first obstruction to Property~PW (although it was not yet interpreted this way then).


\begin{thm}[Houghton \cite{Ho}]\label{houghtonpo}
Let $\Gamma$ be a virtually polycyclic group of Hirsch length $k$. Then a subgroup $\Lambda$ is coforked if and only if it has Hirsch length $k-1$ and has a normalizer of finite index.
\end{thm}

It is sometimes tempting to consider residually virtually abelian finitely generated groups; however since these include finite groups and since virtually abelian finitely generated groups are residually finite, this is just the same as residually finite. To avoid this, the natural definition is the following classical one.

\begin{defn}
A discrete group is crystallographic if it finitely generated, virtually abelian without nontrivial normal subgroup. We say that a discrete group $\Gamma$ is residually crystallographic if the intersection $\VD(\Gamma)$ of normal subgroups $N$ with $\Gamma/N$ crystallographic, is trivial.
\end{defn}

By Bieberbach's theorem, a discrete group is crystallographic if and only it admits a faithful, cocompact proper action on a Euclidean space. (A direct consequence, which can also be seen algebraically by using the notion of FC-center, is that finite index subgroups of crystallographic groups are crystallographic.)

The notation $\VD$ comes from the fact that $\Gamma/\VD(\Gamma)$ is often virtually abelian (see Lemma \ref{reueu}), so that $\VD(\Gamma)$ can be thought as a kind of ``virtual derived subgroup".

\begin{lem}\label{hino}Let $\Gamma$ be a virtually polycyclic group with Hirsch length $k$ and $\Lambda$ a subgroup. If $\Lambda$ has Hirsch length $k-1$ and has its normalizer has finite index in $\Gamma$, then it contains a finite index subgroup of $\VD(\Gamma)$.
\end{lem}
\begin{proof}Let $L_1$ be the normalizer of $\Lambda$; it has finite index. The group $L_1/\Lambda$ has Hirsch length 1 so has a infinite cyclic subgroup $L/\Lambda$ of finite index.

The subgroup $g\Lambda g^{-1}$ only depends on the class of $g$ in $G/L$. Consider the finite intersection $N=\bigcap_{g\in G/L}g\Lambda g^{-1}$. Since $N=\bigcap_{g\in G}g\Lambda g^{-1}$, it is normal in $G$. Also, the diagonal map $L\to\prod_{g\in G/L}L/g\Lambda g^{-1}\simeq\Z^{G/L}$ has kernel equal to $N$ and it follows that $L/N$ is a finitely generated abelian group. Thus $\Gamma/N$ is virtually abelian. If $N'/N$ is its maximal finite normal subgroup, then $\Gamma/N'$ is crystallographic and thus $\VD(\Gamma)\subset N'$. So $\VD(\Gamma)\cap N$ has finite index in $\VD(\Gamma)$ and is contained in $\Lambda$.\end{proof}

\begin{cor}\label{coho}Let $\Gamma$ be a virtually polycyclic group and $\Lambda$ a subgroup. If $\Lambda$ is coforked then it contains a finite index subgroup of $\VD(\Gamma)$.\qed\end{cor}

\begin{prop}\label{polyfw}
Let $\Gamma$ be a virtually polycyclic group. Then $(\Gamma,\VD(\Gamma))$ has relative Property FW.
\end{prop}
\begin{proof}
Consider a transitive commensurating action of $\Gamma$ on a set $X$ with commensurated subset $M$, and $x\in X$, and let $H$ be the stabilizer of $x$. Then by Corollary \ref{coho}, $H$ contains a finite index subgroup $Q_1$ of $\VD(\Gamma)$. Since $\VD(\Gamma)$ is finitely generated, it contains a characteristic subgroup of finite index $Q$ contained in $Q_1$. So $Q$ is normal in $\Gamma$ and fixes a point in the transitive $\Gamma$-set $X$. It follows that the $Q$-action on $X$ is identically trivial. Thus the action of $\VD(\Gamma)$ on $X$ factors through a finite group, which has Property FW, so it leaves invariant a subset commensurate to $M$. This shows that $(\Gamma,\VD(\Gamma))$ has relative Property FW.
\end{proof}

Note that Proposition \ref{polyfw}, which has just been proved using Theorem \ref{houghtonpo}, easily implies Theorem \ref{houghtonpo} (precisely, it boils down the proof of Theorem \ref{houghtonpo} to the case when $\Gamma$ is virtually abelian). Let us now provide a proof of Proposition \ref{polyfw} not relying on Theorem \ref{houghtonpo}, but relying instead on the results of Section \ref{s_ab}.

\begin{proof}[Alternative proof of Proposition \ref{polyfw}]
We argue by induction on the Hirsch length of $\VD(\Gamma)$. If it is zero, there is nothing to prove. Otherwise, $\VD(\Gamma)$ contains a an infinite $\Gamma$-invariant subgroup $\Lambda$ of minimal nonzero Hirsch length. Passing to a finite index characteristic subgroup, we can suppose that $\Gamma$ is abelian and torsion-free. Let us show that $(\Gamma,\Lambda)$ has relative Property FW. Otherwise, by contradiction there is a cardinal definite function $f$ on $\Gamma$ such that $f$ is unbounded on $\Lambda$. Let $P\subset\Lambda$ be the maximal subgroup of $\Lambda$ on which $f$ is bounded (which exists by an easy argument using that $\Lambda$ is finitely generated abelian). By minimality of $\Lambda$, the action of $\Gamma$ on $\Lambda\otimes\Q$ is irreducible. It follows that the intersection of $\Gamma$-conjugates of $P$ is zero. Thus there exist $\gamma_1,\dots,\gamma_k\in\Gamma$ such that $\bigcap_{i=1}^k\gamma_iP\gamma_i^{-1}=\{0\}$. Therefore, defining $f'(g)=\sum_if(\gamma_i^{-1}g\gamma_i)$, the function $f'$ is cardinal definite and is not bounded on any nontrivial subgroup of $\Lambda$. By Proposition \ref{decab}, it follows that $f$ is proper on $\Lambda$. By Proposition \ref{cea} (which, as indicated in its proof, only uses the properness of the cardinal definite function on $\Lambda$), $\Lambda$ is a direct factor of some normal finite index subgroup $H$ of $\Gamma$; thus clearly $\Lambda\cap\VD(H)=\{1\}$. Thus the image of $\Lambda$ in $\Gamma/\VD(H)$ is infinite. Since $H$ has finite index in $\Gamma$, the group $\Gamma/\VD(H)$ is virtually abelian thus its quotient by some finite normal subgroup is a crystallographic quotient of $\Gamma$ in which $\Lambda$ has an infinite image. Thus the image of $\Lambda$ in $\Gamma/\VD(\Gamma)$ is infinite. This contradicts $\Lambda\subset\VD(\Gamma)$.

Thus $(\Gamma,\Lambda)$ has relative Property FW. Since $\Lambda\subset\VD(\Gamma)$, we have $\VD(\Gamma/\Lambda)=\VD(\Gamma)/\Lambda$. By induction, $(\Gamma/\Lambda,\VD(\Gamma)/\Lambda)$ has relative Property FW. So by the relative version of Proposition \ref{extfw} (see Remark \ref{relaver}), $(\Gamma,\VD(\Gamma))$ has relative Property FW.
\end{proof}

If $\Gamma$ is a group, recall that its {\em first Betti number} $\bbb_1(\Gamma)$ is the $\Q$-rank of $\textnormal{Hom}(\Gamma,\Z)$, which is either a finite integer or $+\infty$ (this definition is questionable when $\Gamma$ is infinitely generated but this is the one we use; for instance $\bbb_1(\Q)=0$ with this definition). Note that if $\Lambda\subset\Gamma$ has finite index then $\bbb_1(\Lambda)\ge\bbb_1(\Gamma)$. Also, its {\em first virtual Betti number} is defined as $\vb_1(\Gamma)=\sup_\Lambda \bbb_1(\Lambda)$, where $\Lambda$ ranges over finite index subgroups of $\Gamma$.

\begin{lem}\label{intereu}
Let $\Gamma$ be a group and $N_1,N_2$ normal subgroups. If $\Gamma/N_i$ is crystallographic for $i=1,2$ then so is $\Gamma/(N_1\cap N_2)$.
\end{lem}
\begin{proof}
The group $\Gamma/(N_1\cap N_2)$ is naturally a fibre product of $\Gamma/N_1$ and $\Gamma/N_2$, i.e.\ a subgroup of the product $\Gamma/N_1\times\Gamma/N_2$ both of whose projections are surjective. In particular, it is virtually abelian, and if $F$ is a finite normal subgroup, each of the projections of $F$ is normal and therefore is trivial, so that $F$ is trivial.
\end{proof}

\begin{lem}\label{reueu}
Let $\Gamma$ be a discrete group with $\vb_1(\Gamma)<\infty$. Then $\Gamma/\VD(\Gamma)$ is crystallographic.
\end{lem}
\begin{proof}
Equivalently, we have to show that if $\Gamma$ is residually crystallographic and $k=\vb_1(\Gamma)<\infty$ then $\Gamma$ is crystallographic.

Let $N$ be a normal subgroup of finite index in $\Gamma$ having a surjective homomorphism onto $\Z^k$. Let $N'/[N,N]$ be the largest finite normal subgroup of $\Gamma/[N,N]$; then $\Gamma/N'$ is crystallographic, so it is enough to prove that $N'=1$. Since $\Gamma$ is residually crystallographic, this amounts to proving that for every normal subgroup $P$ such that $\Gamma/P$ is crystallographic, we have $N'\subset P$. By Lemma \ref{intereu}, we can suppose $P\subset N'$ and thus have to show that $P=N'$. The crystallographic group $\Gamma/P$ is the extension of $N'/P$ and $\Gamma/N'$. Since $\vb_1(\Gamma)=\vb_1(\Gamma/N')=k$ and $\Gamma/P$ is virtually abelian, we deduce that $N'/P$ is finite; since this is a finite normal subgroup of $\Gamma/P$, it is then trivial and thus $P=N'$.
\end{proof}


Thus we have the following corollary of Proposition \ref{polyfw}:

\begin{cor}
Let $\Gamma$ be a virtually polycyclic group. Then $\Gamma$ has Property PW if and only it is virtually abelian.
\end{cor}
\begin{proof}
Clearly $\Z^k$ has Property PW and hence, by Proposition \ref{pwfi}(\ref{pwfi1}), every finitely generated virtually abelian group has Property PW.

Conversely, if $\Gamma$ is virtually polycyclic with Property PW, it follows from Proposition \ref{polyfw} and Lemma \ref{reueu} that $\Gamma$ is finite-by-(virtually abelian). Since any virtually polycyclic group is residually finite, it follows that $\Gamma$ is virtually abelian.
\end{proof}

\begin{rem}
The classes of crystallographic groups and residually crystallographic groups is obviously not stable under taking subgroups, since nontrivial finite subgroups are not (residually) crystallographic but can be embedded into crystallographic groups: every finite group $F$ embeds into the crystallographic group $\Z\wr F$. Also, if $\Gamma$ is a torsion-free finite index subgroup in $\SL_3(\Z)$, then $\Gamma\ast\Z$ is torsion-free and residually crystallographic although its subgroup $\Gamma$ is not. A source of (torsion-free) residually crystallographic groups is given by finitely generated RFRS groups (``residually finite rational solvable"): a group $\Gamma$ is RFRS if it has a descending sequence $(\Gamma_n)$ of finite index normal subgroups such that $\Gamma_0=\Gamma$ and $\Gamma_{i+1}$ contains the intersection of all kernels of homomorphisms $\Gamma_i\to\Q$ for all $i\ge 1$: indeed for every $x$ we have $p_i(x)\neq 1$ for large $i$, where $p_i$ is the quotient map to the quotient $\Gamma/[\Gamma_i,\Gamma_i]$ by its largest finite normal subgroup. On the other hand, the class of RFRS groups is stable under taking subgroups.
\end{rem}

Let us give a little variation beyond the virtually polycyclic case. We need the following lemma.

\begin{lem}\label{indexma}
Let $A$ be a finitely generated abelian group and $f$ a proper cardinal definite function on $A$. Then there exists $m$ such that for every abelian overgroup of finite index $B\supset A$ and measure definite function $f'$ on $B$ extending $f$, we have $[B/T_B:A/T_A]\le m$, where $T_A$ and $T_B$ denote the torsion groups in $A$ and $B$. In particular, if $B$ is torsion-free then $[B:A]\le m$.
\end{lem}
\begin{proof}
Define $V=A\otimes_\Z\R$ and let $\Gamma_A$ be the image of $A$ in $V$; it is a lattice. Let $f_0$ be the homogeneous part of $f$, which extends naturally to $V$; since $f$ is proper, $f_0^{-1}(\{0\})=\{0\}$. Consider the open polyhedron $\Omega=\{x\in V:|f_0(x)|<1\}$. There exists a positive lower bound for the covolume of a lattice $\Lambda$ in $V$ such that $\Lambda\cap\Omega=\{0\}$; in particular, there is an upper bound $m$ for the index of an overgroup $\Lambda$ of $\Gamma_A$ such that $\Lambda\cap\Omega=\{0\}$.

Let now $f'_0$ be the homogeneous part of $f'$. Since for $g\in B$ we have $f'_0(g)=\lim_{n\to\infty}f'(g^{n!})/n!$ and $g^{n!}\in A$ for large $n$, we have $(f'_0)_{|A}=f_0$. Also define $\Gamma_B$ as the image $B$ in $V$. So $\Gamma_B$ is a lattice in $V$ containing the lattice $\Gamma_A$, and since $f_0$ takes integer values on $B$, we have $\Gamma_B\cap\Omega=\{0\}$. So $[\Gamma_B:\Gamma_A]\le m$. But $\Gamma_B=B/T_B$ and $\Gamma_A=A/T_A$, so $[\Gamma_B:\Gamma_A]=[B/T_B:A/T_A]\le m$.
\end{proof}

\begin{cor}
Let $G$ be a locally compact group with Property PW. Then every discrete torsion-free abelian subgroup $H$ of finite $\Q$-rank is free abelian (of finite rank).
\end{cor}
\begin{proof}
(Recall that by definition a torsion-free abelian group has finite $\Q$-rank if it is isomorphic to a subgroup of $\Q^k$ for some $k$, and its $\Q$-rank is the minimal possible $k$.)

Let $A\subset H$ be a free abelian subgroup of maximal rank. So we can write $H=\bigcup A_n$ (nondecreasing union) with $[A_n:A]<\infty$. By Lemma \ref{indexma}, we have an upper bound on $[A_n:A]$, and therefore $[H:A]<\infty$ and thus $H$ is free abelian of finite rank.
\end{proof}

\section{Median graphs}\label{medgr}

In the previous chapters, we developed the theory avoiding references to median graphs and CAT(0) cube complexes. In this chapter we pursue the study using the notion of median graphs, still avoiding CAT(0) cube complexes, with the exceptions of Remarks \ref{mediancat0} and \ref{cat0med}, and Corollary \ref{fixmedian}. 

\subsection{Median graphs: main examples and first properties}\label{mger}

On a metric space $(D,d)$, we write, unless ambiguous, $xy=d(x,y)$. 
Define the total interval $[x,y]$, for $x,y\in D$ as the set of $t$ such that $xt+ty=xy$. The metric space $D$ is called {\em median} if for all $x,y,z$, the intersection $[x,y]\cap [y,z]\cap [z,x]$ is a single point, called the median of $(x,y,z)$ and denoted $m(x,y,z)$.

A subset $E$ of a median space $D$ is called {\em totally convex} if $[x,y]\subset E$ for all $x,y\in E$, and {\em biconvex}\footnote{A biconvex subset is sometimes called a {\em halfspace}. In this paper, we rather reserve the word {\em halfspace} to the context of wallings, although in the sequel we will indeed prove that in connected median graphs, there is a walling for which the halfspaces are the biconvex subsets.} if both $E$ and its complement $E^c$ are totally convex; we call {\em proper biconvex subsets} those biconvex subsets distinct from $\emptyset$ and $D$. Denote by $\CB_D$ the set of biconvex subsets of $D$. For $x\in D$, denote by $\CB_D(x)$, or $\CB(x)$ if there is no ambiguity, the set of biconvex subsets of $D$ containing $x$.


A graph (not oriented, without multi-edges and self-loops and identified with its set of vertices) is {\em median} if each of its connected components (as a graph) is a median metric space. 

A {\em median subgraph} of a connected median graph is a connected full subgraph that is isometrically embedded, and stable under taking median of triples.

\begin{exe}[hypercubes]\label{commed}
Let $X$ be a set. Endow the power set $2^X$ with a graph structure by calling $N,N'$ adjacent if $\#(N\tu N')=1$; thus $N,N'$ are in the same component if and only if they are commensurate (in other words, the connected component of $N$ is the set $2^X_N$ of subsets of $X$ having finite symmetric difference with $N$) and their graph distance is then $\#(N\tu N')$. Then this graph is median. Indeed, the total interval $[N,N']$ between any two commensurate subsets $N$, $N'$ is the set of subsets trapped between $N\cap N'$ and $N\cup N'$. It easily follows that the intersection $[N,N']\cap [N',N'']\cap [N'',N]$ is the singleton $\{(N\cap N')\cup (N'\cap N'')\cup (N''\cap N)\}$ (more symmetrically described as the set of elements belonging to at least two of the three subsets $N,N',N''$) and thus the graph is median. 

Note that all components of the graph $2^X$, which are called {\em hypercubes} (or cubes when $X$ is finite), are isomorphic as graphs, but not canonically; it is useful to define all of them altogether, especially when dealing with group actions.
\end{exe}

\begin{exe}\label{commedz}
Let $X$ be a set. Endow $\Z^X$ with a graph structure by calling $f,f'$ adjacent if $f-f'$ or $f'-f$ is the Dirac function at some $x\in X$. Similarly as in Example \ref{commed}, this is a median graph; the connected component of $f$ is the set of $f'$ such that $f'-f$ has a finite support, in which case the total interval $[f,f']$ is the set of functions $g$ such that $\min(f,f')\le g\le\max(f,f')$, and the median of $f,f',f''$ is the function mapping $x$ to the median point in $(f(x),f'(x),f''(x))$.
\end{exe}

\begin{prop}\label{cdmg}
Let $G$ be a topological group and $f$ a cardinal definite function on $G$. Then there exists a connected median graph with a continuous isometric action of $G$ and a vertex $v$ such that $f(g)=d(v,gv)$ for all $G$.
\end{prop}
\begin{proof}
Let $X$ be a discrete continuous $G$-set and $M$ a commensurated subset with open stabilizer such that $f=\ell_M$. Consider the action of $G$ on the power set $2^X$, endowed with the median graph structure given in Example \ref{commed}. Then by assumption, this action preserves the connected component $D=2^X_M$. Moreover, since points in $X$ and $M$ have open stabilizers, so do all elements in $2_M^X$. So the action of $G$ on $X$ is continuous. If $v=M$, then we have $f(g)=d(v,gv)$ for all $g\in G$. 
\end{proof}

Although the construction is canonical, there is no uniqueness statement in Proposition \ref{cdmg}. Here the median graph constructed is, in a certain sense, huge (e.g., it is not locally finite unless $X$ is finite). There are indeed improved versions of the proposition, see \S\ref{icm}. Before this, we will show (Corollaries \ref{medca2} and \ref{orcd}) that, up to multiplication by 2, the converse of proposition holds: every action on a connected median graph and choice of vertex gives rise to a cardinal definite function.

Let us now give basis properties of median graphs. As a warm-up let us begin with the following easy but very useful observation.

\begin{lem}[Bipartite lemma]\label{medeven}
Every median graph is bipartite, in the sense that it admits a 2-coloring of the set of vertices for which every edge is bicolor. In other words, it has no loop of odd length. Still equivalently, the distance is additive modulo 2: $xy+yz\equiv xz\mod 2$ for all $x,y,z$.
\end{lem}
\begin{proof}
If we have a counterexample of length $2n+1$ with $n$ minimal, then it is a geodesic loop (since otherwise we could find a smaller loop of odd length). Hence there exists 3 points $x,y,z$ with $xy=1$, $yz=xz=n$. If $m=m(x,y,z)$, then $m\in\{x,y\}$ and we obtain $xz=yz\pm 1$, a contradiction.
\end{proof}

In particular, any median graph has no loop of length 3. In contrast, median graph usually have many loops of length 4, by the following result.

\begin{prop}\label{cubusc}
In any connected median graph, the fundamental group is generated by squares. More precisely, if we fix a vertex $x_0$ and $V^1$ is the 1-skeleton, then $\pi_1(V^1,x_0)$ is generated by the $\gamma\lambda\gamma^{-1}$ where $\gamma$ ranges over paths emanating from $x_0$ and $\lambda$ ranges over squares based at the endpoint of $\gamma$.
\end{prop}
\begin{proof}
Consider a combinatorial loop $c$ given by consecutive vertices $x_0,x_1,\dots,x_n=x_0$. This means that $x_i$ is adjacent to $x_{i+1}$ for all $i$ (in particular, $x_i\neq x_{i+1}$).
Define $r=r(c)=\max_id(x_0,x_i)$. If $r=0$, there is nothing to do; assume $r>0$.

First, if there exists $i$ such that $x_i=x_{i+2}$, we pass to another homotopic loop of combinatorial length $n-2$ by removing $x_{i+1}$ and $x_{i+2}$. Do this again until this operation is impossible (no backtrack except maybe at $x_0$). Let $c'=(x'_0,\dots,x'_{n'})$ for the resulting loop. 

Define $I_{\max}=\{i:d(x_0,x_i)=r\}$.
Observe that whenever $i,j\in I_{\max}$ are distinct, we have $|i-j|\ge 2$, as a consequence of the bipartite Lemma (Lemma \ref{medeven}).

For each $i$, define $y_i$ as follows: if $i\notin I_{\max}$, we set $y_i=x'_i$; if $i\in I_{\max}$, then $d(x'_0,x'_{i-1})=d(x'_0,x'_{i+1})=r-1$ (since these cannot be equal to $r+1$ by maximality and to $r$ by Lemma \ref{medeven}). Then define $y_i=m(x'_0,x'_{i-1},x'_{i+1})$. Then, since  $x'_{i-1}\neq x'_{i+1}$, we see that $d(x'_0,y_i)=r-2$.

Let $c''=(y_0,\dots,y_{n'})$ be the resulting loop (note that $y_0=x'_0=x_0$). Then $r(c'')<r(c)$, and $c''$ and $c'$ are homotopic up to a finite product of ``squares". By iterating the process (at most $r(c)$ times), we see that $c$ is homotopic to the trivial loop.
\end{proof}



Obviously, every tree is median. Conversely, we obtain:

\begin{cor}\label{treemed}
If a connected median graph has no square (that is, no injective loop of size 4), then it is a tree.\qed
\end{cor}

\begin{rem}
An isometrically embedded subgraph of a connected median graph need not be median. For instance, consider the 3-cube (as a graph with 8 vertices). Removing two opposite vertices yields a graph isomorphic to a hexagon, which is not median (for instance, it contradicts Corollary \ref{treemed}).
\end{rem}

\subsection{Median orientations}

\begin{defn}\label{d_orien}
In a graph, by a {\em directed edge}, we mean a pair $(x,y)$ of adjacent vertices. We say that two directed edges $(x,y)$ and $(x',y')$ are {\em elementary parallel} if $xx'=yy'=1$ and $xy'=x'y=2$. We define the {\em parallelism} relation between directed edges as the equivalence relation generated by {\em elementary parallelism}.

An orientation on a graph is a map $\eps$ from the set of directed edges to $\{\pm 1\}$ such that $\eps(y,x)=-\eps(x,y)$ for every directed edge $(x,y)$; the number $\eps(x,y)$ is called the orientation of $(x,y)$; this is usually represented by an edge pointing from $x$ to $y$ if $\eps(x,y)=1$. An orientation is called {\em median} if all parallel directed edges have the same orientation, or equivalently if all elementary parallel direct edges have the same orientation. 
\end{defn}


\begin{defn}
If $x,y$ are vertices in a connected graph, define
$B_{x,y}=\{z:zx\le zy\}$.
\end{defn}

\begin{lem}
Let $V$ be a connected bipartite graph and $(x,y)$ is a directed edge, then $B_{y,x}$ is equal to the complement of $B_{x,y}$.
\end{lem}
\begin{proof}
This amounts to proving that no $z$ satisfies $zx=zy$, which more generally holds whenever $xy$ is odd.
\end{proof}

\begin{lem}\label{everybi}
Let $V$ be a connected bipartite graph. Then any strict biconvex subset $B$ has the form $B_{x,y}$ for some directed edge $(x,y)$, actually for any directed edge $(x,y)$ such that $x\in B$ and $y\notin B$.
\end{lem}
\begin{proof}
(Note that $(x,y)$ fulfilling the last condition exists by connectedness of $V$.) Let us show $B=B_{x,y}$. It is enough to show that $B\subset B_{x,y}$, since the other inclusion means $B^c\subset B_{y,x}$ and has the same proof.

Let $z$ belong to $B$, and assume by contradiction $z\in B_{y,x}$. Then $y\in [z,x]$. Since both $z,x$ belong to $B$, we deduce that $y\in B$, a contradiction.
\end{proof}

\begin{lem}\label{parb}
Let $V$ be a connected median graph and $(x,y)$, $(x',y')$ be parallel directed edges. Then $B_{x,y}=B_{x',y'}$. In particular,
\begin{itemize}\item the directed edge $(x,y)$ is not parallel to $(y,x)$; equivalently, every connected median graph admits a median orientation;
\item if $x_0,\dots,x_n$, $n\ge 2$, is a geodesic segment, then $(x_0,x_1)$ is parallel to none of $(x_{n-1},x_n)$, $(x_n,x_{n-1})$.
\end{itemize}
\end{lem}
\begin{proof}
We can suppose that $(x,y)$ and $(x',y')$ are elementary parallel. By contradiction, suppose that $z\in B_{x,y}\cap B_{y',x'}$. Set $k=zx$, so that $zy=k+1$. Then, using the bipartite lemma, $zx'\in\{k-1,k+1\}$, $zy'\in\{k,k+2\}$; since $zx'>zy'$, we deduce $zx'=k+1$ and $zy'=k$. Then we see that both $x$ and $y'$ are medians for $(z,x',y)$. Hence they are equal, a contradiction. (Note that we only used the uniqueness of the median.) 

The last statement follows since $B_{y,x}\neq B_{x,y}$.
\end{proof}

Also recall that in a graph, if $D$ is a set of vertices, $\partial D$ is the set of vertices in $D$ adjacent to some vertex outside $D$.

\begin{lem}\label{totcvxb}
In a connected graph, every subset $D$ with totally convex boundary is totally convex.
\end{lem}
\begin{proof}
Let $x_0,\dots,x_n$ be a geodesic segment with $x_0,x_n\in D$. If by contradiction $x_j\notin D$ for some $j$, let $i<j$ be maximal and $k>j$ be minimal such that $x_i,x_k\in D$. Then $i<j<k$ and $x_i,x_k\in\partial D$. Since the latter is totally convex, we deduce that $x_i\in\partial D\subset D$, a contradiction. 
\end{proof}

\begin{lem}\label{starred}
If $V$ is a connected median graph and $(x,y)$ is a directed edge, then $\partial B_{x,y}$ is totally starred at $x$, in the sense that for any $x'\in \partial B_{x,y}$, the total interval $[z,x]$ is contained in $\partial B_{x,y}$. Moreover, there for any $y'\in B_{y,x}$ adjacent to $x'$, the directed edge $(x',y')$ is parallel to $(x,y)$.
\end{lem}
\begin{proof}
It is enough to show that for every $n\ge 1$ and geodesic segment $(x_0,x_1,\dots,x_n)$ in $V$ with $x_0=x$, and $x_n\in\partial B_{x,y}$, we have $x_{n-1}\in\partial B_{x,y}$ and there exists $y_n\in B_{y,x}$ such that $(x_n,y_n)$ is elementary parallel to $(x,y)$. This is proved by induction on $n\ge 1$, the case $n=1$ being trivial. 

We have $yx_n\ge xx_n+1=n+1$. Then $xx_{n-1}=n-1$ and $yx_{n-1}\ge yx_n-1=n$; since $|xx_{n-1}-yx_{n-1}|\le 1$, it follows that $yx_{n-1}=n$. Let $y_n$ be an element in $B_{y,x}$ adjacent to $x_n$. Since $xy_n\neq xx_{n-1}$, we have $y_n\neq x_{n-1}$; since they are both adjacent to $x_n$, we deduce that $x_{n-1}y_n=2$. Let $y_{n-1}$ be the median $m(x_{n-1},y_n,y)$. Then $yy_{n-1}=n-1$ and $y_ny_{n-1}=x_{n-1}y_{n-1}=1$. Then $xy_{n-1}\ge xy_n-y_ny_{n-1}=n$, and $xy_{n-1}\le xx_{n-1}+x_{n-1}y_{n-1}=n$. Thus $y_{n-1}\in B_{y,x}$ and is adjacent to $x_{n-1}$, which implies that $x_{n-1}\in\partial B_{x,y}$.

Moreover, $(x_n,y_n)$ is elementary parallel to $(x_{n-1},y_{n-1})$, which by induction is parallel to $(x,y)$; hence $(x_n,y_n)$ is parallel to $(x,y)$. 
\end{proof}

\begin{thm}\label{tbico}
Let $V$ be a connected median graph. Then for every directed edge $(x,y)$, the subset $\partial B_{x,y}$ is totally convex in $V$, the subset $B_{x,y}$ is biconvex in $V$, and all strict biconvex subsets have this form. Moreover, for every directed edge $(x',y')$, we have $B_{x,y}=B_{x',y'}$ if and only if $(x,y)$ and $(x',y')$ are parallel.
\end{thm}
\begin{proof}
Let us check that $\partial B_{x,y}$ is totally convex. Suppose we have a geodesic segment joining $x'$ to $x''$, which are in $\partial B_{x,y}$. Let $y'\in B_{y,x}$ be adjacent to $x'$. Then by the second assertion of Lemma \ref{starred}, $(x',y')$ is parallel to $(x,y)$, and then by Lemma \ref{parb}, $B_{x,y}=B_{x',y'}$. Then $\partial B_{x',y'}$ is starred at $x'$, by the first assertion of Lemma \ref{starred}, which implies that it contains the given segment between $x'$ and $x''$, proving that $\partial B_{x,y}$ is totally convex.

That every strict biconvex subset has the form $B_{x,y}$ was observed in Lemma \ref{everybi} under the bare assumption that $V$ is bipartite. Conversely, we have to check that $B_{x,y}$ is biconvex. Since its boundary is totally convex as we have just proved, we obtain that it is totally convex by the easy Lemma \ref{totcvxb}. The same argument holds for its complement $B_{y,x}$, and therefore $B_{x,y}$ is biconvex.

For the last statement, Lemma \ref{parb} yields one implication. Conversely, assume that $B_{x,y}=B_{x',y'}$. Then $(x',y')$ is a directed edge with $x'\in B_{x,y}$ and $y'\in B_{y,x}$; by the second assertion of Lemma \ref{starred}, we deduce that $(x',y')$ is parallel to $(x,y)$.
\end{proof}






\begin{exe}\label{orx}
The graphs $2^X$ and $\Z^X$ of Examples \ref{commed} and \ref{commedz} are canonically oriented, namely by putting an oriented edge from $f$ to $f'$ if $f'-f$ is a Dirac function; this orientation is median.
\end{exe}

As a corollary of the last statement in Theorem \ref{tbico}, we have

\begin{cor}Let a group $G$ act by graph automorphisms on a connected median graph. Equivalences:
\begin{enumerate}
\item $G$ preserves some median orientation;
\item $G$ preserves every median orientation;
\item the action of $G$ has no wall inversion, in the sense that for every nonempty biconvex subset $B$ and $g$ we have $gB\neq B^c$.\qed
\end{enumerate}
\end{cor}

\begin{rem}In the classical case of trees, a wall inversion means the existence of an edge inversion, that is, a directed edge $(x,y)$ mapped to $(y,x)$. In the more setting of connected median graphs, an edge inversion is an example of a wall inversion, but it is not the only example. For instance, if $\Z^2$ is endowed with its standard Cayley graph structure, the action of the cyclic group generated by the graph automorphism $(m,n)\mapsto (1-m,n+1)$ has a wall inversion, although it has no edge inversion.
\end{rem}

\begin{cor}\label{bijmur}
Let $V$ be a connected median graph and $B$ a biconvex subset. Then every element $v\in\partial B$ is adjacent to a unique $\phi_B(v)$ in $\partial(B^c)$. The map $\phi_B$ is a graph isomorphism $\partial B\to\partial (B^c)$ with inverse $\phi_{B^c}$.
\end{cor}
\begin{proof}
Let $x\in\partial B$ be adjacent to both $y,y'\in B^c$. By Lemma \ref{everybi}, we have $B_{x,y}=B_{x,y'}$. Hence $y'\in B_{y,x}$, which means that $1=y'x\ge y'y$. Since $yy'$ is even, this forces $yy'=0$. This proves the uniqueness statement, the existence being trivial; also it clear that $\phi_B$ admits $\phi_{B^c}$ as inverse. Finally let us show that it is a graph homomorphism: let $(x_1,x_2)$ be a directed edge in $\partial B$, and $y_i=\phi_B(x_i)$. If $y_1y_2\neq 1$, it follows that $y_1y_2=3$, and hence $x_1\in [y_1,y_2]$. This contradicts Theorem \ref{tbico}, namely that $\partial B^c$ is totally convex. Hence $\phi_B$ is a graph homomorphism; since so is its inverse $\phi_{B^c}$, we deduce that $\phi_B$ is a graph isomorphism.
\end{proof}

\begin{rem}
In a connected median graph, the set of $\partial B$, when $B$ ranges over strict biconvex subsets, are often called {\em hyperplanes}. Also define $\ppt B=\partial B\cup \partial (B^c)$; it is called {\em carrier} of $B$; it is also a totally convex subgraph.

Simple examples (for instance, with $V$ a tree) show that $B\smallsetminus\pt B$ can fail to be totally convex. It can actually be empty (with $B$ proper biconvex), or disconnected. Also, the hyperplane $\pt B$ does not always determine $B$.

Similarly, the carrier $\ppt B$ does not always determine the unordered pair $\{B,B^c\}$ (although this holds in a tree): for instance, in a cube or hypercube, the only thick hyperplane is the graph itself.


On the other hand, the mapping $B\mapsto (\ppt B,\pt B)$ is injective.
\end{rem} 


\subsection{Canonical walling of a median graph}

Let $V$ be a connected graph. For any vertex $x\in V$, recall that $\CB(x)$ denotes the set of biconvex subsets of $V$ containing $x$. 


\begin{prop}\label{numberwap}
Let $V$ be a connected bipartite graph, and $x=x_0,\dots,x_n=y$ a geodesic segment between vertices $x$ and $y$. Then the $B_{x_i,x_{i+1}}$, for $0\le i\le n-1$ are pairwise distinct and 
\[\CB(x)\smallsetminus\CB(y)\subset\{B_{x_i,x_{i+1}}:0\le i\le n-1\};\]
in particular, $\#(\CB(x)\smallsetminus\CB(y))\le d(x,y)$.
\end{prop}
\begin{proof}
It is clear that for $i<j$, we have $B_{x_i,x_{i+1}}\neq B_{x_j,x_{j+1}}$ because only the second contains $x_j$ (using that the segment is geodesic), and the inclusion goes as follows: if $B\in\CB(x)\smallsetminus\CB(y)$, then there exists $i$ such that $x_i\in B$ and $x_{i+1}\notin B$, and then $B=B_{x_i,y_i}$ by Lemma \ref{everybi}.
\end{proof}

\begin{thm}\label{numberwa}
Under the assumptions of Proposition \ref{numberwap}, assume in addition that $V$ is median. Then the inclusion and inequality of Proposition \ref{numberwap} are equalities:
\[\CB(x)\smallsetminus\CB(y)=\{B_{x_i,x_{i+1}}:0\le i\le n-1\};\quad\#(\CB(x)\smallsetminus\CB(y))= d(x,y).\]
\end{thm}
\begin{proof}
All $B_{x_i,x_{i+1}}$ are biconvex by Theorem \ref{tbico} and we obtain the reverse inclusion from Proposition \ref{numberwap}. The last assertion follows.
\end{proof}

\begin{cor}
Let $V$ be a connected bipartite graph. Then the set of biconvex subsets is a self-indexed walling of (the vertex set of) $V$. If moreover $V$ is median, then the wall distance is equal to twice the graph distance.

If moreover $V$ is median, and some median orientation on $V$ is given, and if we call a strict biconvex subset $B_{y,x}$ positive if $(x,y)$ is an oriented edge, then the resulting self-indexed walling on $V$ induces the graph distance.\qed
\end{cor}

\begin{cor}\label{cormed}
If $V$ is a nonempty connected median graph, then the map $x\mapsto\CB(x)$ is a canonical isometric embedding of $(V,2d)$ into the some component of $2^{\CB_V}$; this embedding is equivariant with respect to group actions. Moreover, it is a median homomorphism: $m(\CB(x),\CB(x'),\CB(x''))=\CB(m(x,x',x''))$ for all $x,x',x''\in V$.
\end{cor}
\begin{proof}By Proposition \ref{numberwap}, all $\CB(x)$ for $x\in V$ belong to the same component. 
The symmetric difference $\CB(x)\tu\CB(y)$ is equal to $(\CB(x)\smallsetminus\CB(y))\sqcup(\CB(y)\smallsetminus\CB(x))$ and thus by Theorem \ref{numberwa} it has exactly $2d(x,y)$ elements.

The last assertion is straightforward since an isometry between median metric spaces is automatically a median homomorphism. Still, let us provide a direct argument: observe that for all $x,y\in D$ and $w$ in the total interval $[x,y]$, we have 
\[\CB(x)\cap \CB(y)\subset \CB(w)\subset \CB(x)\cup \CB(y).\]
It follows that 
\[(\CB(x)\cap \CB(x'))\cup (\CB(x)\cap \CB(x''))\cup (\CB(x')\cap \CB(x''))\subset \CB(m(x,x',x''))\]
\[\subset (\CB(x)\cup \CB(x'))\cap (\CB(x)\cap \CB(x''))\cap (\CB(x')\cap \CB(x'')).\]
Since the left-hand and the right-hand term both coincide with $m(\CB(x),\CB(x'),\CB(x''))$, this yields the result.
\end{proof}

\begin{cor}\label{cormed2}
Let $V$ be a connected median graph and $x_0$ a vertex. Then the map $x\mapsto\CB(x_0)\smallsetminus\CB(x)$ is an isometric embedding of $(V,d)$ into $2^{(\CB(x_0))}$ mapping $x_0$ to 0.
\end{cor}
\begin{proof}
Observe that $(\CB(x_0)\smallsetminus\CB(x))\tu (\CB(x_0)\smallsetminus\CB(y))$ is equal to $\CB(x_0)\cap(\CB(x)\tu\CB(y))$. The set $\CB(x)\tu\CB(y)$ contains exactly $d(x,y)$ pairs of opposite biconvex subsets, and for each such pair exactly one belongs to $\CB(x_0)$, whence the cardinal of $\CB(x_0)\cap(\CB(x)\tu\CB(y))$ equals $d(x,y)$.
\end{proof}



\begin{cor}\label{extbi}
If $V$ is a connected median graph and $W$ an isometrically embedded median subgraph, then the map $B\mapsto B\cap W$ from $\CB_V$ to $\CB_W$ is surjective, and every strict biconvex of $W$ has a unique preimage; in other words, every strict biconvex subset of $W$ is contained in a unique (strict) biconvex subset of~$D$.\qed  
\end{cor}

In case of a group action on a median graph by graph automorphisms we obtain the following, which can be viewed as a converse to Proposition \ref{cdmg}.

\begin{cor}\label{medca2}
Let $G$ be a topological group with a continuous action on a connected median graph. Then for every vertex $v\in V$, the function $g\mapsto 2d(v,gv)$ is cardinal definite on $G$.
\end{cor}
\begin{proof}
Indeed, the action of $G$ on $\CB_V$ is continuous (because its point stabilizers contains the edge stabilizers of the $G$-action on $V$), it commensurates $\CB(v)$, which also has an open stabilizer (the stabilizer of $v$), and
\begin{align*}\#(\CB(v)\tu g\CB(v))= & \#(\CB(v)\tu \CB(gv))\\=&\#((\CB(v)\smallsetminus\CB(gv))\sqcup(\CB(gv)\smallsetminus\CB(v)))=2d(v,gv).\qedhere\end{align*} 
\end{proof}

In terms of lengths, the multiplication by 2 ``loss" in the combination of Proposition \ref{cdmg} and Corollary \ref{medca2} can be avoided if we consider orientations, see Corollary \ref{orcd}.


If a connected median graph $V$ is endowed with a median orientation $\eps$, define a biconvex subset to be positive if it has the form $B_{y,x}$ with $(x,y)$ positively oriented (so that the arrow goes from the negative biconvex $B_{x,y}$ to the positive $B_{y,x}$). Let $\CB_V^+$ be the set of positive biconvex subsets, $\CB^+(x)=\CB(x)\cap\CB^+_V$.
For any strict biconvex subset, define $B^+$ as the unique positive biconvex subset in $\{B,B^c\}$ (if necessary, write $B^+(\eps)$ to specify $\eps$).
Then Theorem \ref{numberwa} implies:


\begin{cor}\label{egnb}
Let $V$ be a connected median graph endowed with a median orientation, and $x=x_0,\dots,x_n=y$ a geodesic segment between vertices $x$ and $y$. Then the $B^+_{x_i,x_{i+1}}$, for $0\le i\le n-1$ are pairwise distinct and 
\[\CB^+(x)\tu\CB^+(y)=\{B_{x_i,x_{i+1}}^+:0\le i\le n-1\}.\]
In particular, $\#(\CB^+(x)\tu\CB^+(y))=d(x,y)$.\qed
\end{cor}

Then if $G$ preserves this orientation, its action on $\CB_V$ preserves $\CB_V^+$ and commensurates $\CB^+(x)$.

\begin{cor}\label{orcd}
Let $G$ be a topological group and $\ell:G\to\N$ a function. Then $\ell$ is cardinal definite if and only if there exists a connected median graph, endowed with a median orientation, and a continuous action of $G$ on this graph preserving the structure of oriented graph.
\end{cor}
\begin{proof}
For $\Rightarrow$, it is enough to observe that the proof of Proposition \ref{cdmg} actually provides an action preserving the canonical median orientation on $2^X$ (described in Example \ref{orx}). 

Conversely, given a continuous action on a connected oriented median graph $V$ with some vertex $x$ such that $\ell(g)=d(x,gx)$ for all $g\in G$, we obtain a commensurating action on $(\CB_V,\CB(x))$, which is continuous (in the sense that the stabilizers of points and of $\CB(x)$ are continuous); by restriction, we obtain that the ``sub-action" on $(\CB_V^+,\CB^+(x))$ is continuous as well. The associated cardinal function is then equal to $\ell$, by Corollary \ref{egnb}.
\end{proof}

\subsection{Roller orientations and Roller boundary}

Recall that orientations are introduced in Definition \ref{d_orien}.

\begin{defn}
Let $V$ be a connected bipartite graph. If $z$ is a vertex, define the $z$-{\em orientation} $\eps_z$ to be such that whenever $(x,y)$ is a directed edge, $\eps_z(x,y)=1$ if and only if $xz\ge yz$ (that is, $z\in B_{y,x}$: this means that $(x,y)$ points towards $z$). Define an orientation to be {\em principal} if there exists such a $z$.
\end{defn}

Obviously $x\mapsto\eps_x$ is injective, because if $w\neq z$ are vertices and $x$ is the first vertex in a geodesic segment joining $w$ to $z$, then $\eps_w(w,x)=-1$ and $\eps_z(w,x)=1$.


\begin{defn}
On a connected median graph $V$, define an orientation to be {\em Roller} if it is median and for any finite family of biconvex subsets $B_1,\dots,B_n$, we have $\cap B_i^+\neq\emptyset$. The set $\mathfrak{R}(V)$ of Roller orientations is a compact subset of the set of orientations and, identifying $V$ to its image in $\mathfrak{R}(V)$ through the injective map $z\mapsto\eps_z$, it is called the Roller boundary of $V$.
\end{defn}

Beware that the opposite of a Roller orientation is not necessarily Roller.

On a connected median graph, any principal orientation is Roller, since then the nonempty intersection property holds for arbitrary families. This defines an injective map $V\to\mathfrak{R}(V)$. 

\begin{lem}\label{sqcom}
Let $(x,y_1)$ and $(x,y_2)$ be oriented edges in a connected median graph endowed with a Roller orientation, with $y_1\neq y_2$. Then $\{x,y_1,y_2\}$ is contained in a square: there exists a vertex $z$ such that $zy_1=zy_2=1$ and $zx=2$.
\end{lem}
\begin{proof}
By assumption, $B_{y_i,x}$ is positive for $i=1,2$. Define $s$ to be any point in $B_{y_1,x}\cap B_{y_2,x}$, and $z=m(s,y_1,y_2)$. Then $zy_1=zy_2=1$, and $z\neq x$; hence $zx=2$.
\end{proof}

\begin{lem}\label{carprin}
Let $V$ be a connected median graph endowed with a Roller orientation $\eps$ and $v$ a vertex with no outwards edge (that is, for every $x$ adjacent to $v$, we have $\eps(x,v)=1$). Then $\eps=\eps_v$.
\end{lem}
\begin{proof}
We claim that for every $n\ge 1$ and every finite path $v=x_0,\dots,x_n$ starting at $v$, we have $\eps(x_i,x_{i-1})=1$ for all $1\le i\le n$.  This holds by assumption for $n=1$; assume $n\ge 2$. Then this holds by induction (on $n$) for all $i<n$. In particular $\eps(x_{n-1},x_{n-2})=1$. Assume by contradiction that $\eps(x_{n-1},x_n)=1$. Then by Lemma \ref{sqcom}, there exists an vertex $x'_{n-1}$ such that $x'_{n-1}x_{n-2}=x'_{n-1}x_n=1$ and $x'_{n-1}x_{n-1}=2$. Since $\eps$ is median, we have $\eps(x_{n-2},x'_{n-1})=1$. Moreover, for all $i\le n-2$, we have $x'_{n-1}x_i\ge x_nx_i-1=n-i-1$ and $x'_{n-1}x_i\le x_{n-2}x_i+1=n-i-1$, whence $x'_{n-1}x_i=n-i-1$. Thus $(v=x_0,\dots,x_{n-2},x'_{n-1})$ is a smaller path with $\eps(x'_{n-1},x_{n-2})=-1$, contradicting the induction assumption. Hence $\eps(x_n,x_{n-1})$, finishing the induction step and proving the lemma.
\end{proof}

\begin{lem}\label{denrol}
Let $V$ be a connected median graph. The canonical inclusion map $V\to\mathfrak{R}(V)$, mapping $v$ to $\eps_v$, has a dense image. If $V$ is locally finite, then it also has an open discrete image.
\end{lem}
\begin{proof}
Let $\eps$ be a Roller orientation, and let $e_1,\dots,e_k$ be directed edges. Let us show the result by finding a principal orientation coinciding with $\eps$ on all $e_i$. By assumption $\bigcap B_{e_i}^+(\eps)$ contains some vertex $v$. By definition, $v\in B_{e_i}^+(\eps_v)$. This means that $\eps_v(e_i)=\eps(e_i)$ for all $i$.

To prove the last result, let us check that more generally if $v$ has only finitely many adjacent vertices $x_1,\dots,x_k$, then $\eps_v$ is isolated in $\mathfrak{R}(V)$. Indeed, let $W$ be the set of orientations $\eps$ such that $\eps(x_i,v)=1$ for all $i$. Then $W$ is a clopen neighborhood of $\eps_v$. Let us show that $W=\{\eps_v\}$. If $\eps\in W$, we have $\bigcap_i B_{x_i,v}(\eps)=\{v\}$. 

\end{proof}

\begin{rem}
Roller orientations are often called ultrafilters by analogy (choice for any unordered pair of complementary biconvex subset of a representative, with some compatibility condition), but we have to be careful with this analogy because the set of biconvex subsets is not stable under taking intersections. Besides, let us emphasize that, unlike non-principal ultrafilters, it is usually easy to exhibit non-principal Roller orientations.
\end{rem}

\begin{exe}
If we consider the usual graph structure on $\Z$, it is median and admits exactly two opposite non-principal median orientations, one of which being defined by $\eps(n,n+1)=1$ and $\eps(n,n-1)=-1$ for all $n\in\Z$.
\end{exe}

\begin{exe}
On the hypercube $2^{(X)}=2^X_\emptyset$, if $M\subset X$, we can define the directed edge $(F\sqcup\{x\},F)$ to be oriented if and only if $x\notin M$; let $\eps^M$ be the corresponding orientation. Then $\eps^M$ is a Roller orientation, and is principal if and only if $M$ is finite (then it is equal to $\eps_M$; moreover every median orientation has this form (see Proposition \ref{cacu} for a converse); thus we see that the Roller boundary of $2^{(X)}$ is naturally identified with $2^X$ with its natural product topology.
\end{exe}

\begin{prop}\label{pathgeo}
On a connected median graph endowed with a Roller orientation, every oriented path is geodesic.
\end{prop}
\begin{proof}
It is enough to prove it for a finite path $x_0,\dots,x_n$ with $(x_i,x_{i+1})$ positive for all $i$; in turn, by density (Lemma \ref{denrol}), it is enough to prove it for a principal orientation $\eps_v$. Hence we see that $d(x_i,v)=d(x_0,v)-i$ for all $0\le i\le n$ by an immediate induction on $i$, which implies that the path is geodesic.
\end{proof}

\begin{prop}\label{orgera}
In a connected median graph endowed with a non-principal Roller orientation $\eps$, there exists a positively oriented infinite geodesic ray.
\end{prop}
\begin{proof}
Fix a vertex $x_0$ and define by induction a positively oriented geodesic ray $(x_0,\dots,x_n)$. Since $\eps\neq\eps_{x_n}$, there exists by Proposition \ref{carprin} a neighbor $x_{n+1}$ of $x_n$ such that $\eps(x_{n},x_{n+1})=1$. Hence $(x_0,\dots,x_{n+1})$ is an oriented path; by Proposition \ref{pathgeo}, it is geodesic.
\end{proof}

\begin{cor}
A connected median graph admits a non-principal Roller orientation if and only if it contains an infinite geodesic ray.
In particular, on a bounded connected median graph, every Roller orientation is principal.
\end{cor}
\begin{proof}
The forward implication is given by Proposition \ref{orgera}. Conversely, if $(x_n)_{n\ge 0}$ is an infinite geodesic ray, then there exists, by compactness, a limit point $\eps$ of the sequence of orientations $\eps_{x_n}$. Then $\eps$ is Roller. For $\eps$, the geodesic ray $(x_n)_{n\ge 0}$ is positively oriented, and hence $\eps$ is not principal.
\end{proof}

\begin{exe}
On a tree, any orientation is median, while an orientation is Roller if and only if every vertex has at most one outwards edge (in other words, if $(x,y)$ and $(x,y')$ are oriented edges then $y=y'$). Thus we easily see that the Roller boundary of $T$ can be identified, pointwise, with the union of $V$ with its usual boundary defined as geodesic rays up to eventual coincidence modulo translation.

Even in a tree, $V$ need not be open in $\mathfrak{R}(V)$: for instance, if in a tree there exists a vertex $x_0$ and infinitely many geodesic rays $(x_n^{(i)})_{n\ge 0}$ with $x_0^{(i)}=x_0$ and $x_1^{(i)}\neq x_1^{(j)}$ for all $i\neq j$ (e.g., the tree is regular of infinite valency), then $V$ is not open in $\mathfrak{R}(V)$.
\end{exe}

\begin{exe}
Consider the standard Cayley graph of $\Z^2$. Then its Roller orientations are
\begin{itemize}
\item The principal orientations $\eps_{m,n}$, $(m,n)\in\Z^2$;
\item the orientations $\eps_{m,\infty}$: all vertical edges are oriented upwards, horizontal edges are oriented to the right or to the left according to whether they are contained in $\{(x,y):x\le m\}$ or $\{(x,y):x\ge m\}$; and the similarly defined orientations $\eps_{m,-\infty}$, $\eps_{\pm\infty,n}$;
\item the 4 orientations $\eps_{\pm\infty,\pm\infty}$. For instance, in $\eps_{\infty,\infty}$, all edges are oriented to the right or upwards.
\end{itemize}

\end{exe}


\subsection{Convex hulls}

\begin{defn}
Let $V$ be a connected median graph (identified with its 0-skeleton). The {\em total convex hull} of a subset $S\subset D$ is the intersection $\TConv_V(S)$, of all totally convex subsets of $D$ containing $S$
\end{defn}

\begin{prop}\label{hull}
Let $V$ is a connected median graph. Then the following properties hold:
\begin{enumerate}
\item\label{ic1} every intersection of totally convex subsets of $V$ is totally convex and is a median subgraph; in particular this holds for the total convex hull of any subset;
\item\label{ic3} 
for every $S\subset V$, its total convex hull is equal the intersection of all biconvex subsets containing $S$;
\item\label{ic4} if $S\subset D$ is finite then $\TConv_V(S)$ is finite.
\end{enumerate}
\end{prop}
\begin{proof}
The assertion (\ref{ic1}) is trivial.


The assertion (\ref{ic3}) amounts to proving that if $E\subset V$ is totally convex and $x\in V\smallsetminus E$, then there exists a biconvex subset of $V$ containing $E$ but not $x$.
Let $y$ be a point in $E$ with $d(x,y)$ minimal. Let $y=y_0,y_1,\dots,x$ be a geodesic segment joining $y$ to $x$. Let us check that for all $z\in E$ we have $d(y_1,z)=d(y_1,y)+d(y,z)$. Indeed, consider the median $m=m(y_1,y,z)\in\{y,y_1\}$; we have $m\in E$ by total convexity and hence $m=y$, whence the desired equality. It follows that $B_{y,y_1}$ is a biconvex subset containing $E$ but not $x$.

For (\ref{ic4}), suppose $Y$ finite and define $W=\TConv(Y)$. If $B$ is a strict biconvex subset of $W$, then by Corollary \ref{extbi}, $B$ is contained in a unique strict biconvex subset of $V$ and thus it follows from the definition of $W$ that there exist $y,y'\in Y$ such that $Y$ contains $y$ but not $y'$. Since the number of possible $B$ is finite for each given $(y,y')$ by Proposition \ref{numberwap}, we deduce that the number of biconvex subsets of $W$ is finite. By Theorem \ref{numberwa}, the biconvex subsets of $W$ separate the points in $W$ and we deduce that $W$ is finite.

Let us give an alternative proof of (\ref{ic4}). Thanks to (\ref{ic3}), it is enough to check that every finite subset $Y$ of $V$ is contained in a finite totally convex subset. When $V=2^{(X)}$ (Example \ref{commed}) for some set $X$, this holds, observing that the set of subsets of a given finite subset is totally convex. We obtain the result for $V$ arbitrary by embedding it as an isometrically embedded median subgraph of such a hypercube $V'=2^{(X)}$ by Corollary \ref{cormed2}: pick a finite totally convex subset $E\subset V'$ containing $Y$; then $W\cap V$ is totally convex in $V$ and contains $Y$.
\end{proof}

\begin{rem}
In contrast to Proposition \ref{hull}(\ref{ic4}), it is not true that if $S\subset V$ is bounded then its essential hull is bounded. Indeed, if $V=2^{(X)}$ for some infinite set $X$ and $S$ is the set of singletons, then $S$ is bounded (of diameter 2), and its essential hull is equal to $V$, which is unbounded.
\end{rem}

\begin{defn}\label{defeft}
Let $V$ be a connected median graph, and let $G$ be a group acting on $V$ by graph automorphisms. We say that the $G$-action on $V$ is
\begin{itemize}
\item {\em of finite type} if there are finitely many $G$-orbits of biconvex subsets;
\item {\em essentially of finite type} if there exists a nonempty $G$-invariant totally convex invariant subset $E$ with finitely many $G$-orbits of biconvex subsets.
\end{itemize}
\end{defn}

\begin{prop}\label{orbibi}
Let $V$ be a connected median graph, $x_0$ a vertex of $V$ and let $G$ be a group acting on $V$ by graph automorphisms. Equivalences:
\begin{enumerate}
\item\label{fo1} $\CB(x_0)\cap \mathcal{Z}\in\{\emptyset,\mathcal{Z}\}$ for all but finitely many $G$-orbits $\mathcal{Z}\subset\CB_V$;
\item\label{fo2} for every vertex $x$, $\TConv(Gx)$ has finitely many $G$-orbits of biconvex subsets;
\item\label{fo3} the $G$-action on $V$ is essentially of finite type.
\end{enumerate}
In particular, being essentially of finite type is inherited by totally convex $G$-invariant subgraphs.
\end{prop}
\begin{proof}
It is clear that (\ref{fo2}) implies (\ref{fo3}). For the sequel, first observe that since for every $x$ we have $\CB(x)\tu\CB(x_0)$ finite, the condition of (\ref{fo1}) holds for some $x_0$ if and only if it holds for all $x$ instead of $x_0$.

Suppose (\ref{fo3}); by the previous remark we can suppose $x_0\in E$. The 
set of biconvex subsets of $V$ whose intersection with $E$ is strict biconvex in $E$ is canonically in bijection with the set of strict biconvex subsets of $E$ by Corollary \ref{extbi} and thus consists of finitely many $G$-orbits; if $\mathcal{Z}$ is another $G$-orbit, for every $P\in \mathcal{Z}$ either $P$ contains $E$ or $P\cap E=\emptyset$, and since $E$ is $G$-invariant and $\mathcal{Z}$ is $G$-transitive, this does not depend on the choice of $P$ in $\mathcal{Z}$; thus $\CB(x_0)\cap \mathcal{Z}\in\{\emptyset,\mathcal{Z}\}$ and (\ref{fo1}) is proved.

Finally let us check that (\ref{fo1}) implies (\ref{fo2}). 
Let $B_1,\dots,B_k$ be representatives of $G$-orbits $\mathcal{Z}$ in $\mathcal{H}_D$ such that $\CB(x)\cap \mathcal{Z}\notin\{\emptyset,\mathcal{Z}\}$. Define $E=\TConv(Gx)$. If $B'$ is a proper biconvex subset of $E$, then $B'=B\cap E$ for some biconvex subset $B$ of $D$, by Corollary \ref{extbi}. Since $B$ induces a nontrivial partition of $Gx$, we have $B=gB_i$ for some $i$. Hence $B'=g(B_i\cap E)$ for some $i$; thus $E$ has finitely many $G$-orbits of biconvex subsets.
\end{proof}

We now state a useful corollary of Proposition \ref{cgco2}.

\begin{cor}\label{fmo}
Let $G$ be a topological group generated by a compact subset (or more generally with uncountable cofinality, see \S\ref{s_cof}). Let $G$ act continuously by automorphisms on a nonempty connected median graph $V$. Then the $G$-action on $V$ is essentially of finite type. 
%
\end{cor}
\begin{proof}
Consider the action of $G$ on the set $\mathcal{H}_V$ of biconvex subsets of $V$. This action commensurates the subset $\CB(x_0)$ of biconvex subsets containing a given vertex $x_0$, which has an open stabilizer since it contains the stabilizer of $x_0$. Hence by Proposition \ref{cgco2}, $\CB(x_0)\cap\mathcal{Z}\in\{\emptyset,Z\}$ for all but finitely many orbits $\mathcal{Z}\subset\mathcal{H}_V$. By Proposition \ref{orbibi}, we deduce that the essential hull $V'=\TConv(Gx_0)$ of the orbit $Gx_0$ (Proposition \ref{hull}) is totally convex, $G$-invariant and has finitely many $G$-orbit of biconvex subsets.
\end{proof}



\begin{exe}
Conversely, every topological group $G$ with cofinality $\omega$ (e.g., any infinitely generated countable group) admits an action on a median graph that is not essentially of finite type.
Indeed, write $G$ as the increasing union of a sequence $(H_n)$ of proper open subgroups. The disjoint union $\bigsqcup_n G/H_n$ is naturally the vertex-set of a tree, 
by linking any element in $G/H_n$ to its projection in $G/H_{n+1}$; here biconvex subsets are in bijection with directed edges; the only non-empty invariant subtrees are the disjoint unions $\bigsqcup_{n\ge k}G/H_n$ and have infinitely many edge orbits.
\end{exe}

\subsection{Cubes in median graphs}

\begin{defn}
Let $V$ be a bipartite graph and $x$ a vertex. 
We say that two directed edges $(x,y),(x,z)$ emanating from $x$ are {\em orthogonal} if $y\neq z$ and there exists a vertex $x'\neq x$ adjacent to both $y$ and $z$. (Note that this implies that $\{x,y,z,x'\}$ is a full subgraph, isometric to a square.)

If $(y_i)$ is a family of vertices adjacent to $x$, we call it orthogonal if $(x,y_i)$ is orthogonal to $(x,y_j)$ for all $i\neq j$.
\end{defn}

Denote by $C_k$ the standard $k$-cube, namely $\{0,1\}^k$; denote by $(u_i)$ the canonical basis of $\R^k$, so that $u_i$ can be viewed as a vertex of $C_k$. 
For any subset $I\subset\{1,\dots,k\}$, we also write $u_I=\sum_{i\in I}u_i$, so that the $u_I$ are the vertices of $C_k$.
We think of $C_k$ as a graph, with oriented edges of the form $(u,u+u_i)$ whenever $u$ and $u+u_i$ both vertices; such an edge is called $i$-labeled. 
Besides, if $(x,y)$ is an edge in $D$, we say that it is labeled by the biconvex subset $B=P_{(x,y)}$. Thus $(x,y)$ is labeled by a biconvex subset $B$ if and only if $x\in B$ and $y\notin B$.

\begin{lem}\label{uniqare}
Let $x,x'$ be at distance 2 in a median graph. Then there exist at most 2 elements at distance 1 from both $x$ and $x'$.
\end{lem}
\begin{proof}
Assuming we have 3 distinct such elements $y,y',y''$, they are pairwise at distance 2; hence $x,x'$ belong to all three intervals $[y,y']$, $[y',y'']$, $[y'',y]$, contradicting the uniqueness of the median.
\end{proof}

\begin{lem}\label{cubevrai}
Let $V$ be a median graph and $f$ be a graph homomorphism $C_k\to V$, such that $f(u_{\{i,j\}})\neq f(0)$ and $f(u_i)\neq f(u_j)$ for all $i\neq j$. Then $f$ is injective and is an isometric embedding onto a full subgraph.
\end{lem}
\begin{proof}
We argue by induction on $k$.
The case $k\le 2$ is immediate. Assume that the result holds until $k-1$.

Let us first check the injectivity.
First assume that $k=3$. Suppose that $f(u_I)=f(u_J)$; if $I\cup J\neq\{1,2,3\}$ we are done by induction. Hence up to symmetry (and using the bipartite lemma), we have two cases to consider: $(I,J)=(\{1\},\{1,2,3\})$ or $(I,J)=(\{1,2\},\{2,3\})$. In both cases, we deduce that $d(f(u_1),f(u_{\{2,3\}}))=1$. It follows that $f(u_{\{2,3\}})$ is a median for $f(u_1)$, $f(u_2)$ and $f(u_3)$, but $f(0)$ is also such a median. This implies $f(0)=f(u_{\{2,3\}})$, contradicting the assumption.

 
Now assume $k\ge 3$ arbitrary. Define, for $J\subset\{1,\dots,k-1\}$, $f'(u_J)=f(u_{J\cup\{k\}})$; then $f'$ is a graph homomorphism $C_{k-1}\to V$ satisfying, by the case $k=3$, the assumptions of the lemma. Hence by induction we deduce that $f'$ is injective. We deduce that whenever $I,J$ contains $k$ and are distinct, we have $f(u_I)\neq f(u_J)$. By a change of indices, the same conclusion holds whenever $I\cap J\neq\emptyset$.


Now assume that $f(u_I)=f(u_J)$ and $I,J$ are distinct. By the above, $I\cap J=\emptyset$. Up to switch $I$ and $J$, we can suppose that some $i\in I$ does not belong to $J$.
Using the induction on a smaller cube, we obtain that $1=d(f(u_{I\smallsetminus\{i\}}),f(u_J))=\#(I\tu J)-1$. Hence $\#(I\tu J)=2$; since they are disjoint, we have $I\tu J=I\cup J$, so $I\cup J$ has cardinal at most 2. But then the case $k=2$ implies $I=J$, contradiction. 

Now it remains to check that $f$ is an isometric embedding. Consider any $I,J\subset\{1,\dots,k\}$ and let us show that $d(f(u_I),f(u_J))=\#(I\tu J)$. By induction, $f$ is an isometric embedding in restriction to any $(k-1)$-cube; in particular we obtain the conclusion whenever $I\cup J\neq\{1,\dots,k\}$ or $I\cap J\neq\emptyset$. Now assume $I\sqcup J=\{1,\dots,k\}$; up to a change of origin, we can suppose that $J=\emptyset$ and $I=\{1,\dots,k\}$. Then if by contradiction $d(f(0),f(u_I))\neq k$, then it is equal to $k-2$. Thus $f(0)\in [f(u_i),f(u_I)]$ for all $i\in I$. On the other hand, $f(0)\in [f(u_i),f(u_j)]$ for all $i\neq j$. Hence $f(0)$ is a median for $f(u_I)$, $f(u_i)$, and $f(u_j)$ for all $i\neq j$; but this is also the case of $f(u_{\{i,j\}})$. We deduce that $f(0)=f(u_{\{i,j\}})$, a contradiction.

We deduce that $f$ is an isometric embedding; in particular its image is a full subgraph. (Note that we only used the uniqueness of the median.)
\end{proof}




\begin{lem}\label{fdet}
Let $V$ be a median graph. Let $f:C_k\to V$ be an injective graph homomorphism, such that $f(u_{\{i,j\}})\neq f(0)$ and $f(u_i)\neq f(u_j)$ for all $i\neq j$. Then $f$ is uniquely determined by its restriction to $\{0,u_1,\dots,u_k\}$.
\end{lem}
\begin{proof}
Let us show that $f(u_J)$ is determined by induction on $\#(J)$. For $\#(J)=1$ this holds by definition; for $\#(J)=2$, this holds as a consequence of the assumptions along with Lemma \ref{uniqare}. Define $x_i=f(u_{I\smallsetminus\{i\}})$. 

For $\#(J)\ge 3$, fix three distinct elements $i,j,k$ in $J$; observe that $f(u_J)$ is adjacent to $x_\ell$ for all $\ell\in\{i,j,k\}$. Since $x_i,x_j,x_k$ are pairwise distinct (by the injectivity assumption, which is actually unnecessary by Lemma \ref{cubevrai}), we deduce that $f(u_J)$ is a median of $x_i,x_j,x_k$, and hence is uniquely determined. (Again, note that we only used the uniqueness of the median.)
\end{proof}

\begin{thm}\label{cubumed}Let $V$ be a median graph, with a finite family of directed edges $(x,y_i)_{1\le i\le k}$ based at a single vertex. Equivalences:
\begin{enumerate}
\item\label{cu_ort} the family $(x,y_i)_{1\le i\le k}$ is orthogonal;
\item\label{cu_bic} the intersection $\bigcap_{1\le i\le k}B_{y_i,x}$ is not empty;
\item\label{cu_grh} there exists a graph homomorphism of the standard $k$-cube $C_k$ into $V$ mapping $0$ to $x$ and $u_i$ to $y_i$, and mapping, for all $i\neq j$, $u_i+u_j$ to a vertex distinct from $x$; 
\item\label{cu_iso} there exists an isometric embedding of the standard $k$-cube $C_k$ onto a full subgraph of $V$ mapping $0$ to $x$ and $u_i$ to $y_i$.
\end{enumerate}
Moreover, if these hold, the graph homomorphism in (\ref{cu_grh}) (or the embedding in (\ref{cu_iso})) is unique; its image is equal to the total convex hull of $\{x,y_1,\dots,y_k\}$.
\end{thm}

\begin{proof}
(\ref{cu_iso})$\Leftrightarrow$(\ref{cu_grh}): $\Rightarrow$ is trivial; the converse is ensured by Lemma \ref{cubevrai}.

(\ref{cu_iso})$\Rightarrow$(\ref{cu_bic}): indeed, if $f$ is such an isometric embedding, then $f(u_{\{1,\dots,k\}})$ belongs to this intersection.

(\ref{cu_bic})$\Rightarrow$(\ref{cu_ort}): for $i\neq j$, let $z$ be any element in $B_{y_i,x}\cap B_{y_j,x}$. Then the median $m(z,y_i,y_j)$ is distinct from $x$ and adjacent to both $y_i$ and $y_j$.

To conclude, let us prove (\ref{cu_ort})$\Rightarrow$(\ref{cu_iso}) by induction on $k$.
The case $k=1$ is trivial and the case $k=2$ just follows from the definition (and the bipartite lemma). Assume that $k\ge 3$ and that the implication is proved for $k-1$.

Denote $C_{[i]}$ and $C'_{[i]}$ the subgraph of $C_k$ consisting of the elements $u_I$ with $i\notin I$ (resp.\ $i\in I$); they are both isomorphic to $C_{k-1}$.
 
Then for every $i\neq k$, there exists an isometric embedding $f_i$ of $C_{[i]}$ into $V$ such that $f_i(0)=x$ and $f_i(u_j)=y_j$ for all $j\neq i$. For any $i\neq j$, by Lemma \ref{fdet} applied to the $(k-2)$-cube $C_{[i]}\cap C_{[j]}$, the functions $f_i$ and $f_j$ coincide on $C_{[i]}\cap C_{[j]}$. Hence, writing $I=\{1,\dots,k\}$ there exists a function $f$ on $C_k\smallsetminus\{u_I\}$ whose restriction to $C_{[i]}$ equals $f_i$ for every $i$.

Fix $i$. We can apply the induction hypothesis to the restriction of $f$ to $C'_{[i]}$: thus there exists an isometric graph homomorphism $f'_i:C'_{[i]}\to V$ such that $f'_i(u_i)=f(u_i)(=y_i)$ and $f'_i(u_i+u_j)=f(u_i+u_j)$ for all $j\neq i$. Then for all $j\neq i$, $f$ and $f'_i$ coincide on $C'_{[i]}\cap C_j$. Thus $f'_i(s)=f(s)$ for every $s\in C'_{[i]}\smallsetminus\{u_I\}$. It follows that $f(u_{I\smallsetminus\{j\}})$ and $f(u_{I\smallsetminus\{\ell\}})$ have distance 2 whenever $j,\ell$ are distinct: indeed, we can choose $i$ distinct from $i,j$ and argue that they are equal to $f'_i(u_{I\smallsetminus\{j\}})$ and $f'_i(u_{I\smallsetminus\{\ell\}})$ and use that $f'_i$ is isometric. 

Now let us separately deal with $k=3$ and $k\ge 4$. If $k=3$, the three points $f(u_{I\smallsetminus\{i\}})$, for $i=1,2,3$ are pairwise at distance 2, if we define $f(u_I)$ as their median, then $f$ is a graph homomorphism $C_3\to V$.

Suppose now $k\ge 4$. For any $i$ and $j,\ell$ all three distinct, observe that both points $f'_i(u_I)$, $f(u_{I\smallsetminus\{j,\ell\}})$ are adjacent to the two distinct points $f(u_{I\smallsetminus\{j\}})$ and $f(u_{I\smallsetminus\{\ell\}})$. Since $f'_i(u_I)$ and $f(u_{I\smallsetminus\{j,\ell\}})=f'_i(u_{I\smallsetminus\{j,\ell\}})$ are distinct (by injectivity of $f'_i$), we deduce from Lemma \ref{uniqare} that $f'_i(u_I)$ is the unique point adjacent to both $f(u_{I\smallsetminus\{j\}})$ and $f(u_{I\smallsetminus\{\ell\}})$ and distinct from $f(u_{I\smallsetminus\{j,\ell\}})$. In particular, it only depends on $\{j,\ell\}$; write it as $g(j,\ell)$. Then for $i,m$ distinct, we can choose $j,\ell$ such that all four are distinct, and then $f'_i(u_I)=g(j,\ell)=f'_m(u_I)$. Hence $f'_i(u_I)$ does not depend on $i$; in particular, its distance to $f(u_{I\smallsetminus\{i\}})$ is also equal to 1 and then $f$ is a graph homomorphism $C_k\to V$.

In both cases, we conclude that $f$ is isometric by Lemma \ref{cubevrai}. 
\end{proof}

\begin{prop}
Let $V$ be a connected median graph. The image of every isometric graph homomorphism $f:C_k\to V$ is totally convex.
\end{prop}
\begin{proof}
We can suppose that $V$ is equal to the total convex hull of $f(C_k)$, and we have to show that $f$ is surjective. Let $x$ be a vertex in $V$. Let $I$ be the set of $i$ such that $x\in B_{f(u_i),f(0)}$. Up to change the origin in $f$, we can suppose that $I=\emptyset$. Let $B$ be a biconvex subset containing $x$ and not $f(0)$. Then since $V$ is the total convex hull of $f(C_k)$, by Proposition \ref{hull}(\ref{ic3}), there exists an edge of $f(C_k)$ whose vertices belong to both sides of $B$. Hence by Lemma \ref{everybi}, $B=B_{y,z}$ for some such directed edge; moreover by Lemma \ref{parb}, this directed edge can be chosen to have the form either $(f(0),f(u_i))$ or $(f(u_i),f(0))$. Since by assumption $x\in\bigcap_j B_{f(0),f(u_i)}$, we deduce that $B=B_{f(0),f(u_i)}$; then this contradicts $f(0)\notin B$.
\end{proof}

\begin{defn}
In a median graph, a $k$-cube is the image of an injective graph homomorphism $f:C_k\to V$. A cube is a $k$-cube for some $k$.
\end{defn}

Thus in a connected median graph, every cube is actually isomorphic to $C_k$ and is a totally convex subgraph.

\begin{rem}\label{mediancat0}
Let $V$ be a median graph. It can be viewed as a cubical complex, where cubes are the cubes defined above. If $v$ is a vertex, the link of this cubical complex structure at $v$ is by definition the simplicial graph, a priori with possibly multiple simplices, for which the vertices are the neighbors of $v$ in $V$, and the $k$-simplices are indexed by the $(k+1)$-cubes through $v$: if $C$ is such a cube and if $y_0,\dots,y_k$ are the neighbors of $v$ in $C$, then it corresponds to a $k$-simplex $s_C=\{y_0,\dots,y_k\}$. 

For instance, Lemma \ref{uniqare} means that there are no multiple edges, and the more general Lemma \ref{fdet} implies that there are no multiple simplices (i.e., $C\mapsto s_C$ is injective). Theorem \ref{cubumed} (namely, (\ref{cu_ort})$\Rightarrow$(\ref{cu_iso})) implies that this complex is flag, in the sense that whenever it contains the 1-skeleton of a simplex, it contains the whole simplex. This condition is called the {\em combinatorial local CAT(0) condition}.  

On the other hand, Proposition \ref{cubusc} shows that this cubical complex (or equivalently its 1-skeleton) is simply connected. Combined with the previous condition, this is called the {\em combinatorial CAT(0) condition} and implies that if the $k$-cubes are endowed with the Euclidean metric from $[0,1]^k$ and the complex with the resulting length distance, this is a geodesic CAT(0) metric space. We refer to \cite{BH} for a thorough discussion of polyhedral complexes.
\end{rem}

\begin{rem}\label{cat0med}
Conversely, let $V$ be a connected graph endowed with a family of finite full subgraphs, each of which being isomorphic to some cube, defining a cubical complex. Assume that the link at each vertex satisfies the combinatorial local CAT(0) condition (this means that the link has no double simplex and is flag), and that the resulting complex is simply connected. Then $V$ is a median graph. See \cite{Che}.
\end{rem}

\begin{prop}\label{cacu}
Let $V$ be a nonempty connected median graph. Equivalences:
\begin{enumerate}
\item\label{cacu_1} $V$ is isomorphic to some hypercube;
\item\label{cacu_2} for every vertex $x$, any two distinct edges $(x,y)$ and $(x,y')$ are orthogonal;
\item\label{cacu_2bis} there is no nontrivial inclusion between strict biconvex subsets (or equivalently, any two nonempty biconvex subsets have a nontrivial intersection); 
\item\label{cacu_3} for every finite family $(B_i)$ of pairwise non-opposite nonempty biconvex subsets, we have $\bigcap B_i\neq\emptyset$;
\item\label{cacu_4} every median orientation on $V$ is Roller.
\end{enumerate}
If $V$ is finite, these conditions are also equivalent to: 
\begin{enumerate}
\addtocounter{enumi}{5}
\item\label{cacu_5} $\Aut(V)$ acts transitively on $\CB_V\smallsetminus\{\emptyset,V\}$  .
\end{enumerate}
\end{prop}
\begin{proof}

Suppose (\ref{cacu_2bis}) or (\ref{cacu_3}). If $(x,y)$ and $(x,y')$ are distinct directed edges, then we deduce that $B_{y,x}\cap B_{y',x}$ is nonempty, proving that $(x,y)$ and $(x,y')$ are orthogonal (e.g., by Theorem \ref{cubumed}), so (\ref{cacu_2}) holds. 

That (\ref{cacu_3}) implies (\ref{cacu_4}) is clear, and the converse holds because given $(B_i)$ as in (\ref{cacu_3}), there exists some median orientation for which $B_i=B_i^+$ for all $i$.

That (\ref{cacu_1}) implies each of (\ref{cacu_2}), (\ref{cacu_2bis}) and (\ref{cacu_3}) is clear.

That (\ref{cacu_4}) implies (\ref{cacu_2bis}) is easy by contraposition: if $B,B'$ are strict biconvex subsets and disjoint, we can prescribe them to be positive for some median orientation, which by definition is not Roller.

Assume (\ref{cacu_2}) and let us show (\ref{cacu_1}). Write $\CB^*(x_0)=\CB(x_0)\smallsetminus\{V\}$.
Fix a vertex $x_0$. Define $\Phi:V\to 2^{(\CB^*(x_0))}$ by $\Phi(x)=\CB(x_0)\smallsetminus\CB(x)$. By Corollary \ref{cormed2}, this is an isometric embedding. Let us show that it is surjective; denote by $V'$ its image, which satisfies (\ref{cacu_2}).

First, let $S$ be the set of $B\in\CB^*(x_0)$ such that $\{B\}\notin V'$, and assume by contradiction that $S\neq\emptyset$. Observe that $\bigcup_{x\in V}\Phi(x)=\CB^*(x_0)$, because if $B\in\CB^*(x_0)$, then we can choose $x\notin B$, and hence $B\notin\CB(x)$, which means that $B\in\Phi(x)$. Therefore, let $M\subset\CB^*(x_0)$ be an element of $V'$ of minimal cardinal such that $M$ contains some $s\in S$. Then $M$ is not a singleton, since otherwise we would have $\{s\}\in V'$, contradicting the definition of $S$. Hence, using the first three points in a geodesic segment joining $M$ to $\emptyset$, we can write $M=N\sqcup\{y\}$ and find $x\in N$  such that and $N\smallsetminus\{x\}$ and $N$ belong to $V'$. The minimality of $M$ then implies that $s\notin N$, so that $y=s$. Then the directed edges $(N,N\sqcup\{s\})$ and $(N,N\smallsetminus\{x\})$ being orthogonal in $V'$, the fourth vertex in the square they generate is $M\smallsetminus\{s\}$ and belongs to $V'$; this contradicts the minimality of $M$. Therefore $S=\emptyset$.


Now let us show that $V'=V$.
By contradiction, let $M$ be a finite subset of $\CB(x_0)$, of minimal cardinality, such that $M$ does not belong to $V'$. Then $M\neq\emptyset$. Since $S=\emptyset$, we know that $M$ is not a singleton. So we can write $M=N\sqcup\{x,y\}$ with $x,y\notin N$, $x\neq y$. Then the directed edges $(N,N\sqcup\{x\})$ and $(N,N\sqcup\{y\})$ are orthogonal in $V'$, which by (\ref{cacu_2}) implies that $M$ belongs to $V'$, a contradiction.

It is clear that (\ref{cacu_1}) implies (\ref{cacu_5}) (here in the case of the 0-cube we agree to call the action on the empty set transitive). Conversely, (\ref{cacu_5}) implies that all strict biconvex subsets have the same cardinal; in particular if $V$ is finite, this implies there are no nontrivial inclusion between strict biconvex subsets, showing (\ref{cacu_2bis}).
\end{proof}

\subsection{Gerasimov's theorem}

The following theorem was proved by Gerasimov \cite{Ger} for finitely generated groups. The following generalization, which relaxes the finite generation assumption, uses in a more or less hidden way several of his arguments, although the final layout of the proof is substantially simplified.

\begin{thm}\label{gthm}
Let a group $G$ act isometrically on a connected median graph $V$ with a bounded orbit. Then it has a finite orbit of vertices.
\end{thm}

We need the following general lemma, extracted from Gerasimov's paper:

\begin{lem}\label{1gh}
Let a group $G$ act transitively on an a infinite set $X$. Then for every finite subset $F$ of $G$ there exist $g,g'\in G$ such that $F$, $gF$ and $g'F$ are pairwise disjoint.
\end{lem}
\begin{proof}
Let us first find $g\in G$ with $F\cap gF=\emptyset$. The set $P$ of $g$ such that $F\cap gF\neq\emptyset$ is precisely $\bigcup_{x,y\in F}\{g:gx=y\}$. Fix $x_0\in X$, let $H$ be its stabilizer and fix a finite set $K\subset G$ such that $F=Kx_0$. Then 
\[P=\bigcup_{h,k\in K}\{g:ghx_0=kx_0\}=\bigcup_{h,k\in K}\{g:k^{-1}gh\in H\}=\bigcup_{h,k\in K}(kg^{-1})gHg^{-1}.\]
This is a finite union of left cosets of subgroups of infinite index; by B.H. Neumann \cite{Ne54}, it follows that $P\neq G$. So taking $g\notin P$ we have $F\cap gF=\emptyset$.

Now let us prove the lemma. By the previous case, find $g$ such that $F\cap gF=\emptyset$. Then apply the previous case again to $F\cup gF$: there exists $g'$ such that $F\cup gF$ and $g'F\cup g'gF$ are disjoint. In particular, $F$, $gF$ and $g'F$ are pairwise disjoint.
\end{proof}

\begin{proof}[Proof of Theorem \ref{gthm}]


The group $G$ naturally acts on $\CB_V$, commensurating $\CB(x)$ for each vertex $x$ (by Proposition \ref{numberwap}).


Define $\mathcal{E}(x)$ as the set of $G$-orbits $Z\subset\CB_V$ such that $Z\cap \CB(x)\notin\{\emptyset,Z\}$. Since $\ell_{\CB(x)}$ is bounded, $\CB(x)$ is transfixed by Theorem \ref{btx}. This implies that $\mathcal{E}(x)$ is finite. Define $\mathcal{E}_\infty(x)\subset \mathcal{E}(x)$ as the set of $Z\in\mathcal{E}(x)$ such that $Z$ is infinite, and let $n(x)$ be the cardinal of $\mathcal{E}_\infty(x)$; we have $n(x)\le\#(\mathcal{E}(x))<\infty$.

For $x\in V$, if $n(x)=0$, then the orbit of $\CB(x)$ is finite, since it ranges over subsets $Y$ of $\CB_V$ such that $Y\tu \CB(x)\subset\bigcup_{Z\in\mathcal{E}(x)}Z$. Thus, by injectivity of $x\mapsto \CB(x)$ (Corollary \ref{cormed}) the orbit of $x$ in $V$ is finite, as required.

So to prove the theorem, it is enough to show that whenever $x\in V$ and $n(x)\ge 1$, there exists $x'\in V$ such that $n(x')<n(x)$. Indeed, let $Z\in\mathcal{E}(x)$ be an infinite orbit. By Theorem \ref{btx}, since $Z\cap \CB(x)$ is transfixed, it is either finite or has finite complement in $Z$. In other words, there exists a finite subset $N$ of $Z$ such that $Z\cap \CB(x)$ is either equal to $N$ or $Z\smallsetminus N$. 
By Lemma \ref{1gh}, there exist $g,g'\in G$ such that $N$, $gN$ and $g'N$ are pairwise disjoint. Therefore $m(N,gN,g'N)$ is either equal to $\emptyset$ or $Z$. It follows, defining $x'=m(x,gx,g'x)$, that $m(\CB(x),g\CB(x),g'\CB(x))=\CB(x')$. We see that $\mathcal{E}(x')\subset\mathcal{E}(x)\smallsetminus\{Z\}$. We deduce that $n(x')<n(x)$.  
\end{proof}

\begin{cor}\label{corge}Let $G$ be a topological group with Property FW and let $X$ be a continuous discrete $G$-set. Then $H^1(G,\Z X)=0$, where $\Z X$ is the set of finitely supported functions $X\to\Z$.\end{cor}
\begin{proof}
Any continuous 1-cocycle of $G$ in $\Z X$ defines a continuous affine action of $G$ on the discrete abelian group $\Z X$, which lifts the projection $\Z X\rtimes G\to G$. Note that the action of $\Z X\rtimes G$ on $\Z X$ is generated by the linear action of $G$ and the action by translations. In particular, it preserves the natural partial ordering of $\Z X$. By Property FW and Gerasimov's theorem (Theorem \ref{gthm}), it preserves a finite orbit $\{f_1,\dots,f_n\}$. Then $\min(f_1,\dots,f_n)$ is $G$-invariant and has finite support, hence is a fixed point. This means that the cocycle is a 1-coboundary. Hence $H^1(G,\Z X)=0$.\end{proof}

Let us now give a corollary making use of the non-positively curved cubulation result.

\begin{cor}\label{fixmedian}
Let a group $G$ act isometrically on a connected median graph $V$ with a bounded orbit. Then there is an invariant cube. If moreover the action on $V$ preserving a median orientation, then there is a fixed vertex.
\end{cor}
\begin{proof}
By Theorem \ref{gthm}, there is a finite orbit; by Theorem \ref{hull}(\ref{ic4}), its total convex hull is finite; hence we can suppose that $V$ is finite. Therefore, as observed in Remark \ref{mediancat0}, its canonical cubulation is CAT(0); moreover it is a complete metric space (it is indeed compact). Hence $G$ has a fixed point \cite[Corollary II.2.8]{BH} in the canonical cubulation. This fixed point belongs to the interior of some cube, and therefore this cube $C$ is invariant. 

In case $G$ preserves some median orientation, there exists a unique vertex $v$ in $C$ such that all directed edges $(v,w)$ with $w\in C$ are oriented. Therefore $v$ is invariant by $G$.
\end{proof}

\begin{rem}
It is likely that the corollary can be proved avoiding the non-positively curved cubulation result: the missing step is to find a combinatorial proof that for any non-empty connected finite median graph $V$, there is an $\Aut(V)$-invariant cube in $V$, or still equivalently, show that if $V$ is a non-empty connected finite median graph $V$ and is not isomorphic to a cube (see Proposition \ref{cacu} for various characterizations), then there is a proper nonempty $\Aut(V)$-invariant median subgraph.
\end{rem}

\subsection{Involutive commensurating actions and the Sageev graph}\label{icm}

\subsubsection{Involutions}

\begin{defn} An {\em involutive preposet} $(E,\le,\sigma)$ is a partially preordered set endowed with a order-reversing fixed-point-free involution $\sigma$. 



If $(E,\le,\sigma)$ involutive preposet, define an {\em ultraselection}\footnote{Ultraselections are sometimes called ultrafilters, but this terminology is incoherent and confusing, partly because there is no natural way to characterize ultrafilters as ultraselections. 
Indeed ultrafilters require an additional (and essential!) condition on intersections, which is not reflected here, for instance when $E$ is a power set endowed with inclusion and complementation.} on $E$ as a subset $S\subset E$ satisfying
\begin{itemize}
\item $S$ is a {\em selection}, i.e., namely $x<y$ and $x\in S$ implies $y\in S$ ($x<y$ means $x\le y$ and $x\ge\!\!\!\!\!\!/\; y$).
\item $\sigma(S)=S^c$ (i.e., we have a partition $E=S\sqcup \sigma(S)$).
\end{itemize}
\end{defn}

An equivalent data is that of a function $j:E\to \{0,1\}$ such that $j(x)+j(\sigma(x))=1$ for all $x\in E$, and $j$ is non-decreasing, namely $x< y$ and $j(x)=1$ implies $j(y)=1$ . Given $j$, we get $S$ by $S=j^{-1}(\{1\})\subset E$, and given $S$, we obtain $j=\mathbf{1}_S$.
Note that the set of ultraselections is obviously compact under the pointwise convergence topology.

Say that two ultraselections $S,T$ are {\em incident} if $\#(S\tu T)=2$. In this case, there exists $z\in E$ such that $S\tu T=\{z,\sigma(z)\}$; moreover $z$ is a minimal element of $S$ (in the sense that there is no $x\in S$ such that $x<z$, or equivalently $y\in S$ and $y\le z$ imply $y\ge z$). Conversely, if $S$ is an ultraselection and $z$ is a minimal element of $S$, then $S\cup\{\sigma(z)\}\smallsetminus\{z\}$ is an ultraselection, incident to $S$ by definition. This incidence relation defines a graph structure (non-oriented, with no self-loop and with no multiple edges), denoted by $\textnormal{Sel}(E,\le,\sigma)$. The next two lemmas are straightforward generalizations of Nica's Lemma 4.3 and Proposition 4.5 in \cite{Nic}; on the other hand, the idea of associating a cubing to an abstract poset appeared in Niblo-Reeves in \cite{NRe} (with a few restrictions).

\begin{lem}\label{zcompo}Two ultraselections $S,T$ on the involutive preposet $E$ are in the same connected component of the graph $\Sel(E,\le,\sigma)$ if and only if $S\tu T$ is finite. Moreover, the inclusion of connected components of $\Sel(E,\le,\sigma)$ into $\Sel(E,=,\sigma)$ is isometric.
\end{lem}
\begin{proof}
In the first assertion, the condition is clearly necessary. Let us prove it is sufficient by induction on $2n=\#(S\tu T)$ (which is even because $S\smallsetminus T=\sigma(T\smallsetminus S)$). The case $n=0$ is clear. Otherwise, find an minimal element $z$ in $S\smallsetminus T$. If by contradiction there exists $x\in S$ with $x<z$, then by minimality of $z$, necessarily $x\in T$, but since $T$ is an ultraselection this forces $z\in T$, contradiction. Thus $S'=S\cup\{\sigma(z)\}\smallsetminus\{z\}$ is an ultraselection incident to $j$, and $\#(S'\tu T)=2n-2$, so $S'$ is in the same component as $T$ by induction.

The above proof actually proves the isometric statement as well.
\end{proof}

We have the following lemma (see \S\ref{mger} for the definition of median graph).

\begin{prop}\label{l_median}
The ultraselection graph of any involutive preposet is median.
\end{prop}
\begin{proof}
We begin with the case of the discrete preposet $\Sel(E,=,\sigma)$. Thus an ultraselection here is just a subset $S\subset E$ such that $E=S\sqcup \sigma(S)$. Let $A\subset E$ be a fundamental domain for $\sigma$ (in the sense that $X=A\sqcup\sigma(A)$). Then the function $2^A\to 2^E$ mapping $B$ to $B\sqcup \sigma(A\smallsetminus B)$ is a bijection from $2^A$ to the set of ultraselections of $(E,=,\sigma)$; if $2^A$ is endowed with its median graph structure from Example \ref{commed}, this is a graph isomorphism, so $\Sel(E,=,\sigma)$ is median.

Now let the preposet be arbitrary. By Lemma \ref{zcompo}, the embedding of the graph $\textnormal{Sel}(E,\le,\sigma)$ into the graph $\textnormal{Sel}(E,=,\sigma)$ is isometric (if we allow infinite distances). This shows that total intervals $[x,y]$ in the first graph are contained in total intervals $[x,y]'$ in the second one. In particular, it immediately follows that for all $S_1,S_2,S_3\in \textnormal{Sel}(E,\le,\sigma)$, the intersection $[S_1,S_2]\cap [S_2,S_3]\cap [S_3,S_1]$ contains at most one point; we only have to check that the median point $T=(S_1\cap S_2)\cup (S_2\cap S_3)\cup (S_3\cap S_1)$ belongs to $\textnormal{Sel}(E,\le,\sigma)$, i.e., is an ultraselection. Clearly the set of selections is stable under taking finite unions and intersections (unlike the set of ultraselections!), and therefore $T$ is a selection. Moreover, keeping in mind that $\sigma(S_i)=S_i^c$, we see that $\sigma(T)=T^c$ and thus $T$ is an ultraselection and the proof is complete.
\end{proof}

Let now $G$ be a topological group.
Let $X$ be a continuous discrete $G$-set and $A\subset X$ a commensurated subset with an open stabilizer; let $\sigma$ be a $G$-equivariant involution of $X$ such that $X=A\sqcup\sigma(A)$.
Define, as in Proposition \ref{can_bij}, the corresponding walling $W_x=\{h\in G\mid x\in hA\}$. 
Endow $X$ with the partial order given by $x\le y$ if $W_x\le W_y$. The above graph structure on $X$ is clearly $G$-invariant. On the other hand, $A$ itself is an ultraselection, as well as any translate $gA$. This is clear by observing that $gA=\{x\in X\mid g\in W_x\}$; we thus call translates of $A$ {\em principal ultraselections}. 

A first observation is that all principal ultraselections belong to the same connected component of the graph, by Lemma \ref{zcompo}. 

\begin{defn}
The {\em Sageev graph} associated to $(G,X,A,\sigma)$ is the component of $\textnormal{Sel}(X,\le,\sigma)$ containing principal ultraselections.
\end{defn}

The Sageev graph is connected by definition and is median by Proposition \ref{l_median}; the action of $G$ is continuous. It contains an isometric copy of the set of translates of $A$, with the symmetric difference metric. 


\begin{rem}Actually, Sageev \cite{S95} directly cubulated the set of ultraselections and the link with median graphs was brought out later by Chepoi \cite{Che}. The point of view given here is closer to that of Nica \cite{Nic}.
\end{rem}

Define $\Comm_A^\sigma(X)$ as the set of subsets $M$ of $X$ commensurate to $A$ and such that $M^c=\sigma(M)$; it is a connected component of $\Sel(X,=,\sigma)$. The Sageev graph can be viewed as an isometrically embedded subgraph of $\Comm^\sigma_A(X)$, by Lemma \ref{zcompo}. In general, finding a cubulation associated to a commensurating action amounts to finding ``small" invariant connected median subgraphs of $\Comm_A^\sigma(X)$. Here small can, among others, mean that the graph is locally finite, or that the underlying cubing is finite-dimensional.

A construction, assuming in addition that $G$ as well as the stabilizers in $G$ of points in $X$ are finitely generated, is done by Niblo, Sageev, Scott and Swarup \cite{NSSS}. It consists in replacing the partial preorder on $X$ by a larger one, requiring that $x\preceq y$ if $W_x\smallsetminus W_y$ is ``small" in a suitable sense. This can be done consistently under the assumption that the walling $(W_x)$ is in ``good" position; they actually show that this can always be supposed at the cost of replacing $A$ by a commensurate subset (and using crucially that the acting group and stabilizers are finitely generated).

\subsection{Link with other properties}\label{liin}
Let us check the equivalences of the introduction for Property FW and PW. The following proposition summarizes several of the previous results.

\begin{prop}\label{fc}
Let $G$ be a topological group. Consider the following properties, for functions $f:G\to\R$:
\begin{enumerate}[(I)]
\item\label{fc1} $f$ is cardinal definite;
\item\label{fc4} there exists a topological space $Y$ with a continuous $G$-action (i.e., the map $G\times Y\to Y$ is continuous), a continuous discrete $G$-set $X$, a $G$-walling $\mathcal{W}=(W_x)_{x\in X}$ of $Y$ by clopen subsets, and $y_0\in Y$ such that $f(g)=d_{\mathcal{W}}(y_0,gy_0)$ for all $g$.
\item\label{fc3} there exists a continuous action of $G$ on a connected median graph and a vertex $x_0$ such that $f(g)=d(x_0,gx_0)$ for all $g$;
\item\label{fc2} there exists a continuous action of $G$ on a CAT(0) cube complex and a vertex $x_0$ such that $f(g)=d(x_0,gx_0)$ for all $g$;
\item\label{fc5} there exists an isometric action of $G$ on an ``integral Hilbert space" $\ell^2(X,\Z)$ ($X$ any discrete set), or equivalently on $\ell^2(X,\R)$ preserving integral points and $v_0\in\ell^2(X,\Z)$ such that $f(g)=\|v_0-gv_0\|_2^2$ for all $g\in G$. 
\end{enumerate}
Then (\ref{fc1})$\Leftrightarrow$(\ref{fc4}), (\ref{fc3})$\Leftrightarrow$(\ref{fc2}); (\ref{fc1})$\Rightarrow$(\ref{fc5})$\Rightarrow$(\ref{fc3}), and (\ref{fc3}) for $f$ implies (\ref{fc1}) for $2f$. Thus, roughly speaking, all properties are equivalent up to multiplication by 2.
\end{prop}
\begin{proof}
The implication (\ref{fc1})$\Rightarrow$(\ref{fc4}) follows from Proposition \ref{can_bijt} by taking $Y=G$. Let us show the converse: assume (\ref{fc4}) and consider a $G$-walling $\mathcal{W}$ on $Y$ and $y_0\in Y$ as in (\ref{fc4}). Fix $y_0\in Y$ and pull the walling back to a walling $\mathcal{W}'$ on $G$ by the orbital map $g\mapsto gy_0$. Then for all $g\in G$ we have $d_{\mathcal{W}'}(1,g)=d_{\mathcal{W}}(y_0,gy_0)$. By Proposition \ref{can_bijt}, this is a cardinal definite function.

The equivalence (\ref{fc3})$\Leftrightarrow$(\ref{fc2}) follows from Chepoi's result (Remarks \ref{mediancat0} and \ref{cat0med}) that median graphs are precisely the 1-skeleta of CAT(0) cube complexes. 

The implication (\ref{fc1})$\Rightarrow$(\ref{fc5}) is provided by the Niblo-Roller construction, as formulated in Proposition \ref{sxlp2}. 

Let us now consider (\ref{fc5}). First note that since the closed affine subspace generated by $\ell^2(X,\Z)$ is $\ell^2(X,\R)$, by the GNS-construction (see \cite[Appendix C]{BHV}) isometries of $\ell^2(X,\Z)$ have a unique extension to $\ell^2(X,\R)$, which is affine by the Mazur-Ulam theorem. In particular, both versions of (\ref{fc5}) are equivalent.

We have (\ref{fc5})$\Rightarrow$(\ref{fc3}), by observing that the obvious graph structure on $\ell^2(X,\Z)$, obtained by joining any two points at distance 1, is connected median (it also implies (\ref{fc2}) directly).

Finally, if $f$ satisfies (\ref{fc3}), then $2f$ satisfies (\ref{fc1}) by Corollary \ref{medca2}.
\end{proof}



\begin{lem}\label{vanh}
Let $G$ be a topological group and $X$ a continuous discrete $G$-set. Let $A$ be an arbitrary nonzero discrete abelian group. Suppose that $H^1(G,A^{(X)})=0$. Then every $G$-commensurated subset of $X$ with an open stabilizer is transfixed.
\end{lem}
\begin{proof}
Let $M$ a commensurated subset with an open stabilizer; fix $a\in A\smallsetminus\{0\}$. For $N\subset X$, and $x\in X$, define $\mathbf{1}^a_N(x)$ as equal to $a$ if $x\in N$ and 0 otherwise.

Then $g\mapsto \mathbf{1}^a_M-\mathbf{1}^a_{gM}$ is a continuous 1-cocycle of $G$ in $A^{(X)}$ and hence by the vanishing of $H^1(G,A^{(X)})$ is a 1-coboundary, i.e.\ has the form $f-gf$ for some $f\in A^{(X)}$. Hence the function $h=\mathbf{1}^a_M-f$ is $G$-invariant. So $N=\{x\in X:\;h(x)=a\}$ is $G$-invariant. Write $h=1^a_N+f'$, where $f'$ has finite support. Then $1_M^a-1_N^a=f+f'$, which has finite support, so $M$ and $N$ are commensurate and thus $M$ is transfixed.
\end{proof}

\begin{prop}\label{fwint}Let $G$ be a topological group. Equivalences:
\begin{enumerate}[(i)]
\item\label{fwdef2} $G$ has Property FW;
\item\label{fwcb2} every cardinal definite function on $G$ is bounded;
\item\label{fwca2} every continuous cellular action on any CAT(0) cube complex with the $\ell^2$-metric has bounded orbits (we allow infinite-dimensional cube complexes);
\item\label{fwcf2} every continuous cellular action on any nonempty CAT(0) cube complex has a fixed point;
\item\label{fwmeb2} every continuous action on a connected median graph has bounded orbits;
\item\label{fwme2} every action on a nonempty connected median graph has a finite orbit;
\item\label{fwcod2} (if $G$ is compactly generated and endowed with a compact generating subset) for every open subgroup $H\subset G$, the Schreier graph $G/H$ has at most 1~end;
\item\label{fww2} for every topological space $Y$ with a continuous $G$-action (i.e., the map $G\times Y\to Y$ is continuous), every continuous discrete $G$-set $X$ and every $G$-walling $\mathcal{W}=(W_x)_{x\in X}$ of $Y$ by clopen subsets, the $G$-orbits in $(Y,d_\mathcal{W})$ are bounded.
\item\label{zhilb2} every isometric action on an ``integral Hilbert space" $\ell^2(X,\Z)$ ($X$ any discrete set), or equivalently on $\ell^2(X,\R)$ preserving integral points, has bounded orbits. 
\item\label{zx2} for every continuous discrete $G$-set $X$ we have $H^1(G,\Z X)=0$.
\end{enumerate}
\end{prop}
\begin{proof}
Proposition \ref{fc} immediately entails the equivalences (\ref{fwcb2}) $\Leftrightarrow$ (\ref{fwca2}) $\Leftrightarrow$ (\ref{fwmeb2}) $\Leftrightarrow$ (\ref{fww2}) $\Leftrightarrow$ (\ref{zhilb2}).

One direction in the equivalence (\ref{fwdef})$\Leftrightarrow$(\ref{fwcb}) is trivial and the other follows from Theorem \ref{btx}.

The equivalence (\ref{fwcb2})$\Leftrightarrow$(\ref{fwcod2}) follows from Proposition \ref{ends}.

One direction in the equivalence (\ref{fwmeb2})$\Leftrightarrow$(\ref{fwme2}) is trivial and the other follows from Theorem \ref{gthm}.

(\ref{fwcf2}) $\Rightarrow$ (\ref{fwca2}) is trivial and the converse follows from Corollary \ref{fixmedian}.

Let us finally prove (\ref{fwdef2})$\Leftrightarrow$(\ref{zx2}). The implication $\Leftarrow$ follows from Lemma \ref{vanh}, and conversely the implication $\Rightarrow$ is a corollary of Gerasimov's theorem (Corollary \ref{corge})
\end{proof}


\begin{rem}\label{h1}
In Proposition \ref{fwint}(\ref{zx2}), the ring $\Z$ cannot be replaced by an arbitrary nonzero unital ring. While the proof carries over to show that if $R$ is an arbitrary unital nonzero ring, $X$ is a discrete continuous $G$-set with a commensurated subset $M$ with an open stabilizer, $H^1(G,RX)=0$ implies that $M$ is transfixed, the converse does not always hold. More precisely, consider $n\ge 2$, the group $G=\Z/n\Z$ and the $G$-set $X=\{x\}$ reduced to a singleton. Since $G$ is finite, it has property FW. On the other hand, it is straightforward that  for $R=\Z/n\Z$ we have $H^1(G,RX)$ nonzero (since it consists of the group of group endomorphisms of $\Z/n\Z$). 

The 1-cohomology of $G$ in $\Z X$ has been studied in \cite{Ho78}.
\end{rem}

Note that Proposition \ref{fc} also immediately implies the equivalence between Property PW and its the various reformulations given in the introduction.

\begin{rem}\label{r_wa}
For a topological group $G$, Property FW is also equivalent to the property of having bounded orbits on space with walls (in its most popularized sense, not allowing multiplicities of walls). Obviously it implies this property; to see the converse, assume that $G$ does not have Property FW. Then it has an unbounded cardinal definite function $f$. According to an orbit decomposition, write $f=\sum_{i\in I} n_if_i$, where $(n_i)$ is a family of positive integers and $(f_i)$ is an injective family of nonzero cardinal definite functions, such that no $f_i$ is a proper multiple, i.e.\ has the form $kf''$ with $k\ge 2$ integer and $f''$ cardinal definite (otherwise we decompose again).

Define $f'=\sum_{i\in I}f_i$. Then $f'\le f$ and $f'$ is cardinal definite. We claim that $f'$ is unbounded as well. If $I$ is finite, then $f\le nf'$ where $n=\max_in_i$, so this is clear. If $I$ is infinite, we evoke the contraposition of Corollary \ref{bdf} to infer that $f'$ is unbounded. 

Now write $f_i=\ell_{M_i}$, where $M_i$ is a commensurated subset in a set $X_i$ with open stabilizer, if $(W^{(i)}_x)_{x\in X_i}$ is the associated walling on $G$ given by Proposition \ref{can_bij}. Since $f_i$ is a not proper multiple, the walling $x\mapsto W_x^{(i)}$ is injective. Also since the $f_i$ are distinct, we have $W_x^{(i)}\neq W_y^{(j)}$ for any $i\neq j$, $x\in X_i$ and $y\in X_j$. So the structure of space with walls on $G$ for which the walls are the $W_x^{(i)}$ for $i\in I$ and $x\in X_i$ (and their complement if necessary) gives rise to an unbounded distance on $G$.
\end{rem}

\section{About this paper}\label{abt}

Most of the paper consists of a ``digest" of known results but also contain some new ones; the purpose of this final paragraph is to clarify this. Maybe the main goal of this paper is to start at the beginning (that is, Section \ref{comto}!) with the very elementary notion of commensurating action, rather than the more elaborate notion of CAT(0) cube complex or even median graphs. The introduction of the point of view of commensurating actions in the measurable context in relation to Property~T is due to Robertson and Steger \cite{RS}. Let us now point out the new results of the paper.

\begin{itemize}
\item Maybe the most apparent contribution is the use of topological groups rather than discrete ones. However, this change is mainly secondary and was certainly not the main motivation for writing this paper; it nevertheless provides a more coherent setting, for instance for the study of Property FW for irreducible lattices \cite{Cor2}. It involves some technical issues such as continuity, which are notably addressed in \S\ref{auco}. Also, part of this generalization is to remove finite generation assumptions, even in the case of discrete groups.

\item The notion of $G$-walling is very similar to many definitions appearing in various places. It was cooked up so that the (almost tautological) Proposition  \ref{can_bij} holds. Although the correspondence between actions on (various forms of) spaces with walls and commensurating actions (sometimes dressed up as almost invariant subsets) is known in both directions, such a simple statement as Proposition  \ref{can_bij} was not previously extracted.

\item The finiteness result of \S\ref{s_cof} and some of its corollaries (including Proposition \ref{fwfwp}(\ref{fwww})) seem to be new, even when specified to finitely generated discrete groups; they can be seen as a combinatorial interpretation of the folklore Corollary \ref{fmo}. By Chepoi's correspondence, a version of Corollary \ref{fmo} holds for CAT(0)-cubings. Corollary \ref{fmo} was obtained for proper actions of finitely generated groups on locally compact finite-dimensional cube complexes \cite[Proposition 3.12]{CS}, but Pierre-Emma\-nuel Caprace and Fr\'ed\'eric Haglund independently mentioned to me that the argument can be adapted to provide the general statement of Corollary \ref{fmo} (when $G$ is finitely generated, without extra-assumptions).

\item The notion of Property FW was briefly addressed in Sageev's question ``which classes of finitely generated groups admit a coforked subgroup?" \cite[Q1]{S97}; the underlying property would rather be FW', which is less convenient to deal with for infinitely generated groups. Anyway none of these properties was subsequently studied for its own right and the propositions following the definition (from Proposition \ref{fwfwp} to Proposition \ref{extfw}) can be termed as new. For instance the stability by extensions is trickier than we could expect.

\item Most equivalences in Proposition \ref{fwint} (i.e., those of the introduction, restated for topological groups) are classical or essentially classical, but the characterizations (\ref{zhilb2}) and (\ref{zx2}) are new (note that Proposition \ref{fwint}(\ref{zx2}) is not as obvious as it may seem at first sight, see Remark \ref{h1}).

\item A part of Section \ref{s_ab} is a direct proof of Haglund's result on non-distortion of cyclic subgroups. The application of Example \ref{sl2z2} is new. Proposition \ref{cez} is entirely new as well as its application Example \ref{sl2ti}. I am not comfortable to decide whether to call Proposition \ref{polyfw} a new result or an immediate application of Houghton's Theorem \ref{houghtonpo}; however, had this this observation been made earlier, the question of finding a group with the Haagerup Property with no proper action on a CAT(0) cube complex (or space with walls) would not have been considered as open between the late nineties and Haglund's paper \cite{Hag}. Section \ref{s_ab} also emphasizes new properties (uniform non-distortion, $\VD(\Gamma)$) which hopefully could prove relevant in other contexts.
\end{itemize}

\end{document}